\documentclass[english,11pt]{article}
\usepackage[german,french,english]{babel}
\usepackage[cp850]{inputenc}
\usepackage{latexsym,graphicx, fancybox}
\usepackage{graphicx}
\usepackage{caption}

\usepackage{amssymb, amsmath, amsfonts}
\usepackage{amsmath, amssymb}
\usepackage{color}
\usepackage{tikz}
\usepackage{latexsym}
\usepackage{tocloft}        

\usepackage{mathrsfs}

\usepackage{natbib}

\usepackage{etoolbox}
\apptocmd{\thebibliography}{\setlength{\itemsep}{0.01pt}}{}{}

\usepackage{microtype}
	\usepackage{amsthm}
	
	\newcommand{\dd}{\,\mathrm{d}}
	\newcommand{\R}{\mathbb{R}}
	\newcommand{\N}{\mathbb{N}}
	\newcommand{\E}{\mathbb{E}}
	\renewcommand{\P}{\mathbb{P}}
	
	\newcommand{\Hcal}{{\mathcal H}}
	\newcommand{\fdot}{{\,\cdot\,}}
	\newcommand{\Tcal}{{\mathcal T}}
	\usepackage{authblk}

	\newtheorem{Theorem}{Theorem}[section]
	\newtheorem{Definition}[Theorem]{Definition}
	\newtheorem{Proposition}[Theorem]{Proposition}
	\newtheorem{Assumption}[Theorem]{Assumption}
	
	\newtheorem{Lemma}[Theorem]{Lemma}

	\newtheorem{Example}[Theorem]{Example}
	
	\newtheorem{assumption}{Assumption}[section]

	\usepackage{thm-restate}
	\usepackage{float}
	\usepackage{hyperref}
	
\begin{document}
	
\selectlanguage{english}
\title{\bf On (fake) Stationarity in Stochastic Volterra 
	Equations  with Affine Drift and Regular Kernels }
\author{
	Emmanuel Gnabeyeu\\
	Laboratoire de Probabilit\'es, Statistique et Mod\'elisation (LPSM), UMR~8001,\\
	\textit{Sorbonne Universit\'e and Universit\'e Paris Cit\'e}, Paris, France. \\
	E-mail: \texttt{emmanuel.gnabeyeu\_mbiada@sorbonne-universite.fr}
	\and
	Gilles Pag\`es\\
	Laboratoire de Probabilit\'es, Statistique et Mod\'elisation (LPSM), UMR~8001,\\
	Sorbonne Universit\'e, case 158, 4, pl. Jussieu, F-75252 Paris Cedex 5, France. \\
	E-mail: \texttt{gilles.pages@sorbonne-universite.fr}
}
\date{August 30, 2026}
  \maketitle
  \renewcommand{\abstractname}{Abstract}
 
\vspace{-.5cm}
\begin{abstract}	

We investigate the fake stationarity properties of solutions to forward Stochastic Volterra Integral Equations (SVIEs) with affine drift and long-memory (regular) kernels, both on finite horizons and in the long-run regime. 
By either deriving explicit closed-form specifications for the deterministic initial condition \(\phi\) and the mean-reversion function \(\mu\) appearing in the drift, or by introducing a deterministic stabilizing factor \(\varsigma\) in the diffusion coefficient associated with the kernel while keeping \(\mu\) fully flexible, we show that it is possible to induce a \emph{fake stationary} regime, in the sense that all marginal distributions share the same mean and variance. This extends the results of~\cite{EGnabeyeu2025} to the long-memory setting.
Afterwards, using a refined asymptotic analysis, we further establish that, in both frameworks, the time-shifted solutions of these long-memory SVIEs converge weakly, in the functional sense, toward a family of \(L^2\)-stationary processes sharing the same covariance structure, for suitable classes of diffusion coefficients. 
These results are applied to a class of exponential-fractional Stochastic Volterra Integral Equations driven by an \(\alpha\)-gamma fractional integration kernel, in the particular regime \(\alpha \geq 1\), which regularizes diffusion paths and invoke {\em long-term memory}, persistence or long range dependence.
\end{abstract}

\vspace{-.9cm}
\section{Introduction} 
The theory of 
stochastic Volterra integral equations (SVIEs) which has its origins in the 1980s,
have recently attracted considerable attention due to their ability to capture rough sample path behaviour (see, e.g.,~\cite{DiNunno2023} for a survey of the literature) as well as their flexibility in modelling both short- and long-range dependence structures~\cite{Samorodnitsky2016}.
By considering a deterministic continuous function \( \phi \), typically normalized such that \( \phi(0) = 1 \), in the absence of jumps, a general convolution-type stochastic Volterra equation with affine drift on \(\mathbb{R}\) can be written for any \(T>0\) as
\begin{equation}\label{eq:Volterrameanrevert}
	\begin{cases}
		X_t = X_0\phi(t) +\int_0^t K(t-s)(\mu(s)-\lambda X_s)ds + \int_0^t K(t-s)\sigma(s,X_{ s})dW_s, \quad X_0\perp\!\!\!\perp W.\\
		X_0 : (\Omega, \mathcal{F}, \mathbb{P}) \to (\mathbb{R}, {\cal B}or(\mathbb{R})) \text{ is a given initial random variable}
	\end{cases}
\end{equation}
where $\lambda>0$, $\mu :[0, T] \to \R$ in the drift is a Borel function and the diffusion coefficient $\sigma : [0, T] \times \mathbb{R} \to \mathbb{R}$ is a Lipstchiz continuous function
and $K$ a deterministic kernel modeling the memory or hereditary structure of the system.  The process $(W_t)_{t \geq 0}$ is an $\mathbb{R}$-valued standard Brownian motion independent of $X_0$, both defined on a probability space $(\Omega, \mathcal{F}, \P)$ and $\mathcal{F}_t \supset \mathcal{F}_{t}^{X_0, W}$ is a filtration satisfying the usual conditions.
Among the most prominent examples of  Volterra kernel is the fractional Riemann--Liouville kernel \(K_\alpha(t)
=
\frac{t^{\alpha-1}}{\Gamma(\alpha)},
\quad \alpha\in\left(\frac12,\frac32\right),\)
where \(\alpha := H-\frac12\), with \(H\) denoting the Hurst parameter. 
This kernel allows for a flexible incorporation of rough sample path properties when $H\in\big(0,\frac12\big)$ and some memory features when $H\in\big(\frac12, 1\big]$.
Another important example covered in this work is the exponentially damped fractional kernel \(	K_{\alpha,\rho}(t)
=
\frac{t^{\alpha-1}}{\Gamma(\alpha)}e^{-\rho t},
\; \rho>0.\)
These stochastic Volterra processes which have recently attracted much attention in the mathematical finance community following the empirical observation in \cite{GatheralJR2018} that volatility paths exhibit low H\"older regularity (\(H \approx 0.1\)), have
been originally introduced mostly with non-singular kernel for modelling population dynamics and more generally in biology
and physics \cite{mohammed1998}, in order to generalize modelling to non-Markovian random systems with memory effects and irregular behaviour. They were also motivated particularly by the physics of heat transfer \cite{gripenberg1990} and have undergone extensive mathematical study. 

In the late 1990s, attempts were made within the financial community to incorporate long-memory effects into continuous-time stochastic volatility models. This shift was largely motivated by the need to capture persistent dependencies observed in financial markets, particularly through fractional Brownian motion (see \cite{CoutinD2001, ComteR1998}). Earlier studies, such as those by Comte and Renault \cite{ComteR1998}, found that \(H > 1/2\) was a key parameter in capturing long memory in volatility dynamics.
Within this framework, Volterra equations~\eqref{eq:Volterrameanrevert} with fractional kernels offer a tractable alternative to SDEs driven by genuine \(H\)-fractional Brownian motion, while retaining the long-memory features associated with \(H > 1/2\).



The existence of stationary solutions and the study of long-time behaviour are central problems in stochastic analysis. Their extension to stochastic Volterra integral equations (SVIEs) with time-independent coefficients has attracted considerable attention in recent years. The non-Markovian and non-semimartingale nature of SVIEs makes the analysis of ergodicity, invariant measures, and stability substantially more challenging than in the classical SDE setting. Existing ergodicity results rely mainly on Markovian lifting methods based on the Laplace representation of completely monotone kernels (see, e.g.,~\cite{hamaguchi2024,Jacquieretal2022,huber2024}) or on the analysis of Volterra-type Riccati equations associated with completely monotone kernels (see~\cite{friesen2022volterra,BenAlayaFriesenKremer2025}) in the specific settings of affine Volterra processes. These techniques, however, do not extend naturally to regular kernels, where complete monotonicity may no longer hold (typically $K_{\alpha,\rho}$ when $\alpha>1$ and $\rho\geq0$).

In this work, we derive conditions ensuring that stochastic Volterra equations of the form~\eqref{eq:Volterrameanrevert}, driven by long-memory kernels (typically corresponding to \(\alpha\in\left(1,\frac32\right)\)), admit a \emph{fake stationary regime} (in the terminology introduced in~\cite{Pages2024}), in the sense that the solution to~\eqref{eq:Volterrameanrevert} either has time-invariant first and second moments (\emph{fake stationary regime of type I }) or the same marginal distribution at all times in the Gaussian case (\emph{ fake stationary regime of type II}, leading in particular to pseudo-Ornstein--Uhlenbeck type dynamics.)
This allows us to cover the full range of Hurst coefficient, namely \( H :=\alpha - \tfrac{1}{2} \in (0, 1) \), thereby complementing the results obtained in the rough setting in~\cite{Pages2024,EGnabeyeu2025}.
Specifically, our main result follow that of \cite{EGnabeyeu2025} and the approach proceeds along two complementary directions. In the autonomous diffusion setting of section~\ref{subsec:fakestatioIntrin}, we derive explicit closed-form characterizations of the deterministic initial condition \(\phi\) together with the mean-reversion function \(\mu\). Alternatively, by allowing the volatility coefficient to be separable in time and state in section~\ref{subsec:fakestatioStabil}, we preserve full flexibility in the specification of \(\mu\) while introducing a time-dependent (deterministic) multiplicative stabilizing factor in the Brownian convolution term, required to satisfy an appropriate functional convolution equation involving the derivative of the resolvent associated with the  Volterra kernel.
Moreover, we establish the existence of limiting distributions for these stabilized {\em long-memory} Volterra processes. Formally, we prove that as \( t \to \infty \), the shifted process \( (X^t_s)_{s \geq 0} \), defined by \( X^t_s := X_{s+t} \), functionally weakly converges to a limiting continuous process.  Unlike in~\cite{friesen2022volterra,EGnabeyeuPR2025,EGnabeyeuR2025}, the limiting process need not be stationary, in the sense of invariance of its finite-dimensional distributions under time shifts. However, we prove that, under \emph{fake stationary} regime, the limiting process is at least weak \( L^2 \)-stationary.

 \vspace{-.25cm}
\paragraph{Outline.}
The remainder of the paper is organized as follows: Section~\ref{sec:main_result} introduces the necessary preliminaries and provides a summary of the main results.  In Section~\ref{subsec:background}, we review key properties of stochastic Volterra equations with convolution kernels, including results on existence, moment control, and a special focus on processes with affine drift.
Sections~\ref{subsec:fakestatioIntrin} and~\ref{subsec:fakestatioStabil} are devoted to the development of the framework associated with the weak stationary regime of SVIEs~\eqref{eq:Volterrameanrevert}, in the spirit of~\cite{Pages2024, EGnabeyeu2025} with an example of a \emph{fake stationary regime} when the state-dependent diffusion coefficient is a trinomial function and the analysis of the long-run behavior of these time-inhomogeneous processes as time tends to infinity. Specifically, we establish that, for such stabilized processes, the functional weak asymptotics of the time-shifted process \( (X_{t+s})_{s \geq 0} \) as \( t \to +\infty \), turns out to be a weakly \( L^2 \)-stationary process.
Finally, in Section~\ref{sec:appl2}, we apply these results to the case of SVIEs with an \(\alpha\)-fractional integration kernel for \(\alpha \in \left(1, \frac{3}{2}\right)\) ({\em long-term memory, persistence or long range dependence}), where the case \(\alpha \in \left(\frac{1}{2}, 1\right)\) has been extensively studied in \cite[Section 5, Theorem 5.2]{Pages2024, EGnabeyeu2025}. In Section~\ref{sec:appl3}, we further extend the application to SVIEs with an \(\rho\)-exponential \(\alpha\)-fractional integration kernel for \(\alpha \in \left(1, \frac{3}{2}\right)\) involving both the \emph{long-term memory} effects inherent to those Volterra equations.

\vspace{-.3cm}
\section{Summary of standing assumptions and main results}\label{sec:main_result}
We summarize in this part the framework and the main results of our study.
We first introduce the notations, then the tools and basic framework in Section~\ref{subsec:main_tools}, and finally present a streamlined version of
our principal theorems in Section~\ref{subsec:main_result}.
\vspace{-.3cm}
\paragraph{Notations.}
\smallskip
\noindent $\bullet$ From now on, let \(T>0\) be fixed.  We denote $\mathbb{T} = [0, T] \subset \mathbb{R}_+$, ${\rm Leb}_d$ the Lebesgue measure on $(\R^d, {\cal B}or(\R^d))$, $d\geq1$.

\noindent $\bullet$ $\mathbb{X} := {\cal C}([0,T], \R^d)  (\text{resp.} \quad {\mathcal C_0}([0,T], \R^d))$ denotes the set of continuous functions (resp. null at 0)  from $[0,T]$ to $\R^d $ and ${\cal B}or(\mathbb{X})$ denotes the  Borel $\sigma$-field of $\mathbb{X}$ induced by the $\sup$-norm topology. 

\smallskip 
\noindent $\bullet$ For \(p\in[1,+\infty)\), \({\cal L}_{\mathrm{loc}}^p(\mathbb{R}_+,{\rm Leb}_1)\) denotes the space of measurable functions \(f:\mathbb{R}_+\to\mathbb{R}\) satisfying \(f\in L^p([0,T],{\rm Leb}_1),
\; \forall\,T>0,\) that is \(\int_0^T |f(t)|^p\,dt <+\infty,
\; \forall\, T>0.\)

\smallskip 
\noindent $\bullet$ For $p\in(0,+\infty)$, $L_{\mathbb H}^p(\P)$ or simply $L^p(\P)$ denote the set of  $\mathbb H$-valued random vectors $X$  defined on a probability space $(\Omega, {\cal A}, \P)$ such that $\|X\|_p:=(\E[\|X\|_{\mathbb H}^p])^{1/p}<+\infty$. For $f: E\to \R$, $\displaystyle \|f\|_{\sup}= \sup_{x\in E}|f(x)|$

\smallskip
\noindent$\bullet$ $\mathrm{Leb}_1$ (resp. $\mathrm{Leb}_{\mathbb{R}_+}$) denotes the Lebesgue measure on $\R$ (resp. $\mathbb{R}_+ = [0,+\infty)$).

\smallskip 
\noindent $\bullet$ For $f, g \!\in {\cal L}_{\R_+,loc}^1 (\R_+, {\rm Leb}_1)$, we define their convolution by $f*g(t) = \int_0^tf(t-s)g(s)ds$, $t\ge 0$.

\smallskip 
\noindent $\bullet$ For $f, g \!\in {\cal L}_{\R_+,loc}^2(\R_+, {\rm \text{Leb}_1})$ and $W$ a Brownian motion, we define  their stochastic convolution by 
\centerline{$
f\stackrel{W}{*}g = \int_0^t f(t-s)g(s) dW_s, \quad t\ge 0.
$}

\smallskip 
\noindent $\bullet$ For a random variable/vector/process $X$, we denote by $L(X)$ or $[X]$ its law or distribution. 

\smallskip 
\noindent $\bullet$ For short, for any \(p\in[1,+\infty]\), we will denote by ${\cal L}^p(\R_+)$ (resp. ${\cal L}_{\mathrm{loc}}^p(\R_+)$) the space ${\cal L}^p(\mathbb{R}_+,{\rm Leb}_1)$ (resp. ${\cal L}_{\mathrm{loc}}^p(\mathbb{R}_+,{\rm Leb}_1)$).

\smallskip 
\noindent $\bullet$ $X\perp \! \! \!\perp Y$  stands for independence of random variables, vectors or processes $X$ and $Y$.  

\smallskip 
\noindent $\bullet$ $\Gamma(a) = \int_0^{+\infty} u^{a-1} e^{-u} \, du, \; a > 0, \; 
\text{and} \quad 
B(a, b) = \int_0^1 u^{a-1} (1 - u)^{b-1} \, du, \quad a, b > 0.$
We will extensively use the classical identities:
$\Gamma(a + 1) = a \, \Gamma(a) 
\; \text{and} \; 
B(a, b) = \frac{\Gamma(a)\Gamma(b)}{\Gamma(a + b)}.$

\smallskip 
\noindent $\bullet$  Let $\mathrm{Pol}(\mathbb{R})$ denote the ring of all polynomials on $\mathbb{R}$ and $\mathrm{Pol}_n(\mathbb{R})$ the subspace consisting of polynomials of degree at most $n$.

\noindent Throughout this paper, we work on a filtered probability space $(\Omega,\mathcal{F},\mathbb{F}:=(\mathcal{F}_t^{X_0,W})_{t\in[0,T]},\mathbb{P})$
suppor-\\ting a \(1\)-dimensional Brownian motion \(W\), where \((\mathcal{F}_t^{X_0,W})_{t\in[0,T]}\) denotes the filtration generated by \(X_0\) and \(W\).
\subsection{Main tools and standing assumptions}\label{subsec:main_tools}
 In this subsection, we provide specific analytical tools, such as the \emph{resolvent} and the solution of the \emph{Wiener--Hopf equation}.\\
\noindent Throughout this work, we adopt, following \cite{Pages2024, EGnabeyeu2025}, the following definition of the resolvent, which is also discussed in works such as \cite{Pruss1993} and offers a distinct perspective from the functional resolvent introduced in \cite{gripenberg1990}.

\vspace{-.3cm}
\paragraph{Solvent core of a convolution kernel.}
Let $K$ be a convolution kernel   satisfying 
 $\int_{0}^tK(u)du>0$ for every $t>0$. 
For every $\lambda \!\in \R$,  the {\em resolvent or Solvent core} $R_{\lambda}$ associated to $K$ and $\lambda$ is defined as the unique solution -- if it exists --  to the deterministic Volterra equation
 \vspace{-.2cm}
\begin{equation}\label{eq:Resolvent}
	\forall\,  t\ge 0,\quad R_{\lambda}(t) + \lambda \int_0^t K(t-s)R_{\lambda}(s)ds = 1,
	 \vspace{-.3cm}
\end{equation}
or, equivalently, written in terms of convolution, 
$R_{\lambda}+\lambda K*R_{\lambda} = 1.$
This equation is also known as \textit{resolvent equation} or \textit{renewal equation}. Its solution  always satisfies $R_{\lambda}(0)=1$ and admits the formal \textit{Neumann series expansion} \footnote{Recall that $ K^{1*} = K $ and $ K^{k*}(t) = \int_0^t K(t - s) \cdot K^{(k-1)*}(s) \, ds.$}:
 \vspace{-.2cm}  
\begin{equation}\label{eq:Resolvent3}
	R_{\lambda} = \sum_{k\ge 0} (-1)^k \lambda^k (\mbox{\bf 1}*K^{k*}). 
	 \vspace{-.2cm}
\end{equation}
where \(K^{k*}\) denotes the $k$-fold $*$
product of \(K\) with itself, with the convention $K^{0*}= \delta_0$ (Dirac mass at $0$).
\medskip From now on we will assume that the kernel $K$ has a finite Laplace transform  \(\forall\, t>0, \;L_K(t)<+\infty.\) Note that, as mentioned e.g. in \cite{Pages2024}, if the (non-negative) kernel $K$  satisfies 
\vspace{-.1cm}
\begin{equation}\label{eq:Kcontrol}
	0\le K(t)\le Ce^{bt }t^{a-1}\mbox{  for some }\; a,\, C>0, \; b\geq0. 
	\vspace{-.2cm}
\end{equation} then, by induction  \(\mbox{\bf 1}*K^{*n}(t) \le C^n e^{bt}\frac{\Gamma(a)^n}{\Gamma(an+1)} t^{an},\) so that for such kernels, the above series~\eqref{eq:Resolvent3} is absolutely converging for every $t>0$ implying that the function   $R_{\lambda}$ is well-defined on $(0, +\infty)$. 

\smallskip
\noindent{\bf Remark} 1. If $K$ is continuous on $(0,+\infty)$, then the resolvent $R_{\lambda}$ is differentiable and one checks that $f_{\lambda}=-R'_{\lambda}$ satisfies, for every  $t>0$, \(-f_{\lambda}(t) +\lambda \big( R_{\lambda}(0)K(t) - K *f_{\lambda}(t)\big)=0\)
that is $f_{\lambda}$ is solution to the equation
\begin{equation}\label{eq:flambda-eq}
	f_{\lambda} +\lambda K *f_{\lambda}=\lambda   K.
\end{equation}
2. Taking the Laplace transform from both side of the above equality~\eqref{eq:flambda-eq}, we have that :
$L_{f_{\lambda}}(t)(1+\lambda L_K(t))=\lambda L_K(t) $, $t>0$. Consequently, 
 \vspace{-.2cm}
\begin{equation}\label{eq:LaplaceTf_lambda}
	L_{f_{\lambda}}(t) = \frac{\lambda L_K(t)}{1+\lambda L_{K}(t)}.
	 \vspace{-.2cm}
\end{equation}
so that, for $\lambda \geq 0, $ $L_{f_{\lambda}}(t) \equiv 0$ if and only if 
$L_K(t) \equiv 0$ i.e. if and only if $K=0$ by the injectivity of Laplace transform.

3. If $\displaystyle \lim_{t\to +\infty} R_{\lambda}(t) =0$ then, one also has that $\int_0^{+\infty} f_{\lambda}(t)dt = 1 -R_{\lambda}(+\infty) = 1$.Moreover, if $R_{\lambda}$ turns out to be non-increasing, then $f_{\lambda}$ is non-negative and satisfies $0\le f_{\lambda} \le \lambda K$, so that $f_{\lambda} $ {\em is a probability density}.

\smallskip
In what follows, unless otherwise specified, \(K\) denotes a kernel (exponential, fractional, or Gamma), and \(f_\lambda\) denotes the negative of the derivative of its \(\lambda\)-resolvent \(R_\lambda\). In the case of the fractional (resp. Gamma) kernel, we use the subscript \({}_{\alpha,\lambda}\) (resp. \({}_{\alpha,\rho, \lambda}\)) whenever the context is sufficiently clear.
\begin{Example}[Examples of kernels, Laplace transform and $\lambda-$ Resolvent.]\label{Ex:SolventGammaKernel}\textcolor{white}{.}
	
	\smallskip
	\noindent {\em 1. Trivial kernel (Markov)}. Let $K(t) = \mbox{\bf 1}_{\R_+}(t)$. It obviously satisfies
	~\eqref{eq:Kcontrol}. Then $R_{\lambda} (t)= e^{-\lambda t}$. 
	
	\smallskip
	\noindent {\em 2. Fractional integration kernel}. Let $K(t) = K_{\alpha}(t) = \frac{u^{\alpha-1}}{\Gamma(\alpha)} \mbox{\bf 1}_{\R_+}(t),  \; \alpha>0.$
	These kernels trivially satisfy~\eqref{eq:Kcontrol}.
	This  family of  kernels corresponds to the fractional integrations of order $\alpha >0$. 
	
	\noindent The Laplace transform associated to the kernel $K_\alpha$ reads, for $t>0$
	$L_{K_\alpha}(t):=\int_0^{+\infty}e^{-t u}K_\alpha(u)du= t^{-\alpha}.$
	It follows from the  easy identity  $K_{\alpha}*K_{\alpha'}= K_{\alpha+\alpha'}$ and the Neumann series expansion provided in equation~\eqref{eq:Resolvent3} that the resolvent
	reads: 
	\begin{equation}\label{eq:SolventFracKernel}
		R_{\alpha,\lambda}(t) = \sum_{k\ge 0} (-1)^k \frac{\lambda ^k t^{\alpha k}}{\Gamma(\alpha k+1)}= E_{\alpha}(-\lambda t^{\alpha} ) \; t\ge 0,\; \text{where}\; E_{\alpha}(t) = \sum_{k\ge 0} \frac{t^k}{\Gamma(\alpha k+1)},\ t\!\in \R.
	\end{equation}
	i.e. $E_{\alpha}$ denotes the {\em standard   Mittag-Leffler function}. For $\lambda >0$, we also define minus its derivative, namely, the function $f_{\alpha, \lambda}:=- R'_{\alpha, \lambda}$ on $(0,+\infty)$ by
	\begin{align}
		\label{eq:DerivSolventFracKernel} f_{\alpha, \lambda}(t) &= - R'_{\alpha, \lambda}(t) = \alpha\lambda t^{\alpha-1} E'_{\alpha}(-\lambda t^{\alpha})  = \lambda t^{\alpha-1}\sum_{k\ge 0}(-1)^k\lambda^k \frac{t^{\alpha k}}{\Gamma(\alpha (k+1))}
	\end{align}
	Note that when $\alpha =1$ (i.e. $K=K_1=\mbox{\bf 1}$), $E_{1}(t) = e^t$, $R_{1,\lambda}(t)= e^{-\lambda t}$ and $f_{1,\lambda}(t)=\lambda e^{-\lambda t}$.
	%
	
	\smallskip
	\noindent {\em 3. Exponential-fractional integration kernel}.  Let $K(t):=K_{\alpha , \rho }(t):= e^{-\rho t}\frac{ t^{\alpha - 1}}{\Gamma(\alpha)}  \cdot \mathbf{1}_{(0,\infty)}(t)
	$, for \(\alpha >0 \) and \(\rho > 0\). One checks that these kernels also satisfy trivially~\eqref{eq:Kcontrol}.
	It follows from the easy identity \(K_{\alpha, \rho} * K_{\alpha^\prime, \rho} = K_{\alpha + \alpha^\prime, \rho}\) and the Neumann series expansion provided in ~\eqref{eq:Resolvent3} that the resolvent
	reads:
	{\small
		\begin{equation}\label{eq:SolventGammaKernel}
			R_{\alpha, \rho, \lambda}(t) =(1*\delta_0)(t) + \sum_{k\ge 1} (-1)^k \lambda^k (\mbox{\bf 1}*K_{\alpha, \rho}^{(k*)}) = \mathbf{1}_{\mathbb{R}_+}(t) + \sum_{k \geq 1} (-1)^k \lambda^k \int_0^t \frac{e^{-\rho s} s^{k\alpha -1}}{\Gamma(k\alpha)} \,ds.
		\end{equation}
	}
	Hence, if $\lambda > 0$, we define the function $f_{\alpha, \rho, \lambda}:= - R_{\alpha, \rho, \lambda} $ on $(0, +\infty)$ by:
	{\small
		\begin{equation}\label{eq:DerivSolventGammaKernel}
			f_{\alpha, \rho, \lambda}(t) = -\frac{d}{dt} R_{\alpha, \rho, \lambda}(t)
			= -\sum_{k \geq 1} (-1)^k \lambda^k \frac{e^{-\rho t} t^{k\alpha -1}}{\Gamma(k\alpha)} = \lambda e^{-\rho t} t^{\alpha - 1} \sum_{k \geq 0} (-1)^k \lambda^k \frac{ t^{\alpha k }}{\Gamma(\alpha (k+1))}. 
		\end{equation}
	}
		We note that \(	f_{\alpha, \rho, \lambda}(t) = e^{-\rho t} f_{\alpha, \lambda}(t) =\alpha \lambda e^{-\rho t} t^{\alpha-1} E^\prime_{\alpha}( - \lambda t^{\alpha})\). A straightforward computation using Tauberian Final Value Theorem shows that \(\lim_{t \to \infty} R_{\alpha, \rho, \lambda}(t) = \frac{1}{1 + \lambda \rho^{-\alpha}} \in [0, 1)\) and \(\lim_{t \to \infty} f_{\alpha, \rho, \lambda}(t) = 0\).
\end{Example}

\vspace{-.4cm}
\paragraph{Standing assumptions.}
 Throughout the paper,
we will always work under the following assumptions:
\begin{Assumption}[On the convolution kernels]\label{assum:convolKernel} \textcolor{white}{.}\\
		Assume that the kernel $K\in {\cal L}^{2\beta}_{loc}(\R_+)$ for some \(\beta\geq1\) and satisfies for every \(0<T<+\infty\):
	{\small   
		\begin{equation}\label{eq:contKtilde}
			(\widehat {\cal K}^{cont}_{T, \widehat \theta})\;\;\exists\,\widehat\kappa_{T, \beta}< +\infty,\; \exists\, \widehat\theta\in (0,1],\;\forall\bar{\delta}\!\in (0,T],\; \widehat \eta(\delta) := \sup_{t\in [0,T]} \Big[\int_{(t-\bar{\delta})^+}^t \hskip-0,25cm K\big(t-u\big)^{2\beta} du\Big]^{\frac{1}{2\beta}}\le \widehat \kappa_{T,\beta} \,\bar{\delta}^{\,\widehat \theta} \; 
		\end{equation}
		\begin{equation}\label{eq:Kcont}
			({\cal K}^{cont}_{T,\theta}) \;\exists\, \kappa_{_T,\beta}< +\infty,\; \exists\, \theta >0,\; \forall\,\delta \!\in (0,T),\;
			\sup_{t\in [0,T]} \Big[\int_0^t |K(s+\delta)-K(s)|^{2\beta}ds \Big]^{\frac{1}{2\beta}} \le  \kappa_{_T,\beta}\,\delta^{\theta}.
		\end{equation}
	}
\end{Assumption}
\noindent One readily verifies that the kernels \(K\) in Examples~\ref{Ex:SolventGammaKernel}~$2.$ and~$3.$ satisfy~\eqref{eq:contKtilde} and~\eqref{eq:Kcont} of Assumption~\ref{assum:convolKernel} for every \(\alpha>\frac12\) with any \(\beta\in[1,\frac{1}{2(\alpha-1)^-})\) and \(\widehat \theta=\theta=(\alpha-1 + \frac{1}{2\beta})\wedge 1\) ( see e.g.~\cite{RiTaYa2020} or~\cite{JouPag22} or~\cite{GnabeyeuPages2026} along many others) 
Here we adopt the convention $1/0 = +\infty$ and \(x^-:=\max(-x, 0)\).

\bigskip
Consider the Volterra equation~\eqref{eq:Volterrameanrevert} with diffusion coefficient with separable variables $\forall (t,x) \in \mathbb{T} \times \mathbb{R}, \; \sigma(t,x) = \varsigma(t)\,\sigma(x),$ where \(\varsigma(t)>0\) and \(\sigma(x)>0\). 
Let \(X_0\) be a random initial condition, and define
\[0\neq m_0:=\mathbb{E}[X_0], \quad v_0:=\text{Var}(X_0) \geq 0, \quad \text{and} \quad \bar \sigma^2_0:=\mathbb{E}[\sigma^2(X_{0})] \geq 0.\]
\begin{assumption}[$\lambda$-resolvent $R_{\lambda}$ of the kernel]\label{ass:resolvent} We assume that the $\lambda$-resolvent $R_{\lambda}$ of the kernel $K$ is well-defined
	on $(0, +\infty)$ and that, together with the mean-reverting function \(\mu\), it satisfies the following conditions for every $\lambda > 0$:
	\begin{equation}\label{eq:hypoRlambda}
		({\cal R}_\lambda)\quad
		\left\{
		\begin{array}{ll}
			(i) & R_{\lambda}(t) \text{ is } \text{differentiable on } \mathbb{R}^+,\; R_{\lambda}(0)=1 \text{ and } \lim_{t \to +\infty}R_{\lambda}(t) =a \in [0,1[, \\
			(ii) &  f_{\lambda} \in {\cal L}^2(\mathbb{R}_+), f_\lambda(t) \neq 0\; dt-a.e., \text{ where we set } \quad f_{\lambda} := -R'_{\lambda} \text{ for } t > 0,\\
			(iii) & \mu \text{
				$ C^1$-function such that }  
			\lim_{t\to +\infty} \mu (t) = \mu_{\infty} \in \mathbb{R} \text{ and } \frac{\mu_\infty}{m_0}\geq 0.
		\end{array} 
		\right.
	\end{equation}
\end{assumption}
\noindent Note that $L_{f_\lambda}(t) > 0\; dt-a.e$ and under assumptions \(({\cal R}_\lambda)\) $(i)$ and $(ii)$, \( f_{\lambda} \) is a \((1-a)\)-sum measure, i.e., \( \int_0^{+\infty} f_{\lambda}(s) \, ds = 1-a \).
In fact, since $L_K(t) >0$ for any $t\in \R_+$, we deduce from~\eqref{eq:LaplaceTf_lambda} that $L_{f_{\lambda}}>0$. Additionally,
\[\int_0^{+\infty} f_{\lambda}(s)\,ds = [1 - R_{\lambda}(s)]_{s=0}^{s=+\infty} = -\lim_{s\to +\infty}R_{\lambda}(s) + R_{\lambda}(0) = 1-a\]
	Moreover, as a direct application of~\cite[Lemma 3.5.]{EGnabeyeu2025}, \(\lim_{t\to +\infty} \int_0^t f_{\lambda}(t-s) \mu (s)ds = \mu_{\infty} (1-a) \).
\begin{Definition}\label{def:stabilizer}
	We will call the stabilizer (or corrector) of the scaled stochastic Volterra equation ~\eqref{eq:Volterrameanrevert} a (locally) bounded Borel function \( \varsigma \), which is a solution (if any) to the functional equation:
	\begin{equation}\label{eq:VolterraStabilizer}
		\textit{($E_{\lambda, c}$)}: \quad\forall\, t\ge 0, \quad c \lambda^2\big(1-(\phi - f_{\lambda} * \phi)^2(t) \big) =  (f_{\lambda}^2 * \varsigma^2)(t).
	\end{equation}
	where 
	\begin{equation}\label{eq:notcst_phi}
		c = \frac { v_0 }{\bar \sigma^2_0}>0 \quad  \textit{and} \quad  \phi(t) = 1 - \lambda \int_0^t K(t-s) \left( \frac{\mu(s)}{\lambda m_0} - 1 \right) \, ds\end{equation}
	Consequently, the stabilizer \(\varsigma\) depends on the parameters \(\lambda\) and \(c\); whenever no ambiguity arises, we will simply denote \(\varsigma = \varsigma_{\lambda,c}\).
\end{Definition}
\noindent{\bf Remarks:}
\noindent 1.	Inserting Equation~\eqref{eq:notcst_phi}, ($E_{\lambda, c}$) in ~\eqref{eq:VolterraStabilizer} can also be re-written as follows:
\begin{equation}\label{eq:VolterraStabilizer2}
	\textit{($E_{\lambda, c}$)}: \quad\forall\, t\ge 0, \quad c \lambda^2\Big(1-\big(1 - \frac{(f_{\lambda} * \mu)_t}{\lambda m_0}\big)^2 \Big) =  (f_{\lambda}^2 * \varsigma^2)(t) \quad  \textit{with} \quad c = \frac { v_0 }{\bar \sigma_0^2} \quad  \textit{and} \quad \varsigma := \varsigma_{\lambda,c} .
\end{equation}
\noindent2. Note that ~\eqref{eq:VolterraStabilizer}-\eqref{eq:VolterraStabilizer2} has a solution \( \varsigma_{\lambda, c} \) for some \( c > 0 \) if and only if it has a solution \( \varsigma_{\lambda, 1} \) when \( c = 1 \), and \( \varsigma_{\lambda, c} = \sqrt{c} \varsigma_{\lambda, 1} \). Hence, ($E_{\lambda, c}$) can be replaced by ($E_{\lambda, 1}$) denoted ($E_{\lambda}$) for simplicity.

\smallskip
\noindent 3. By~\cite[Lemma 3.10]{EGnabeyeu2025}, the solution to the above equation~\eqref{eq:VolterraStabilizer} (if any) is unique and belongs to $ {\cal L}^1_{\text{loc}}((\R_+))$. The crucial requirement is the existence and positivity of the solution.

\smallskip
\noindent 4. 
If \(\mu\in {\cal L}^1(\mathbb{R}_+)\), then \(\varsigma^2\in {\cal L}^1(\mathbb{R}_+)\). Indeed, integrating both sides of Equation~\eqref{eq:VolterraStabilizer2} and applying Fubini--Tonelli's theorem together with a suitable change of variables yields
\[\|f_\lambda\|_{{\cal L}^2(\R_+)}^2\int_{0}^\infty \varsigma^2(s)\,ds = \frac{2c\lambda}{m_0}\|f_\lambda\|_{{\cal L}^1(\R_+)}\|\mu\|_{{\cal L}^1(\R_+)}- \frac{c}{m_0^2}\|f_{\lambda} * \mu\|_{{\cal L}^2(\R_+)}^2.\]
The conclusion then follows from Young's convolution inequality together with the fact that \(f_\lambda\) is a \((1-a)\)-sum measure with finite \( {\cal L}^2(\R_+)-\)norm under Assumption~\(({\cal R}_\lambda)\).

\medskip
We introduce the following standing existence assumption for Equation~\eqref{eq:VolterraStabilizer}-\eqref{eq:VolterraStabilizer2}, with particular emphasis on the existence and positivity property. 
\begin{Assumption}[On the stabilizer]\label{ass:on_stabilizer}
	Let \(\lambda>0\). There exists a  
	positive Borel solution \( \varsigma_{\lambda} \) on \( ]0, +\infty) \) of the equation $ \textit{($E_{\lambda}$)}: \quad
	\forall\, t> 0, \quad \lambda^2\big(1-\big(1 - \frac{(f_{\lambda} * \mu)_t}{\lambda m_0}\big)^2(t) \big) =  (f_{\lambda}^2 * \varsigma^2)(t)$.
\end{Assumption}
To unify the two settings under consideration in this work, namely, the autonomous and the time-dependent diffusion ones, we assume that the equation~\eqref{eq:VolterraStabilizer} is satisfied by $\varsigma$ and we adopt the convention
\begin{equation}\label{eq:VolterraStabilizerCont}
	\varsigma = \begin{cases}
		1, & \text{in the autonomous diffusion setting , i.e., } \quad \sigma(t,x) =\sigma(x), \\
		\varsigma_{\lambda,c} \neq 1, & \text{in the time-dependent diffusion setting , i.e., }  \quad \sigma(t,x) = \varsigma_{\lambda,c}(t)\,\sigma(x).
	\end{cases} .
\end{equation}
	\begin{Assumption}[Integrability and H\"older continuity]\label{ass:int_holregul}
	Let \( \lambda, c > 0 \), and assume the kernel \( K \) and its \( \lambda \)-resolvent \( R_{\lambda} \) satisfy:
	\begin{equation}\label{eq:Integf}
	\int_0^{+\infty} f_{\lambda}^{2\beta}(u) \, du < +\infty \quad \text{for some } \beta > 1, \quad f_\lambda \in {\cal L}^1(\R_+) \cap {\cal L}^\infty(\R_+)
	\end{equation}
	and there exists \( \vartheta \in (0, 1)\), and a real-valued positive function \(h_{\lambda}\) such that
	 for every \(t\geq0,\;\)
	\begin{equation}\label{eq:Regulf}
		|f_{\lambda}(t + \bar{\delta}) - f_{\lambda}(t)| \leq h_\lambda(t)\bar{\delta}^{\vartheta} \quad\text{with}\; h_\lambda \in {\cal L}^1(\R_+) \cap {\cal L}^\infty(\R_+) \quad\text{and}\quad h_\lambda*\varsigma^2_{\lambda,c} \in {\cal L}^\infty(\R_+).
	\end{equation}
\end{Assumption}
This assumption provide a uniform H\"older continuity or H\"older regularity with exponent \( \vartheta \) for the function \( f_{\lambda} \), ensuring controlled behavior as \( t \) and \( t + \delta \) become arbitrarily close.

\smallskip
\noindent{\bf Remark: } if \(\mu \in {\cal L}^1(\R_+)\), then by the claim~$4.$ of the Remark following Definition~\ref{def:stabilizer} above, \(\varsigma^2_{\lambda,c} \in {\cal L}^1(\R_+)\) and since we have moreover \(h_\lambda\in {\cal L}^\infty(\R_+)\), it follows by Young's inequality that \(h_\lambda*\varsigma^2_{\lambda,c} \in {\cal L}^\infty(\R_+)\).

\paragraph{Wiener--Hopf transform.}
Let $\lambda > 0$ and let $\mu : \mathbb{R} \to \mathbb{R}$ be a Borel function. Assume the kernel $K$ satisfies the above assumptions $({\cal R}_\lambda)$, \eqref{eq:contKtilde} and \eqref{eq:Kcont} from Assumption \ref{assum:convolKernel}  and its $\lambda$-resolvent $R_\lambda$ is well-defined
on $(0, +\infty)$.
Then, the solution $(X_t)_{t \geq 0}$ of the Volterra equation ~\eqref{eq:Volterrameanrevert} also satisfies:
\begin{equation}\label{eq:Volterrameanrevert_}
	X_t = X_0\Big(\phi(t) - \int_0^t f_{\lambda}(t-s) \phi(s) \, ds \Big) + \frac{1}{\lambda} \int_0^t f_{\lambda}(t-s) \mu(s) \, ds + \frac{1}{\lambda} \int_0^t f_{\lambda}(t-s) \sigma(s, X_{s}) \, dW_s.
\end{equation}	
The converse also holds, namely any process satisfying~\eqref{eq:Volterrameanrevert_} is also a solution to the original Volterra equation~\eqref{eq:Volterrameanrevert}. This follows directly from~\cite[Proposition~3.1]{EGnabeyeu2025}.

\smallskip
\noindent {\bf Remark.}   Notably, in the Markovian case, the Wiener--Hopf equation~\eqref{eq:Volterrameanrevert_} amounts to applying It\^o's lemma to the transformed process \( e^{\lambda t} X_t \). In fact, if $K(t) = \mathbf{1}$ in the Volterra equation, then \(
R_\lambda(t) = e^{-\lambda t} \) and  \( f_\lambda(t) = \lambda e^{-\lambda t},\) so that  the above computation corresponds to  It\^o' s Lemma applied to $e^{\lambda t} X_t$.
\vspace{-.1cm}
\subsection{Main result}\label{subsec:main_result}
We now present our main result. First, we recall the main concept underlying the notion of stationary regime for the ``scaled" stochastic Volterra equation~\eqref{eq:Volterrameanrevert}.
\vspace{-.1cm}
\begin{Definition}[Fake stationary regime of type I and II (see. \cite{Pages2024})]\label{def:stationarity_order}.
	\begin{enumerate}
		\item The process $(X_t)_{t\ge 0}$ starting from $X_0\!\in L^2(\P)$,  exhibit a fake stationary regime of type I if:
		\begin{equation}\label{eq:FakeStats}
			\forall\, t\ge 0, \quad	\mathbb{E}[X_t] = m_0, \quad \text{Var}(X_t) = v_0 \geq 0, \quad \text{and} \quad \bar{\sigma}^2(t) = \mathbb{E}[\sigma^2(X_{t})]=\bar \sigma^2_0 \geq 0.
			\vspace{-.1cm}
		\end{equation}
		\item The process $(X_t)_{t\ge 0}$ starting from $X_0$ 
		has a fake stationary regime of type II if \( (X_t)_{t \geq 0} \) has the same marginal distribution, i.e., \( X_t \overset{d}{=}X_0 \) for every \( t \geq 0 \).
	\end{enumerate}
	\vspace{-.1cm}
\end{Definition}
Let \(\sigma(t,x) = \varsigma(t)\,\sigma(x)\) where \(\varsigma\) satisfies Equation~\eqref{eq:VolterraStabilizerCont} and
consider like \cite{Pages2024, EGnabeyeu2025} a square-root trinomial diffusion coefficient \(\sigma\) i.e. \(\sigma^2(x) \in \mathrm{Pol}_2(\mathbb{R})\) (Polynomial of degree 2).
	\begin{equation}\label{eq:sigmafakeI&II}
		\sigma(x) = \sqrt{ \kappa_0 +\kappa_1\,x +\kappa_2\,x^2}\quad \mbox{ with }\quad  \kappa_i\ge 0,\;i=0,2, \;\kappa^2_1 \le 4\kappa_2\kappa_0. 
	\end{equation}
	Define \(\ell_\infty
	:=1-(1-a)\frac{\mu_\infty}{\lambda m_0} \). Note that \(\ell_\infty=\lim_{t\to +\infty} 1-\frac{(f_\lambda * \mu)(t)}{\lambda m_0} \). 
\begin{Theorem}\label{thm:mainGamma}
	Let \( \alpha \in \left( 1, \frac32\right) \) , let \(\rho\geq0\), let \( K(t) = K_{\alpha,\rho}(t) =  e^{-\rho t}\frac{t^{\alpha-1}}{\Gamma(\alpha)} \), \( t > 0 \) the exponential-fractional kernel, set \(\phi\) as in~\eqref{eq:notcst_phi} and let $\sigma(t,x):= \varsigma(t)\sigma(x)$ in ~\eqref{eq:Volterrameanrevert}-\eqref{eq:Volterrameanrevert_} where \( \sigma : \mathbb{R} \to \mathbb{R} \) is the Lipschitz continuous function~\eqref{eq:sigmafakeI&II}, with \([\sigma]_{\text{Lip}}=\kappa_2\) and $\varsigma $ is a non-negative, continuous function given by Equation~\eqref{eq:VolterraStabilizerCont} for some fixed $\lambda >0$ and 
	 \( c \in \left( 0, \frac{\alpha-1}{8\kappa_2} \right) \). Let \( X_0 \in L^p(\P) \) for some $p>\frac{2}{\alpha-1}>2$ such that \( \mathbb{E}[X_0] =m_0\) and \(Var(X_0) = v_0=\frac{c\sigma^2( m_0)}{1-c\kappa_2}.\)
	Then,
	\begin{enumerate}
		\item Assume that the diffusion coefficient  $\sigma$ is non-degenerated (in the sense $\sigma(m_0) \neq0$) 
		
		\noindent $(a)$ Case $\kappa_2 > 0$. The solution \( (X_t)_{t \geq 0} \) to the Volterra equation~\eqref{eq:Volterrameanrevert} or ~\eqref{eq:Volterrameanrevert_} starting from $X_0$ has a fake stationary regime of type I, in the sense that
		\vspace{-.3cm}
		\begin{equation}\forall\, t\ge 0, \quad 
		\mathbb{E}[X_t] = m_0, \; \text{Var}(X_t) = v_0= \frac{c\sigma^2(m_0)}{1 - c\kappa_2},\; \text{and}\;\; \mathbb{E}[\sigma^2(X_t)] = \bar{\sigma}_0^2 =\frac{\sigma^2(m_0)}{1 - c\kappa_2}.
		\vspace{-.1cm}
		\end{equation}
		\noindent $(a)$ Case $\kappa_2 = 0$. Then $\kappa_1=0$ and \((X_t)_{t\geq0}\) is a Gaussian process with a fake stationary regime of type~II, with \( \nu \) as one-dimensional marginal distribution, provided that if $X_0 \sim \nu := \mathcal{N}\!\left(m_0,c\kappa_0^2\right)$.
		
		\item The family of shifted processes \( X_{t+\cdot}, t \geq 0 \), is \( C \)-tight as \( t \to +\infty \) and its (functional) limiting distributions are all \( L^2 \)-stationary processes with mean \(m_0\) and autocovariance function \( C_\infty \) given, for \( t_1, t_2 \geq 0 \), \( t_1 \leq t_2 \), by \(C_\infty:=C_{f_{\alpha,\rho,\lambda}}(t_1,t_2)\)
		\vspace{-.3cm}
			\begin{equation}
				\text{Cov}(X_{t+t_1}, X_{t+t_2}) \overset{t \to +\infty}{\to} C_{f_{\alpha,\rho,\lambda}}(t_1,t_2) := \ell_\infty^2 v_0 + \frac{\varsigma_{\infty}^2 v_0 }{c\lambda^2}\int_0^{+\infty}  f_{\alpha,\rho,\lambda}(t_2-t_1+u)f_{\alpha,\rho,\lambda}(u)du.
			\end{equation}
		\vspace{-.3cm}
		$$\varsigma_{\infty}^2 :=\lim_{t \to +\infty}\varsigma_{\alpha,\rho,\lambda,c}^2(t)  = \begin{cases}
			1, & \text{in the autonomous diffusion setting of~\eqref{eq:VolterraStabilizerCont}},\\
			\frac{c\lambda^2 (1-\ell_\infty^2)}{\int_0^{+\infty}f^2_{\alpha,\rho, \lambda}(s)ds}, & \text{in the time-dependent diffusion setting of~\eqref{eq:VolterraStabilizerCont}}.
		\end{cases} $$
		
		\item {\em $L^p$-confluence:} If $m_0= (1-a) \frac{\mu_\infty}{\lambda}$ i.e. \(\ell_\infty=0\), then for any starting random variable \( X_0^\prime  \in L^p(\P) \), a solution to the Volterra SDE ~\eqref{eq:Volterrameanrevert}-\eqref{eq:Volterrameanrevert_} starting from \( X_0^\prime \) satisfies \( \| X_t^\prime - X_t \|_p\to 0\) as $t \to \infty$ and
	
		\centerline{$\mathbb{E}[X_t] \to m_0, \quad \text{and} \quad \text{Var}(X_t) \to \frac{c\sigma^2(m_0)}{1 - c\kappa_2} \quad \text{as} \quad t \to +\infty.$}
	\end{enumerate}
\end{Theorem}
\noindent The main objective of Sections~\ref{sec:appl2} and~\ref{sec:appl3} is to establish that, for both the fractional kernel \(K_\alpha\) and its exponentially damped counterpart (the gamma kernel) \(K_{\alpha,\rho}\) the associated \(\lambda\)-resolvent \(R_\lambda\) and minus its derivative \(f_\lambda\) satisfy Assumptions~\ref{ass:resolvent} and~\ref{ass:int_holregul}.
Moreover, in the particular case considered in our numerical experiments (\(\mu\) constant; see Section~\ref{Sec:Num}), we show that there exists a positive solution to Equation~\eqref{eq:VolterraStabilizer2}, so that Assumption~\ref{ass:on_stabilizer} holds.

\vspace{-.4cm}
\section{Fake stationarity of a scaled stochastic Volterra Integral equation}\label{sec:genfakestatio}
\vspace{-.2cm}
 We first begin with a brief discussion of the well-posedness of such Volterra SDEs. Then, in sections~\ref{subsec:fakestatioIntrin} and~\ref{subsec:fakestatioStabil}, we will study the consequences of the three constraints in Equation~\ref{eq:FakeStats} under different settings of the Volterra equation~\eqref{eq:Volterrameanrevert}.
\vspace{-.3cm}
\subsection{A primer on Volterra SDE with convolution and Long-memory kernels}\label{subsec:background}
We are interested in the convolution stochastic Volterra equation: 
\vspace{-.2cm}
\begin{equation}\label{eq:Volterra}
	X_t=x_0(t) +\int_0^t K(t-s) b(s,X_s)ds+\int_0^t K(t-s)\sigma(s,X_s)dW_s,\quad t\ge 0.
	\vspace{-.1cm}
\end{equation}
where $b:\mathbb{T}\times\R\to \R$, $\sigma:\mathbb{T}\times\R\to\R$ are Borel  measurable, $K\!\in {\cal L}^2_{loc}(\R_+)$  is a convolution kernel and $(W_t)_{t\ge 0}$ is a standard  Brownian motion independent from the $\R$-valued initial random function $x_0$ both defined on a probability space $(\Omega,{\cal A}, \P)$. Let $(\mathcal F_t)_{t\ge 0}$ be a filtration (satisfying the usual conditions) such that $x_0(0)$ is ${\cal F}_0$-measurable  and $W$ is an $({\cal F}_t)$-Brownian motion independent of $x_0$. 

\smallskip
As will be shown in Sections~\ref{sec:appl2} and~\ref{sec:appl3}, the stabilizing function \(\varsigma\) appearing in the inhomogeneous diffusion setting \eqref{eq:VolterraStabilizer}~\eqref{eq:VolterraStabilizerCont} is may be unbounded e.g. for \(K_{\alpha}\) when $\alpha>1$. Consequently, the associated diffusion coefficient does not satisfy uniform-in-time linear growth and Lipschitz conditions. This motivates to state a refined well-posedness conditions for the solution, as highlighted below.

\begin{Assumption}[On Volterra Equations with convolution kernels]\label{assum:CoefsVolterra}
	For every $T>0$, assume that the drift $b$ and the diffusion coefficient
	$\sigma$ satisfy time-dependent Lipschitz and linear growth conditions, with coefficients $h_{b,\sigma}$ and $\ell_{b,\sigma}$,respectively.
	{\small
		\begin{align*}
			(i) &\; \forall t \in [0,T], \; \forall x, y \in \mathbb{R}, \; |b(t, x) - b(t, y)| + |\sigma(t, x) - \sigma(t, y)| \le h_{b,\sigma}(t)\, |x - y|\;\\
			(ii)&\; \forall t \in [0,T],\; \forall x \in \mathbb{R},\; |b(t, x)| + |\sigma(t, x)| \leq \ell_{b,\sigma}(t)\, (1 + |x|) \;
			 , \\
			(iii) &\; \text{Moreover, the initial random function  $x_0$ is $(\mathcal F_t)_{t\in[0,T]}$-adapted and satisfies for some } \delta > 0, \text{and } p > 0, \\ 
			&\quad \mathbb{E}\Big(\sup_{t \in [0,T]} |x_0(t)|^p \Big) < +\infty, \quad \mathbb{E} |x_0(t') - x_0(t)|^p \leqslant C_{T,p} \Big( 1 + \mathbb{E}\big[ \sup_{t \in [0,T]} |x_0(t)|^p \big] \Big) |t' - t|^{\delta p},
		\end{align*}
	where, in accordance with Assumption~\ref{assum:convolKernel}, the functions $h_{b,\sigma}$, $\ell_{b,\sigma}$ and the Kernel $K$ 
	are integrable on $[0,T]$ with an exponent strictly larger than 2, namely,
		\vspace{-.2cm}
		\begin{equation}\label{eq:Hol_Intergral_Cond}
			\int_0^T\big(\ell_{b,\sigma}(u)^{{\frac{2\beta}{\beta-1}}} + h_{b,\sigma}(u)^{{\frac{2\beta}{\beta-1}}}
			+K(u)^{2\beta}\big)du<+\infty, \; \text{for some} \; \beta>1.
			\vspace{-.3cm}
		\end{equation}
	}
\end{Assumption}
\noindent{\bf Remark :} In particular $\ell_{b,\sigma}, h_{b,\sigma} \in {\cal L}^2([0,T])$. Moreover, if the diffusion coefficient $\sigma(\cdot,\cdot)$ instead satisfies a Lipschitz-linear growth assumption {\em uniformly in time}, then 
 we only require the weaker assumption \(K, \ell_{b}, h_{b} \in {\cal L}^2([0,T]) \), which is sufficient to establish the below existence result.
\begin{Theorem}[Existence, Pathwise Regularity and Maximal Inequality]\label{thm:ExistenceVolterraL2}
	Let $T > 0$ and let $p \geq p_{eu}:=2\vee\frac{1}{\delta}\vee \frac{1}{\widehat\theta}\vee\frac{1}{\theta}$ such that
	 Assumptions~\eqref{assum:convolKernel} and~\eqref{assum:CoefsVolterra} are in force. Then, the Volterra equation~\eqref{eq:Volterra} admits, up to a $\P$-indistinguishability, a unique $({\cal F}_t)$-adapted strong solution $X=(X_t)_{t\in [0,T]}$, pathwise continuous.
	This solution also satisfies: 
	\vspace{-.2cm}
	\begin{equation}\label{eq:Lpincrements}
		\forall\, s,t\in [0,T],\quad
		\mathbb{E}\left[|X_t-X_s|^p\right]
		\le
		C_{p,T}
		\Big(
		1+ \mathbb{E}\big[\sup_{u\in[0,T]}|x_0(u)|^p\big]
		\Big)
		|t-s|^{p(\delta\wedge\theta)}.
		\vspace{-.2cm}
	\end{equation}
	And thus, \( t \mapsto X_t \) admits a H\"older continuous modification (still denoted \( X \) in lieu of \( \tilde{X} \) up to a \( \mathbb{P} \)-indistinguishability).
	More specifically, $X$ admits a version which is H\"older continuous on $[0,T]$ of any order \(a\in\Big(0,(\delta\wedge\widehat\theta\wedge\theta)-\frac1p\Big)\). Consequently,
	$\forall\, a\in(0,\delta\wedge\widehat\theta\wedge\theta)$,
	\vspace{-.2cm}
	\begin{equation}\label{eq:Holderpaths}
		\Big\|
		\sup_{0\le s<t\le T}
		\frac{|X_t-X_s|}{|t-s|^a}
		\Big\|_p^p
		=
		\mathbb{E}\Bigg[
		\sup_{s\neq t\in[0,T]}
		\frac{|X_t-X_s|^p}{|t-s|^{ap}}
		\Bigg]
		<
		C_{a,p,\beta, K, b,\sigma,T}
		\Big(
		1+ \mathbb{E}\big[\sup_{u\in[0,T]}|x_0(u)|^p\big]
		\Big),
		\vspace{-.1cm}
	\end{equation}
	for some positive constant
	$C_{a,p, \beta, K, b,\sigma,T}$ that only depends on $a$, $p$, $\beta$, $K$, $b$,$\sigma$ and $T$.
	In particular, it holds that
	\vspace{-.1cm}
	\begin{equation}\label{eq:supLpbound}
		\Big\|
		\sup_{t\in[0,T]}|X_t|
		\Big\|_p^p
		=
		\mathbb{E}\Big[
		\sup_{t\in[0,T]}|X_t|^p
		\Big]
		\le
		C'_{p,T}
		\Big(
		1+ \mathbb{E}\big[\sup_{u\in[0,T]}|x_0(u)|^p\big]
		\Big).
		\vspace{-.1cm}
	\end{equation}
	for some positive constant
	$C'_{p,T}=C_{a, p, \beta, K, b,\sigma,T}$.
\end{Theorem}
For clarity and conciseness, the proof of this Theorem is deferred to Appendix \ref{app:lemmata}, where the main technical results are presented.

\noindent{\bf Remarks :} 1. If $\ell_{b,\sigma}$ and $h_{b,\sigma}$ are defined on $\R_+$ and satisfy~\eqref{eq:Hol_Intergral_Cond} for every \(T > 0\), then the SVIE~\eqref{eq:Volterra} has a unique strong solution on \([0,\infty)\).\\
2. If $x_0:=X_0\phi$ as in~\eqref{eq:Volterrameanrevert}, where the function $\phi$ satisfies the H\"older continuity condition in~\ref{assum:CoefsVolterra}-$(iii)$, then the constraint $p \geq p_{eu}$ in Theorem~\ref{thm:ExistenceVolterraL2} can be relaxed to every \(p > 0\) such that \(\lVert X_0\rVert_p < +\infty\), thanks to~\cite[Lemma D.1]{JouPag22}; see also~\cite{GnabeyeuPages2026}. \\
3. This result can be viewed as a generalization, to convolution-type kernels, of the classical strong existence and uniqueness result for pathwise continuous solutions established in \cite[Theorem 1.1]{JouPag22} as (partial) extension of \cite[Theorem 3.1 and Theorem 3.3]{ZhangXi2010}, which holds only when the starting value \(X_0\) has finite polynomial moments of any order (the framework is more general and allows for any initial random function \(x_0\)). The proof is postponed to the appendix.

\bigskip
From now we focus on  the special case of a {\em scaled} stochastic Volterra equation~\eqref{eq:Volterrameanrevert} where $x_0\equiv X_0\phi$ and associated to a convolution kernel $K:\R_+\to \R_+$ satisfying~\eqref{eq:contKtilde} and ~\eqref{eq:Kcont} of Assumption \ref{assum:convolKernel}.
We consider a mean-reverting function $\mu :[0, T] \to \R$ to be a Borel funtion, in which case
 the drift $b(t,x) = \mu(t)-\lambda  x$ in~\eqref{eq:Volterra} is clearly Lipschitz continuous in $x$. We also assume that the diffusion $\sigma: \mathbb{T}\times \mathbb{R} \to  \R$ is multiplicatively
time-inhomogeneous, namely \(\sigma(t,x) = \varsigma(t)\,\sigma(x)\) where \(\sigma\) is {\em Lipschitz continuous} in $x$ and \(\varsigma\) as in~\eqref{eq:VolterraStabilizerCont}.


\medskip
\noindent
One checks that, when the deterministic initial profile \(\phi\) in~\eqref{eq:Volterrameanrevert}-\eqref{eq:Volterrameanrevert_} is given by~\eqref{eq:notcst_phi} and the initial random variable \(X_0\) is such that \(\E[X_0]=m_0\), then the Volterra SDE~\eqref{eq:Volterrameanrevert}-\eqref{eq:Volterrameanrevert_}  has constant first moment over time and is rewritten with \(\sigma(t,x) = \varsigma(t)\,\sigma(x)\) where \(\varsigma\) as in~\eqref{eq:VolterraStabilizerCont} as:
\begin{equation}\label{eq:CondMean}
	X_t = X_0 - \frac{1}{\lambda m_0}\Big(X_0 - m_0\Big) \int_0^t f_{\lambda}(t-s) \mu(s)\dd s +  \frac{1}{\lambda}\int_0^t f_{ \lambda}(t-s)\varsigma(s)\sigma( X_{s})dW_s, \quad X_0\perp\!\!\!\perp W.
\end{equation} 
 We shall therefore work with Equation~\eqref{eq:CondMean} as our baseline Volterra process, now starting from any initial random variable \(X_0 \in L^2(\P)\).

 \medskip
 In the following parts, 
 we derive conditions under which Equation~\eqref{eq:CondMean} admits a fake stationary regime in the sense of Definition~\ref{def:stationarity_order}. We further establish the existence of limiting distributions for the family of time-shifted processes in both the autonomous and the time-dependent diffusion settings, and characterize the mean and covariance functions of the canonical process under any such weak limit point.
 Be aware that, we do not have uniqueness of the limit distributions since they are
 not characterized by their mean and covariance functions, except in Gaussian setting.  
 
 \smallskip\noindent We denote by $\Omega_0:=C(\R_+, \R)$, the canonical space and $\stackrel{(C)_w}{\rightarrow }$ will stand for functional weak convergence on $C(\R_+,\R)$ equipped with the topology of uniform convergence  on compact sets.
\subsection{Finite-time intrinsic fake stationarity and long run functional weak behaviour}\label{subsec:fakestatioIntrin}
In this section, we assume a time homogeneous or autonomous volatility coefficient, i.e. \(\forall\, (t,x)\!\in \mathbb{T}\times \mathbb{R},\;\sigma(t,x) = \sigma(x)\)
and a Borel bounded mean-reverting function $\mu :[0, T] \to \R$.
Therefore, both the drift
\(b(t,x)=\mu(t)-\lambda x\) and \(\sigma(\cdot,\cdot)\) are  Lipschitz continuous in $x$ uniformly in time and in particular satisfy assumptions~\eqref{assum:CoefsVolterra}~$(i)$ and~$(ii)$. 
Hence, by Theorem~\eqref{thm:ExistenceVolterraL2} (see also \cite[Theorem~1.1]{JouPag22}, refining \cite[Theorems~3.1 and~3.3]{ZhangXi2010}), for every $T>0$, Equation~\eqref{eq:Volterra} admits a unique $(\mathcal{F}^{X_0,W}t){t\ge0}$-adapted continuous solution on $[0,T]$, starting from $X_0\in L^p(\mathbb P)$, $p\ge2$, independent of $W$, and satisfying~\eqref{eq:supLpbound}. Since $T>0$ is arbitrary, this yields a unique strong solution on $[0,\infty)$.
\begin{assumption}[On the mean-reversion function \(\mu\)]\label{ass:MeanRevert}  Let \(m_0\neq0\), $c\geq \frac{1}{\lambda^2} \int_0^\infty f_\lambda^2(s)\,ds$  and define \(h_{\lambda,c}(t):=\sqrt{1-\frac{1}{c\lambda^2}\int_0^t f_\lambda^2(s)\,ds}\). We assume that the mean-reverting function \(\mu\) is a solution to the Volterra convolution equation of the first kind
	\begin{equation}\label{eq:conv_mu_general} 
		(f_{\lambda} * \mu)(t) = \lambda m_0 \big( 1- h_{\lambda,c}(t) \big)=:g_{\lambda, c}(t), \qquad t \ge 0. 
	\end{equation}
\end{assumption}
Under assumption~\ref{ass:MeanRevert}, the Volterra SDE~\eqref{eq:CondMean}, starting from any initial random variable \(X_0 \in L^2(\P)\) can be rewritten as:
\begin{equation}\label{eq:ConstMeanVar}
	X_t = X_0 -\big(X_0 - m_0\big) 	\big(
	1-h_{\lambda,c}(t)\big) +  \frac{1}{\lambda}\int_0^t f_{\lambda}(t-s)\sigma(X_{s})dW_s.
\end{equation}
\begin{Example}[Exponential and Fractional kernels]\label{Ex:FakeStatAutonome_Lm}
	Let $\alpha \in (1, 2)$, $\rho\geq0$ and consider the {\em exponential-fractional or Gamma integration Kernel} $K_{\alpha,\rho}(t) = e^{-\rho t}\frac{u^{\alpha-1}}{\Gamma(\alpha)} \mbox{\bf 1}_{\R_+}(t)$, with
	$f_{\alpha,\rho, \lambda}$ defined by Equation~\eqref{eq:DerivSolventGammaKernel} of Example~\ref{Ex:SolventGammaKernel}. 
	Let $c\geq \frac{1}{\lambda^2} \int_0^\infty f_{\alpha,\rho, \lambda}^2(s)\,ds$ be given. Define \(g_{\alpha,\rho,\lambda,c}(t):=\lambda m_0
	\Big(1-\sqrt{1-\frac{1}{c \, \lambda^2 }\int_0^t f^2_{\alpha,\rho,\lambda}(s)\,ds}\Big),\) for any \(t \ge 0\) and consider the integral representation:	
	\begin{equation}\label{eq:mu_generalGamma}
		\forall\,t \ge 0\quad	\mu (t)
		:=
		g_{\alpha,\rho,\lambda,c}(t)
		-
		\frac{\Gamma(1+\alpha) \sin(\pi \alpha)}{\lambda \pi } \int_{0}^{t}  e^{-\rho (t-s)} (t-s)^{-\alpha-1} g_{\alpha,\rho,\lambda,c}(s) ds.
	\end{equation}
	Assumption~\ref{ass:MeanRevert} holds true in the sense that
	the mean-reverting function $\mu$ in Equation~\eqref{eq:conv_mu_general} is uniquely determined by Equation~\eqref{eq:mu_generalGamma}.
	In fact, taking Laplace transforms on both side of Equation~\eqref{eq:mu_generalGamma} yields
	\[
	L_\mu(s)=\bigl(1+L_{k_{\alpha,\rho,\lambda}}(s)\bigr)L_{g_{\alpha,\rho,\lambda,c}}(s),\quad\text{where}\quad k_{\alpha,\rho,\lambda}(t)=\frac{1}{\lambda\Gamma(-\alpha)}e^{-\rho t}t^{-\alpha-1},
	\quad t>0 .
	\]
	Moreover, one checks that \(L_{k_{\alpha,\rho,\lambda}}(s)=\frac{(s+\rho)^\alpha}{\lambda}.\)
	Hence \(L_\mu(s)=\frac{\lambda+(s+\rho)^\alpha}{\lambda}L_{g_{\alpha,\rho,\lambda,c}}(s).\) Now, we note that:
	\begin{equation}\label{eq:LpT_fGamma}
		L_{f_{\alpha, \rho, \lambda}}(s) = L_{-R'_{\alpha, \rho, \lambda}}(s) = -s L_{R_{\alpha, \rho, \lambda}}(s) + R_{\alpha, \rho, \lambda}(0) = \frac{-s}{s (1 + \lambda (s + \rho)^{-\alpha})} + 1
		= \frac{\lambda}{\lambda + (s + \rho)^{\alpha}}
	\end{equation}
	Consequently, \(L_{f_{\alpha, \rho, \lambda}}(s)\times L_\mu(s)=L_{g_{\alpha, \rho, \lambda,c}}(s)\).
	By uniqueness of Laplace transforms, we deduce that the representation ~\eqref{eq:mu_generalGamma} indeed solves the convolution equation~\eqref{eq:conv_mu_general}.
	Moreover, defining \(C_{\alpha,\lambda,c}:= \frac{m_0 \lambda}{2c\,(2\alpha-1)\,\Gamma(\alpha)^2}\) and using $1-\sqrt{1-x}\underset{0^+}{\sim}\frac{x}{2}$,  one checks that 
	$	g_{\alpha, \rho,\lambda,c}(t)
	\underset{0^+}{\sim} C_{\alpha,\lambda,c}	\, t^{2\alpha-1}$ so that by a dual version of the  Hardy-Littlewood Tauberian Theorem for Laplace transform (see for example \cite{FellerII} or \cite{BiGoTe1989}), at least heuristically, $\mu(t)
	\underset{0^+}{\sim} 
	\frac{C_{\alpha,\lambda,c}\,\Gamma(2\alpha)}{\lambda\,\Gamma(\alpha)}
	\, t^{\alpha-1}=
	\frac{m_0\,\Gamma(2\alpha)}{2c(2\alpha-1)\Gamma(\alpha)^3}
	\, t^{\alpha-1}.$ 
	One checks that $\mu(0)=0$  and that $\mu_\infty:=\lim_{t\to+\infty}\mu(t)
	=
	\lambda m_0 \big(1+\frac{\rho^\alpha}{\lambda}\big)
	\big(
	1-\sqrt{1-\frac{1}{c\,\lambda^2}
		\int_0^\infty f_{\alpha,\rho,\lambda}^2(s)\,ds}
	\big)$ in accordance with Assumption~\ref{ass:resolvent}~$(iii)$. It is straightforward that \(\mu\) is continuous on \(\mathbb{R}_+\).
\end{Example}
\begin{Theorem}[Time homogeneous settings]\label{Thm:Fake-funcWeak_Intrinsic}
	Let \( \lambda, c > 0 \), let  \( \sigma : \mathbb{R} \to \mathbb{R} \) be a Lipschitz continuous function. Assume Assumption \ref{ass:int_holregul} and Assumption~\ref{ass:on_stabilizer} on Equation \( (E_{\lambda}) \) and Assumption~\ref{ass:MeanRevert} are in force. Let \( (X_t)_{t \geq 0} \) be the solution to the scaled Stochastic Volterra Integral Equation ~\eqref{eq:ConstMeanVar} starting from  a random variable \( X_0 \).
	
	\medskip
	\noindent {\em (a) Finite-time intrinsic fake stationarity:}
	\noindent If \( \sigma \) is given by equation~\eqref{eq:sigmafakeI&II}, then
	the conditions of Equation~\eqref{eq:FakeStats} are satisfied for all \(t \geq 0\) only in the following two cases:
	
	\begin{itemize}
		\item [$(i)$]  If $\bar \sigma^2_0 =0$, then $\sigma\big( m_0\big)=0$ and $X_t =  m_0$ $\P$-$a.s.$, which represents a ({\em degenerate}) fake stationary regime (of type I).
	
		\item [$(ii)$]  If $\bar \sigma^2_0 >0$, set $ \displaystyle c:=\frac{v_0}{\bar \sigma_0^2}>0$ and assume that $c\geq \frac{1}{\lambda^2}
	\int_0^\infty f_\lambda^2(s)\,ds$,  
	   \(\sigma\) in \eqref{eq:sigmafakeI&II} with $\kappa_2 > 0$ is
	not degenerated (i.e. $\sigma^2\big( m_0\big)\neq0$) and let \( X_0 \in L^2(\mathbb{P}) \) such that \(\mathbb{E}[X_0] = m_0, \; \text{ and} \;
	\mathrm{Var}(X_0) = v_0= \frac{c\sigma^2(m_0)}{1 - c\kappa_2}\) : The solution \( (X_t)_{t \geq 0} \) to the Volterra equation~\eqref{eq:ConstMeanVar} has a fake stationary regime of type I, in the sense that
	\[
	\forall\, t\ge 0, \quad 
	\mathbb{E}[X_t] = m_0 
	, \; \text{Var}(X_t) = v_0= \frac{c\sigma^2(m_0)}{1 - c\kappa_2}\; \text{and}\;\; \mathbb{E}[\sigma^2(X_t)] = \bar{\sigma}_0^2 =\frac{\sigma^2(m_0)}{1 - c\kappa_2}.\]
   \end{itemize}
   Moreover if $m_0= (1-a) \frac{\mu_\infty}{\lambda}$ i.e. $\ell_\infty=0$, for any starting value \( X_0 \in L^2(\P) \),
   
   \centerline{$
   	\mathbb{E}[X_t] \to m_0, \quad \text{and} \quad \text{Var}(X_t) \to \frac{c\sigma^2(m_0)}{1 - c\kappa_2} \quad \text{as} \quad t \to +\infty.
   	$}
	\medskip
	\noindent {\em (b) C-tightness of time-shifted processes.} Assume $c> \frac{1}{\lambda^2}
	\int_0^\infty f_\lambda^2(s)\,ds$ and let
	\begin{equation}\label{eq:func_weak_intrin}
		X_0 \in L^p(\mathbb{P}) \quad \text{and} \quad
		\begin{cases}
			p = 2 \quad \text{and} \quad c < \frac{1}{[\sigma]_{\rm Lip}^2} \quad \text{if} \quad \vartheta \wedge \frac{\beta - 1}{2\beta} > \frac{1}{2}, \\
			p > \frac{1}{\vartheta \wedge \frac{\beta - 1}{2\beta}} \quad \text{and} \quad c < \frac{1}{4p [\sigma]_{\rm Lip}^2} \quad \text{if} \quad \vartheta \wedge \frac{\beta - 1}{2\beta} \leq \frac{1}{2}.
		\end{cases}
	\end{equation}
	Then,  as \( t \to +\infty \) the family of shifted processes \( (X_{t+u})_{u \geq 0} \) is C-tight and uniformly integrable, square uniformly integrable whenever \( p > 2 \). Consequently, for every weak limit \(P\) on
	 on the canonical space \( \Omega_0 \), the canonical process \( Y_t(\omega) = \omega(t) \) admits a \( \left( \vartheta \wedge \frac{\beta - 1}{2\beta} - \frac{1}{p} - \eta \right) \)-H\"older pathwise continuous \( P \)-modification  for every sufficiently small \( \eta > 0 \). Equivalently, there exists a subsequence \(t_n\uparrow\infty\) and a continuous process
	 \((X_u^\infty)_{u\ge0}\) such that
	 \[
	 (X_{t_n+u})_{u \geq 0} \Rightarrow (X^{\infty}_u)_{u \geq 0}
	 \; \text{weakly in } C(\mathbb{R}_+; \mathbb{R}) \text{ as } n \to \infty\; \text{and}\; \forall\, t\geq0,\; \mathbb{E}[X^{\infty}_t] =
	 \ell_\infty \mathbb{E}[X_0]  + m_0(1-\ell_\infty).\]
	\noindent {\em (c)  Functional weak long-run behavior.} Assume furthermore that the solution \((X_t)_{t \geq 0}\) of the volterra equation~\eqref{eq:ConstMeanVar} has a fake stationary regime of type I, starting from a random variable $X_0 \in L^2(P)$ with mean \( m_0\) and variance \(v_0\). Then, for \( t_1, t_2 \geq 0 \), \( t_1 \leq t_2 \),
	{\small
		\begin{equation}\label{eq:funclongRun_intrinsic}
			\text{Cov}(X_{t+t_1}, X_{t+t_2}) \overset{t \to +\infty}{\to} C_{f_{\lambda}}(t_1,t_2) := v_0 \Big( \ell_\infty^2  + \frac{ 1 }{c\lambda^2}\int_0^{+\infty}  f_{\lambda}(t_2-t_1+u)f_{\lambda}(u)du\Big).
		\end{equation}
	}
	Thus, under any functional limiting distribution \( P \), the canonical process \( Y \) is a (weak) \( L^2 \)-stationary process
	with mean \(  m_0 \) and autocovariance function \( C_{f_\lambda}(s,t) \), for \( s, t \geq 0 \).
	
	\medskip
	\noindent {\em (d) Stationary Gaussian Case.}  When \( \sigma(x)\equiv \sigma > 0 \) is constant (corresponding to the case \(\kappa_2=0\) in~\eqref{eq:sigmafakeI&II}, for which \(\kappa_1=0\) and $\sigma(x)=\sqrt{\kappa_0}$) and \(X_0\) has a Gaussian distribution, (say \( X_0 \sim \mathcal{N}\left( m_0, v_0\right)\) with \(v_0=c\sigma^2 \)), then the solution \( (X_t)_{t \geq 0} \)  to~\eqref{eq:ConstMeanVar} is a Gaussian process and \(X_t\sim X_0,\quad t\geq0,\)
	that is, \((X_t)_{t \geq 0}\) admits a fake stationary regime of type II, and satisfies \((X_{t+\cdot})\stackrel{(C)_w}{\longrightarrow} X^{(\infty)} \; \mbox{ as }t\to +\infty\)
	where $X^{(\infty)}$ 
	 is the 
	stationary Gaussian process with mean \(  m_0\)
	and covariance function \( C_{f_\lambda}(\cdot,\cdot) \).
\end{Theorem}
\noindent {\bf Proof:}
\smallskip
\noindent  {\sc Step~1.} {\em (a) (Finite-time intrinsic fake stationarity).} 
First, set \(\bar \sigma^2(t) = \E[\sigma^2(X_t)]\) for all \( t \geq 0 \) and 
suppose the conditions of Equation~\ref{eq:FakeStats} hold for all \( t \geq 0 \), using~\eqref{eq:ConstMeanVar}, one readily checks that  $\forall\, t>0,\; \E\, X_t = \E\, X_0=  m_0$ and
a straightforward computation shows that
\begin{equation}\label{eq:VolterraVar}
	\forall\, t\ge 0, \quad v_0\big(1-h_{\lambda,c}^2(t)\big) = \frac{1}{\lambda^2 }\int_0^t f_{\lambda}^2(t-s)\bar \sigma^2(s)ds.
\end{equation}
 Consequently, if \(\bar \sigma^2_0=0\) as $\bar \sigma^2(t) = \bar \sigma^2_0=0$ for all \( t \geq 0 \), we get $v_0=0$ since $ h_{\lambda,c}(t) \neq 1 $ (at least for infinitely many $t$). As a consequence, ${\rm Var}(X_t)=0$ for every $t\ge 0$. It follows that  $X_t= m_0$ $\P$-$a.s.$. As a result, $\sigma^2 \big( m_0 \big)= \bar \sigma^2_0 = 0$.

\noindent If $\bar \sigma^2_0>0$,  then, using the second-order Taylor Expansion of $\sigma^2$ in~\eqref{eq:sigmafakeI&II} around $m_0$, we have under the conditions~\eqref{eq:FakeStats}, $\bar \sigma^2_0 = \bar \sigma^2(t)= \sigma^2(m_0) + (\kappa_1+2\kappa_2)\times 0 + \kappa_2 v_0 $.  The above equation~\eqref{eq:VolterraVar} reads with \(c:=\frac{v_0}{\bar \sigma^2_0}\), \(v_0\big(1-h_{\lambda,c}^2(t)\big) = \big(\sigma^2(m_0) + \kappa_2 v_0 \big)  \frac{1}{c \lambda^2 }\int_0^t f_{\lambda}^2(s)ds\) i.e. \(v_0= \frac{c\sigma^2(m_0)}{1 - c\kappa_2}\) since  \(h_{\lambda,c}(t)< 1 \) for $t$ large enough, which is clearly solution to Equation~\eqref{eq:VolterraVar}.
Conversely one checks that this constant value for the variance solves Equation~\eqref{eq:VolterraVar}. To prove that it is the only one, we refer to~\cite[Proposition 3.12]{EGnabeyeu2025}.  If $\ell_\infty=0$ the confluence properties follows from \cite[Proposition 4.4]{EGnabeyeu2025}.

\smallskip
\noindent  {\sc Step~2.} {\em (b) (Kolmogorov criterion).} It follows from ~\eqref{eq:func_weak_intrin} that either \( p = 2 \) and \( c < \frac{1}{[\sigma]_{\rm Lip}^2} \), or \( p > 2 \) and \( c < \frac{1}{4\,p\,[\sigma]_{\rm Lip}^2} \). Hence, \cite[Proposition 4.3]{EGnabeyeu2025} implies that $ \sup_{t\ge 0}\Big\|  |X_t-m_0|\Big\|_p < +\infty$.
The function \( \sigma \) having at most affine growth, we derive that \( \sup_{t \geq 0} \Big\| |\sigma(X_t)| \Big\|_p < +\infty \).
Now, we can establish C-tightness by the Kolmogorov criterion. Let \( p\geq2 \) and $c$ satisfying constraint ~\eqref{eq:func_weak_intrin}. One writes  for $s,\,t\ge 0$ with $s\le t$ owing to equation ~\eqref{eq:ConstMeanVar}:
\begin{equation}\label{eq:Tight_intrin} X_t-X_s = \big(X_0 - m_0\big) \Big(h_{\lambda,c}(t)-h_{\lambda,c}(s) \Big) + \frac{1}{\lambda}\Big(J(t)-J(s)\Big),\; J(t):=\int_0^tf_{\lambda}(t-u)\sigma(X_{u})dW_u.
\end{equation}
On the one hand, as $c> \frac{1}{\lambda^2}
\int_0^\infty f_\lambda^2(s)\,ds$, \(h_{\lambda,c}(t)\geq h_{*}:=\sqrt{1-\frac{1}{c \, \lambda^2 }\int_0^\infty f_\lambda^2(u)\,du}
>0\) for every \(t\geq0\).
 One checks together with Assumption~\ref{ass:int_holregul} that,
 the first term on the right-hand side of~\eqref{eq:Tight_intrin} satisfies
\[
\Big\| (X_0 - m_0)\big(h_{\lambda,c}(t)-h_{\lambda,c}(s)\big) \Big\|_p
\leq\frac{\|X_0-m_0\|_p}{ 2 c \, \lambda^2 h_{*}} \Big(\int_0^{+\infty}f^{2\beta}_{\lambda}(u) du\Big)^{\frac{1}{\beta}} |t-s|^{\big(1-\frac{1}{\beta}\big)}
\leq
C_{X_0,m_0, p,\lambda,\beta,c,f_\lambda}
|t-s|^{\frac{\beta-1}{\beta} },
\]
where \(C_{X_0,m_0, p,\lambda,\beta,c,f_\lambda}<\infty\) by Equations~\eqref{eq:Integf}--\eqref{eq:Regulf} in Assumption~\ref{ass:int_holregul}.
On the other hand, combining the $L^p$-BDG inequality and the generalized Minkowski inequality, one derives from~\eqref{eq:func_weak_intrin} that,
\begin{align*}
	 \Big\| |J(t)&-J(s)| \Big\|_p \le \left\| |\int_s^{t}f_{\lambda}(t-u)\sigma(X_{u})\mathrm{d}W_u |\right\|_p + \left\| |\int_0^s \left( f_{\lambda}(t-u) - f_{\lambda}(s-u) \right)\sigma(X_{u})dW_u | \right\|_p\nonumber\\
	&\le
	C^{BDG}_p
	\Bigg[
	\Bigg(\int_s^t f_\lambda^2(t-u)\, \big\|\sigma(X_u)\big\|_p^2 \,du\Bigg)^{\frac12}
	+
	\Bigg(\int_0^s \big(f_\lambda(t-u)-f_\lambda(s-u)\big)^2
	\big\|\sigma(X_u)\big\|_p^2 \,du\Bigg)^{\frac12}
	\Bigg] \nonumber \\
	&\le C_p \left[\Big(\int_0^{+\infty}f^{2\beta}_{\lambda}(u) du\Big)^{\frac{1}{2\beta}} |t-s|^{\frac{\beta-1}{2\beta}} + \left(\int_0^{+\infty} \big(f_{\lambda}(t-s+u)-f_{\lambda}(u)\big)^2du\right)^{\frac12}\right] \\
	&\leq  C_{p, \sigma,  \varsigma, \lambda, f_{\lambda}, X_0} \cdot |t-s|^{ \vartheta\wedge\frac{\beta -1}{2\beta}} \quad \text{where we set} \quad C_p  \equiv C^{BDG}_p \sup_{u\ge 0}\Big\||\sigma(X_{u})|\Big\|_{p}.
\end{align*}
Finally, after noticing that \(\frac12-\frac{1}{2\beta} \leq \frac{\beta-1}{\beta}\) the previous estimates imply the existence of a  real constant \(C_{p,\lambda}>0\) such that
\[
\mathbb{E}\big[|X_t-X_s|^p\big]
\leq
C_{p,\lambda}\,
|t-s|^{p\left(\vartheta\wedge \frac{\beta-1}{2\beta}\right)}.
\]	
As $p(\vartheta\wedge \frac{\beta-1}{2\beta})>1$, it follows from  Kolmogorov's $C$-tightness criterion (see~\cite[Theorem~2.1]{RevuzYor}), that the family of  shifted processes $X_{t+\cdot}$, $t\ge 0$, is $C$-tight.

\smallskip
\noindent  {\sc Step~3.} {\em (c) Asymptotic $L^2$-stationarity}:
Let us consider the asymptotic  covariance between $X_{t+t_1}$ and $X_{t+t_2}$, $0<t_1<t_2$ when $X_t$ starts for $X_0$ with mean $m_0$,  variance $v_0$ and $\bar \sigma^2 = \E\, \sigma(X_t)^2$, $t\ge 0$.
\begin{align*}
	{\rm Cov}(X_{t+t_1}, X_{t+t_2})& = {\rm Var}(X_0)h_{\lambda,c}(t+t_1)h_{\lambda,c}(t+t_2)+ \frac{1}{\lambda^2}\int_0^{t+t_1\wedge t_2} f_{\lambda}(t+t_2-s)f_{\lambda}(t+t_1-s)\E \,\sigma^2(X_s)ds\\
	&=  {\rm Var}(X_0)h_{\lambda,c}(t+t_1)h_{\lambda,c}(t+t_2)+ \frac{\bar \sigma^2}{\lambda^2}\int_0^{t+t_1} f_{\lambda}(t_2-t_1+u)f_{\lambda}(u)du\\
	&\stackrel{t\to+\infty}{\longrightarrow} v_0\ell_\infty^2 + \frac{v_0}{c\lambda^2}\int_0^{+\infty}  f_{\lambda}(t_2-t_1+u)f_{\lambda}(u)du
\end{align*}
where we also used that $h_{\lambda,c}(t)\to \ell_\infty$ as $t\to +\infty$ owing to Equation~\eqref{eq:mu_generalGamma} and the definition of \(\ell_\infty\).
 
\medskip
\noindent{\sc Step 4}. {\em (d). The case $\sigma(x)=\sigma$ } is obvious once noted that the solution $(X_t)_{t\geq0}$ to~\eqref{eq:ConstMeanVar} is a Gaussian process (see also the proof of Theorem~\ref{Thm:funcWeak} further on). One easily verifies that \(\overline{\sigma}_0^2=\sigma^2,\)
and that $(X_t)_{t\geq0}$ in~\eqref{eq:ConstMeanVar}
with \(c=\frac{v_0}{\sigma^2}\) has constant mean \(m_0\) and variance
\(v_0=c\sigma^2\) if \(X_0\sim\mathcal{N}\left(m_0,v_0\right)\).
		  The proof is completed. \hfill $\square$

\medskip
\noindent 	
The preceding discussion highlights that imposing both constant mean and constant variance in Equation~\eqref{eq:Volterrameanrevert}-\eqref{eq:Volterrameanrevert_} is highly restrictive, as it effectively dictates explicit closed-form structures for both the initial function $\phi$ and the mean-reversion term $\mu$ in the drift. To alleviate this rigidity, and following the approach of~\cite{Pages2024, EGnabeyeu2025}, we introduce a deterministic stabilizing factor in the diffusion coefficient while keeping $\mu$ fully flexible. Under this formulation, we adopt Equation~\eqref{eq:CondMean} as our baseline Volterra SDE with constant mean.
\subsection{Fake stationarity with a stabilizer: Finite-time and asymptotics}\label{subsec:fakestatioStabil}
In this section, we consider the volatility coefficient $\sigma(t,x)$ to be time-dependent or inhomogeneous, admitting the separable factorization: 
\begin{equation}
	\forall (t,x) \in \mathbb{T} \times \mathbb{R}, \qquad \sigma(t,x) = \varsigma(t)\,\sigma(x), \quad \varsigma(t):=\varsigma_{\lambda,c}(t), \sigma(x) > 0.
\end{equation}
where \(\varsigma\) is a (locally) integrable Borel function to be specified later.  
Assume $\exists \beta>1$ such that \(\varsigma \in {\cal L}^{\frac{2\beta}{\beta-1}}_{loc}(\R_+)\) (see e.g.~\eqref{eq:varsigma} for any $\beta>\frac{1}{3-2\alpha}$, $\alpha\in(1,\frac32)$) and the mean-reverting function \( \mu : \mathbb{R}_+ \to \mathbb{R} \)  belongs to either \(\mathcal L^\infty(\mathbb R_+)\) or \( {\cal L}^{\frac{2\beta}{\beta-1}}_{loc}(\R_+)\). Therefore, both the drift
 \(b(t,x)=\mu(t)-\lambda x\) and \(\sigma(\cdot,\cdot)\) satisfy assumptions~\eqref{assum:CoefsVolterra}~$(i)$ and~$(ii)$. 
Then, Equation~\eqref{eq:Volterrameanrevert} has a unique \((\mathcal{F}^{X_0, W}_t)_{t>0}\)-adapted pathwise continuous solution \( (X_t)_{t\geq0} \) starting from \( X_0 \in L^p(\mathbb{P}), p\geq2 \) independent of \(W\) and satisfying~\eqref{eq:supLpbound} for every $T>0$.
This follows by applying the existence Theorem~\eqref{thm:ExistenceVolterraL2} above
to each time interval \( [0,T] \), \( T \in \mathbb{N} \).  Since $T\geq0$ is arbitrary, we conclude that \eqref{eq:Volterra} admits a unique strong solution on $[0,\infty)$.

\smallskip
\noindent{\bf Remark: } In this particular case of a multiplicatively separable diffusion coefficient, an alternative well-posedness can be obtained from \cite[Theorem 1.1]{JouPag22} , since the transformed kernels \(K_1(t,s):=K(t-s)\) and \(K_2(t,s):=K(t-s)\varsigma(s)\) for every \(0\leq s\leq t \leq T\) then satisfy the assumptions of that result.

	\begin{Theorem}[Time-dependent or inhomogeneous diffusion coefficient $\sigma$]\label{Thm:funcWeak}
		Let \( \lambda, c > 0 \), let  \( \mu : \mathbb{R}_+ \to \mathbb{R} \)  be a Borel bounded function 
		 and let \( \sigma : \mathbb{R} \to \mathbb{R} \) be a Lipschitz continuous function. Assume Assumptions~\ref{ass:resolvent}, ~\ref{ass:on_stabilizer} on Equation \( (E_{\lambda}) \) and~\ref{ass:int_holregul} are in force. Let \( (X_t)_{t \geq 0} \) be the solution to the scaled Stochastic Volterra Integral Equation ~\eqref{eq:CondMean} starting from an initial random variable \( X_0 \).
		
		\medskip
		\noindent {\em (a) Finite-time fake stationarity with a stabilizer.} 
		If \(\sigma\) is given by \eqref{eq:sigmafakeI&II} with $\kappa_2 > 0$ and is
		not degenerated (i.e. $\sigma\big( m_0\big)\neq0$), then let $\varsigma=\varsigma_{\lambda,c}$,
		assumed 
		continuous solution to~\eqref{eq:VolterraStabilizer2} for some $c\in\left(0,\frac{1}{\kappa_2}\right)$ and  
		let  \( X_0 \in L^2(\mathbb{P}) \) with \(\mathbb{E}[X_0] = m_0, \; \text{ and} \;
		\mathrm{Var}(X_0) = v_0= \frac{c\sigma^2(m_0)}{1 - c\kappa_2}\).  The solution \( (X_t)_{t \geq 0} \) to~\eqref{eq:CondMean} has a fake stationary regime of type I, in the sense that
		\[
		\forall\, t\ge 0, \quad 
		\mathbb{E}[X_t] = m_0
		, \; \text{Var}(X_t) = v_0= \frac{c\sigma^2(m_0)}{1 - c\kappa_2}\; \text{and}\;\; \mathbb{E}[\sigma^2(X_t)] = \bar{\sigma}_0^2 =\frac{\sigma^2(m_0)}{1 - c\kappa_2}.\]
		Moreover if $m_0= (1-a) \frac{\mu_\infty}{\lambda}$ i.e. $\ell_\infty=0$, for any starting value \( X_0 \in L^2(\P) \),
		
		\centerline{$
		\mathbb{E}[X_t] \to m_0, \quad \text{and} \quad \text{Var}(X_t) \to \frac{c\sigma^2(m_0)}{1 - c\kappa_2} \quad \text{as} \quad t \to +\infty.
		$}
		\medskip
		\noindent {\em (b) C-tightness of time-shifted processes.} Assume
		\begin{equation}\label{eq:func_weak}
			X_0 \in L^p(\mathbb{P}) \; \text{and} \;
			\begin{cases}
				p = 2 \; \text{and} \; c < \frac{1}{[\sigma]_{\rm Lip}^2} \quad \text{if} \quad \vartheta^* > \frac{1}{2}, \\
				p > \frac{1}{\vartheta^*} \; \text{and} \; c < \frac{1}{4p [\sigma]_{\rm Lip}^2} \; \text{if} \; \vartheta^* \leq \frac{1}{2},
			\end{cases}\hspace{-.3cm}\quad \text{where}\quad \vartheta^*:= \frac12  (\vartheta \wedge (1-\frac{1}{2\beta})).
		\end{equation}
		Then, the family of shifted processes \( (X_{t+u})_{u \geq 0} \) is C-tight and uniformly integrable, square uniformly integrable for \( p > 2 \) as \( t \to +\infty \). For any limiting distribution \( P \) on \( \Omega_0 := C(\mathbb{R}_+, \mathbb{R}) \), the canonical process \( Y_t(\omega) = \omega(t) \) has a \( \left( \vartheta^* - \frac{1}{p} - \eta \right) \)-H\"older pathwise continuous \( P \)-modification for sufficiently small \( \eta > 0 \). 
		That is, along a sequence \( u_k \uparrow \infty \), there exists a process \( X^{\infty} \) with continuous sample paths such that 
		\[
		(X_{t+u_k})_{t \geq 0} \Rightarrow (X^{\infty}_t)_{t \geq 0}
		\quad \text{weakly in } C(\mathbb{R}_+; \mathbb{R}) \text{ as } u_k \to \infty.\]
		\noindent Any limiting process \( X^{\infty} \) satisfies \( \forall t \ge 0, \quad X^{\infty}_t \in L^p(\P) \)  for each $p \geq 2$ and its first moment is given by
		\[\mathbb{E}[X^{\infty}_t] =
		\ell_\infty \mathbb{E}[X_0]  + (1-a) \frac{\mu_\infty}{\lambda}.\]
		\noindent Moreover, if \(\ell_\infty=0\), the shifted processes of two solutions \( (X_t)_{t \geq 0} \) and \( (X_t')_{t \geq 0} \) are \( L^p \)-confluent, i.e. there exists a non-increasing function \( \bar{\varphi}_{\infty,p} : \mathbb{R}_+ \to [0, \frac{1}{(1-\sqrt{4pck})^2}] \) with \( \lim_{t \to +\infty} \bar{\varphi}_{\infty}(t) = 0 \) and
		\[W_p\left( \left[ (X_{t+t_1}, \ldots, X_{t+t_N}) \right], \left[ (X'_{t+t_1}, \ldots, X'_{t+t_N}) \right] \right) \leq \sqrt{N} \bar \varphi_{\infty,p} (t) {\cal W}_p\big([X_0],[X'_0]\big)\to 0 \quad \text{as} \quad t \to +\infty.\]
		Hence, the functional weak limiting distributions of \( [X_{t+\cdot}] \) and \( [X'_{t+\cdot}] \) coincide, meaning that if \( [X_{t_n+\cdot}] \stackrel{(C)_w}{\rightarrow} P \) for some subsequence \( t_n \to +\infty \), then \( [X'_{t_n+\cdot}] \stackrel{(C)_w}{\rightarrow} P \) and vice versa.
		
		\medskip
		\noindent {\em (c)  Functional weak long-run behavior.} Assume furthermore that the solution \((X_t)_{t \geq 0}\) of the volterra equation ~\eqref{eq:CondMean} has a fake stationary regime of type I, starting from a random variable $X_0 \in L^2(P)$ with mean \( m_0\) and variance \(v_0\).Then for any limiting distribution \(X^{\infty}\),  \( \mathbb{E}[X^{\infty}_t] =  m_0
		\) 
		while its autocovariance function is, for \( t_1, t_2 \geq 0 \), \( t_1 \leq t_2 \), given by \(\text{Cov}(X_{t_1}^\infty, X_{t_2}^\infty)=C_{f_{\lambda}}(t_1,t_2)\)
		{\small
			\begin{equation}\label{eq:funclongRun}
				\text{Cov}(X_{t+t_1}, X_{t+t_2}) \overset{t \to +\infty}{\to} C_{f_{\lambda}}(t_1,t_2) := v_0 \Big(\ell_\infty^2  + \frac{ (1-\ell_\infty^2)}{\int_0^{+\infty}f^2_{\lambda}(s)ds}\int_0^{+\infty}  f_{\lambda}(t_2-t_1+u)f_{\lambda}(u)du\Big).
			\end{equation}
		}
		Thus, under any functional limiting distribution \( P \), the canonical process \( Y \) is a (weak) \( L^2 \)-stationary process
		with mean \(  m_0 \) and covariance function \( C_{f_\lambda}(s,t) \), for \( s, t \geq 0 \).
		
		\medskip
		\noindent {\em (d) Stationary Gaussian:}  If \( \sigma(x) = \sigma > 0 \) is constant (e.g. \(\kappa_2=0\) in~\eqref{eq:sigmafakeI&II}, for which \(\kappa_1=0\) and $\sigma(x)=\sqrt{\kappa_0}$) and if \(X_0\sim\nu:=\mathcal N\left( m_0, v_0\right)\) with \(v_0=c\sigma^2 \), 
		then \((X_t)_{t\ge0}\) is a Gaussian process with a fake stationary regime of type II 
		where \(\nu\) denotes its one-dimensional marginal distribution. Moreover, its satisfies
		\(
		X_{t+\cdot} \stackrel{(C)_w}{\longrightarrow} \mathcal{GP}(f_{\lambda}) \quad \text{as} \quad t \to +\infty,\)
		where \( \mathcal{GP}(f_\lambda) \) is the 
		stationary Gaussian process with mean \(  m_0\)
		and covariance function \( C_{f_\lambda}(\cdot,\cdot) \).
	\end{Theorem}
		\noindent {\bf Proof of Theorem \ref{Thm:funcWeak}.}
		Owing to~\cite[Theorem 3.6]{EGnabeyeu2025}, the {\em functional equation } \(\textit{($E_{\lambda, c}$)}\) in~\eqref{eq:VolterraStabilizer}-\eqref{eq:notcst_phi} is a necessary condition for the fake stationary identities~\eqref{eq:FakeStats} to be satisfied.
		 
		\smallskip
		\noindent  {\sc Step~1.} The claim in {\em (a)} follows from~\cite[Proposition 3.14.]{EGnabeyeu2025} or by mimicking the proof of  Theorem~\ref{Thm:Fake-funcWeak_Intrinsic}~{\em (a)}-$(ii)$ where  the autonomous $\sigma (x)$ is replaced by $\varsigma(t)\sigma(x)$, together with the identity~\eqref{eq:VolterraStabilizer2}.
		
	\smallskip
	\noindent  {\sc Step~2.} {\em (b) (Kolmogorov criterion).}
	By~\eqref{eq:func_weak}, either \(p=2\) with
	\(c<\frac{1}{[\sigma]_{\rm Lip}^2}\), or \(p>2\) with
	\( c < \frac{1}{4\,p\,[\sigma]_{\rm Lip}^2} \). 
	 Hence,
	\cite[Proposition~4.3]{EGnabeyeu2025} yields $ \sup_{t\ge 0}\Big\| |X_t-m_0|\Big\|_p < +\infty$. The affine growth of \(\sigma\) then immediately gives \( \sup_{t \geq 0} \Big\| |\sigma(X_t)| \Big\|_p < +\infty \). 
	Now, we can establish a uniform bound on H\"older increments.  
	Let \( p\geq2 \) and $c$ satisfying constraint ~\eqref{eq:func_weak}. One writes  for $s,\,t\ge 0$ with $s\le t$ and owing to equation~\eqref{eq:CondMean}:
	\begin{equation}\label{eq:Tight} X_t-X_s = \frac{1}{\lambda m_0}\Big(X_0 - m_0\Big) \left( \int_0^s f_{\lambda}(s-u) \mu(u)du -\int_0^t f_{\lambda}(t-u) \mu(u)du \right) + \frac{1}{\lambda}\Big(J(t)-J(s)\Big).
	\end{equation}
	where we set $J(t):=\int_0^tf_{\lambda}(t-u)\varsigma(u)\sigma(X_{u})dW_u$. We also define $ I(t) :=  \frac{(X_0 - m_0)}{\lambda m_0}\int_0^t f_{\lambda}(t-u) \mu(u)  \, du$. As \(\mu \in {\cal L}^\infty(\R_+)\) i.e. \(\mu\) bounded, one gets using the triangle inequalities and \(h_\lambda\) in~\eqref{eq:Regulf}:
    	\begin{align*}
    	\Big\| |I(t)-I(s)| \Big\|_p
    	&\le \frac{\|X_0-m_0\big\|_p}{\lambda m_0}\sup_{u\geq0} |\mu(u)|\times \Big( \int_{s}^{t} |f_{\lambda}(t-u)|\, du + \int_{0}^{s} |\left( f_{\lambda}(t-u) - f_{\lambda}(s-u) \right)|\,du  \Big) \\
    	&\le \frac{\|\mu\|_{\infty}}{\lambda m_0} \|X_0-m_0\big\|_p \times \Big( \Big(\int_0^{+\infty} f_{\lambda}^{2\beta}(u)du \Big)^{\frac{1}{2\beta}}  (t-s)^{1-\frac{1}{2\beta}} + |t-s|^\vartheta \int_0^s h_{\lambda}(u)\,du \Big)\\
    	&\le C_{X_0,p,\lambda,\beta,\mu,f_\lambda}|t-s|^{ (\vartheta\wedge (1-\frac{1}{2\beta}))},\quad \text{since} \quad \int_0^\infty h_{\lambda}(u)\,du<\infty,\quad (h_{\lambda} \in {\cal L}^\infty(\R_+)).
    \end{align*}
    The constant \(C_{X_0,p,\lambda,\beta,\mu,f_\lambda}\) is finite thanks to Equations~\eqref{eq:Integf}--\eqref{eq:Regulf} from Assumption~\ref{ass:int_holregul}.
	On the other hand, combining the $L^p$-BDG inequality and the generalized Minkowski inequality, one derives from~\eqref{eq:func_weak} that,
	\begin{align}
		&\ \Big\| |J(t)-J(s)| \Big\|_p \le \left\| |\int_s^{t}f_{\lambda}(t-u)\varsigma(u)\sigma(X_{u})\mathrm{d}W_u |\right\|_p + \left\| |\int_0^s \left( f_{\lambda}(t-u) - f_{\lambda}(s-u) \right)\varsigma(u)\sigma(X_{u})dW_u | \right\|_p\nonumber\\
		&\hspace{.9cm}\le
		C^{BDG}_p
		\Big[
		\Big(\int_s^t f_\lambda^2(t-u)\, \big\|\sigma(X_u)\big\|_p^2 \varsigma^2(u)\,du\Big)^{\frac12}
		+
		\Big(\int_0^s \big(f_\lambda(t-u)-f_\lambda(s-u)\big)^2
		\big\|\sigma(X_u)\big\|_p^2 \varsigma^2(u)\,du\Big)^{\frac12}
		\Big] \nonumber \\
		&\hspace{1cm}\le
		C^{BDG}_p 
		\sup_{u\ge0}\big\|\sigma(X_u)\big\|_p
		\left[
		\left(\int_s^t f_\lambda^2(t-u)\varsigma^2(u)\,du\right)^{\frac12}
		+
		\left(|t-s|^{2\vartheta}\int_0^{s} h_{\lambda}^2(s-u)\varsigma^2(u)\,du\right)^{\frac12}
		\right].\nonumber
	\end{align}	
	Now, set $ I(t) :=  \frac{1}{\lambda m_0}\int_0^t f_{\lambda}(t-u) \mu(u)  \, du$ and note using~\eqref{eq:func_weak} and the condition on \(h_\lambda\) in~\eqref{eq:Regulf} that:
	\begin{align*} \int_s^t f_\lambda^2(t-u)\varsigma^2(u)\,du&=\int_0^t f_\lambda^2(t-u)\varsigma^2(u)\,du-\int_0^s f_\lambda^2(s-u)\varsigma^2(u)\,du + \int_0^s \big(f_\lambda^2(s-u)-f_\lambda^2(t-u)\big)\varsigma^2(u)\,du\\
		&\leq 2 c \lambda^2 \Big( 1+ \frac{\|\mu\|_{\infty}}{\lambda m_0} \|f_\lambda\|_{{\cal L}^1(\R_+)} \Big) \Big|I(t)-I(s)\Big| + 2 \|f_\lambda\|_\infty |t-s|^{\vartheta} \int_0^{s} h_{\lambda}(s-u)\varsigma^2(u)\,du\\
		&\leq C_{p, \mu,\varsigma, f_{\lambda}}|t-s|^{ (\vartheta\wedge (1-\frac{1}{2\beta}))},\quad \text{since} \quad \sup_{s\ge 0}\int_0^s h_{\lambda}(s-u)\varsigma^2(u)\,du<\infty.
	\end{align*}
	where the penultimate inequality follows from Equation~\eqref{eq:VolterraStabilizer2}, together with Assumption~\ref{ass:int_holregul}. Finally,
	\begin{align*}
		\Big\| J(t)-J(s) \Big\|_p &\leq C_p \sup_{u\ge 0}\Big\||\sigma(X_{u})|\Big\|_{p} \left[C_{p,\mu,\varsigma,f_\lambda}
		|t-s|^{\frac12\big(\vartheta\wedge (1-\frac1{2\beta})\big)} + \Big(\|h_{\lambda}\|_\infty|t-s|^{2\vartheta}\int_0^{s} h_{\lambda}(s-u)\varsigma^2(u)\,du\Big)^{\frac12}\right] \\
		&\le
		C_{p,\mu,\varsigma, \sigma,\lambda,f_\lambda}
		\left(
		|t-s|^{\left(\frac{\vartheta}{2}\right)\wedge \left(\frac12-\frac1{4\beta}\right)}
		+
		|t-s|^{\vartheta}
		\right).
	\end{align*}
	Finally, after noticing that \(\frac12-\frac1{4\beta}\leq 1-\frac{1}{2\beta},\) what precedes proves the existence of a  real constant \(C_{p,\lambda}>0\) such that
	\[
	\mathbb{E}\big[|X_t-X_s|^p\big]
	\leq
	C_{p,\lambda}\,
	|t-s|^{\frac{p}{2}\left(\vartheta\wedge (1-\frac{1}{2\beta})\right)}.
	\]
	\noindent {\em Conclusion of this step:}
	Putting all these estimates together, and recalling the definition of \(\vartheta^*\) from Equation~\eqref{eq:func_weak}, define, for every \(u\geq 0\), the process \(X^u\) by \( X^u_t = X_{t+u} \), where \( t \geq 0 \). Then \( X^u \) has continuous sample paths and satisfies
	\[	\sup_{u \geq 0} \mathbb{E}[|X^u_t - X^u_s|^p] \leq C(p)|t-s|^{p\vartheta^*}
	\quad \text{for } 0 \leq t - s \leq 1.\]
	As $p\vartheta^*>1$ according to equation \eqref{eq:func_weak}, it follows from  Kolmogorov's $C$-tightness criterion (see ~\cite[Theorem 2.1, p. 26, 3rd edition]{RevuzYor} 
	or \cite[Lemma 44.4, Section IV.44, p.100]{RogersWilliamsII}), that the family of  shifted processes $X_{t+\cdot}$, $t\ge 0$, is $C$-tight i.e. \( (X^u)_{u \geq 0} \) is tight on \( C(\mathbb{R}_+; \mathbb{R}) \) (hence the existence of a weak
	continuous accumulation point thanks to Prokhorov's theorem) with limiting distributions P under which the canonical process has the announced H\"older pathwise regularity. Therefore, we conclude that along a sequence \( u_k \uparrow \infty \), the process \( X^{u_k} \) converges in law to some continuous process \( X^{\infty} \).
	
	An application of Fatou's lemma shows that any limiting process (resp. the limit distribution) has a finite moment of any order, i.e., $\quad 
	\forall t>0, \quad \mathbb{E}[|X^{\infty}_t|^p] \leq \sup_{u \geq 0} \mathbb{E}[|X_u|^p] < \infty.$
	
	For the first moment formula, we note using ~\cite[Equation~$(3.19)$ and Lemma 3.5]{EGnabeyeu2025} that
	
	\[
	\mathbb{E}[X_t] \longrightarrow \ell_\infty \mathbb{E}[X_0]  + (1-a) \frac{\mu_\infty}{\lambda} \quad \text{as } t \to \infty.
	\]
	Since \( \sup_{t \geq 0} \mathbb{E}[|X_t|^2] < \infty \), we easily conclude that $\quad 
	\lim_{t \to \infty} \mathbb{E}[X_t] = \mathbb{E}[X^{\infty}_t].$
	
	\bigskip
	\noindent{\sc Step 3}. {\em (c) Asymptotic weak $L^2$-stationarity. }
	Now let us consider the asymptotic  covariance between $X_{t+t_1}$ and $X_{t+t_2}$, $0<t_1<t_2$ when $X_t$ starts for $X_0$ with mean  \( m_0\), variance $v_0$ and $\bar \sigma^2 = \E\, \sigma(X_{t})^2$, $t\ge 0$ constant over time. Let \(\ell_\lambda(t):=\big(1-\frac{(f_\lambda * \mu)(t)}{\lambda m_0}\big)\) for short.
	Using $\text{Cov}(aU + b, cV + d) = ac \, \text{Cov}(U, V)$ and equation ~\eqref{eq:CondMean}, we have:
		\begin{align*}
			{\rm Cov}(X_{t+t_1}, X_{t+t_2}) &= {\rm Var}(X_0) \ell_\lambda(t+t_1)\ell_\lambda(t+t_2) + \E\left[\int_0^{t+t_1} f_{\lambda}(t+t_2-s)f_{\lambda}(t+t_1-s)\frac{\varsigma^2(s)}{\lambda^2}\sigma^2(X_{s})ds\right]\\
			&=  {\rm Var}(X_0)\ell_\lambda(t+t_1)\ell_\lambda(t+t_2) + \frac{\bar \sigma^2}{\lambda^2}\int_0^{t+t_1} f_{\lambda}(t_2-t_1+u)f_{\lambda}(u)\varsigma^2(t+t_1-u) du.
		\end{align*}
	As $ f_{\lambda}(t_2-t_1+\cdot)f_{\lambda}\!\in {\cal L}^2(\R_+)$ since $f_{\lambda}\!\in {\cal L}^2(\R_+)$ , $\mbox{\bf 1}_{\{0\le u \le t+t_1\}}\varsigma^2(t+t_1-u)\to  \frac{c\lambda^2 (1-\ell_\infty^2)}{\int_0^{+\infty}f^2_{\lambda}(s)ds}$ for every $u\!\in \R_+$ as $t\to +\infty$ (owing to~\cite[Lemma 3.9]{EGnabeyeu2025}) and $\lim_{t\to+\infty}\ell_\lambda(t)= \ell_\infty$, we have:
	\[{\rm Cov}(X_{t+t_1}, X_{t+t_2})\stackrel{t\to+\infty}{\longrightarrow} \ell_\infty^2 {\rm Var}(X_0) + \frac{c\bar \sigma^2 (1-\ell_\infty^2)}{\int_0^{+\infty}f^2_{\lambda}(s)ds}\int_0^{+\infty}  f_{\lambda}(t_2-t_1+u)f_{\lambda}(u)du.\]
	Using \(v_0 = c \bar \sigma^2\), we obtain \(C_{f_{\lambda}}(t_1,t_2)\) in Equation~\eqref{eq:funclongRun}.
	The confluence result follows from \cite[Proposition 4.4]{EGnabeyeu2025} with \( \bar \varphi_{\infty,p} (t) = \sup_{u \geq t} \varphi_{\infty,p}^{1/2} (u) \).
	In fact, let $X$ and $X'$ be two solutions of  Equation~\eqref{eq:CondMean} starting from $X_0$ and $X'_0$ respectively, both square integrable. 
	By 
	 \cite[Proposition 4.4]{EGnabeyeu2025}, we derive that for every $0\le t_1<t_2< \cdots < t_{_N}<+\infty$
	\[{\cal W}_p\big([(X_{t+t_1}, \cdots, X_{t+t_{_N}})], [(X'_{t+t_1}, \cdots, X'_{t+t_{_N}})])\leq \sqrt{N} \bar \varphi_{\infty,p} (t) {\cal W}_p\big([X_0],[X'_0]\big)\to 0 \mbox{ as }t\to +\infty.\]	
	\noindent As  a consequence, the  weak limiting distributions  of $[X_{t+\cdot}]$ and $[X'_{t+\cdot}]$ are the same in the sense that,  if $[X_{t_n+\cdot}]\stackrel{(C)_w}{\longrightarrow} P$ for some subsequence $t_n \to +\infty$ (where $P$ is a probability measure on $C(\R_+, \R)$ equipped with the Borel $\sigma$-field induced by the sup-norm topology), then  $[X'_{t_n+\cdot}]\stackrel{(C)_w}{\longrightarrow} P$ and  conversely.
	
	\bigskip
	\noindent{\sc Step 4}. {\em (d) Stationary Gausian case.} This result stems first from the fact that \( (X_t)_{t \geq 0} \) is a Gaussian process, implying that its limiting distributions in the functional weak sense are also Gaussian. Secondly, a Gaussian process is completely characterized by its mean and covariance functions.\\
	In fact, when \( \sigma(x) = \sigma > 0 \quad \forall x \in \mathbb{R}\) and \( X_0 \) follows a Gaussian distribution, the process \( X \) is Gaussian, which implies (at least for finite-dimensional weak convergence, i.e., weak convergence of all marginals of any order) that,$
	(X_{t + \cdot}) \stackrel{(C)_w}{\longrightarrow} \mathcal{GP}(f_{\lambda}) \quad \text{as} \quad t \to +\infty,
	$
	where \( \mathcal{GP}(f_{\lambda}) \) is a Gaussian process with mean \( m_0 \) and covariance function given above. \hfill $\Box$
	
	\medskip
	\noindent\textbf{Remark.} The theorem remains valid for
	\(\mu\in\mathcal{L}^1(\mathbb{R}_+)\) with \(\vartheta^*=\vartheta\),
	provided that a solution to the Volterra
	SDEs~\eqref{eq:Volterrameanrevert}--\eqref{eq:Volterrameanrevert_}
	exists. This is indeed guaranteed by
	Theorem~\eqref{thm:ExistenceVolterraL2} if, in addition, \( \mu\in {\cal L}^{\frac{2\beta}{\beta-1}}_{loc}(\R_+)\) for some \(\beta>1\). In this case,
	Remark~4 following Definition~\ref{def:stabilizer} yields, \(\int_0^{+\infty} \varsigma^2(u)\,du<\infty\).
		\section{Applications to Fractional Stochastic Volterra Integral equations}\label{sec:appl2}
	Let consider the {\em fractional integration kernel} of Example~\ref{Ex:SolventGammaKernel}~$2$, where $\alpha = H + \frac{1}{2}$, with $H$ denoting the Hurst coefficient:
	\begin{equation}\label{eq:frackernel}
		K(t) = K_{\alpha}(t) = \frac{u^{\alpha-1}}{\Gamma(\alpha)} \mbox{\bf 1}_{\R_+}(t),  \quad \alpha>0.
	\end{equation}
	From Equation~\eqref{eq:SolventFracKernel}, we have that its \(\lambda-\)resolvent is given by \(	R_{\alpha,\lambda}(t) = \sum_{k\ge 0} (-1)^k \frac{\lambda ^k t^{\alpha k}}{\Gamma(\alpha k+1)}= E_{\alpha}(-\lambda t^{\alpha} )= e_{\alpha}(\lambda^{1/\alpha}t) \; t\ge 0,\)
	where $E_{\alpha}$ denotes the standard   Mittag-Leffler function and  $e_{\alpha}$, the alternate Mittag-Leffler function.
	\begin{equation}\label{eq:MittagLeffler}
		E_{\alpha}(t) = \sum_{k\ge 0} \frac{t^k}{\Gamma(\alpha k+1)},\ t\!\in \R \quad \textit{and} \quad e_{\alpha}(t) := E_{\alpha}(-t^{\alpha}) = \sum_{k\ge 0} (-1)^k \frac {t^{\alpha k}}{\Gamma(\alpha k +1)}, \quad t\ge 0.
	\end{equation}
	
	In section 5 of \cite{Pages2024} (see~\cite{GorenfloMainardi1997} further on), the author demonstrated that for such kernels $K_{\alpha}$, with $\frac{1}{2} < \alpha < 1$ (`` rough models''
	),
	$E_{\alpha}$ is increasing and differentiable on the real line with $\displaystyle \lim_{t\to +\infty}E_{\alpha} (t) =+\infty$ and $E_{\alpha}(0)=1$. In particular, $E_\alpha$ is an homeomorphism from $(-\infty, 0]$ to $(0,1]$. Consequently, 
	the resolvent $R_{\alpha,\lambda}$ satisfies the monotonicity assumption $({\cal R}_\lambda)\; $ of~\cite{Pages2024} for all $\lambda > 0$. 
	Moreover, it was shown that
	for every $\lambda >0$, the function $f_{\alpha,\lambda}  := -R_{\alpha,\lambda}$ given by~\eqref{eq:DerivSolventFracKernel} exists and is positive on $(0,+\infty)$ whenever \(\alpha \in (\frac12,1)\), so that
	$f_{\alpha,\lambda}$ is a probability density
	called {\em Mittag-Leffler density}  and is square-integrable with respect to the Lebesgue measure on $\mathbb{R}_+$. Consequently, the results established in \cite{Pages2024}, particularly in Sections 2, 3, and 4, apply to the case $\sigma(t, x) = \sigma(t)$ (Gaussian setting) and $\sigma(t, x) = \varsigma(t)\sigma(x)$.
	
	\medskip
	Note that, those properties does not holds for \(\alpha \in (1,2)\) and in this paper, {\em our Assumption $({\cal R}_\lambda)\; $ is a slightly relaxed version of that of \cite{Pages2024}.}
	The purpose of this part is to investigates if the results on {\em fake stationarity} extend to the general case where $ \alpha \in \R^*_+$. We show that for $0 < \alpha < 2$, the resolvent $R_{\alpha,\lambda}$ of $K_{\alpha}$  satisfy our relaxed monotonicity assumption $({\cal R}_\lambda)\; $ for all $\lambda > 0$, and that $f_{\alpha,\lambda} := -R_{\alpha,\lambda}$ exists and is square-integrable with respect to the Lebesgue measure on $\mathbb{R}_+$, for $1 < \alpha < 2$ (``long memory volatility models''). As a result, the findings from Sections \ref{subsec:fakestatioIntrin} and \ref{subsec:fakestatioStabil}, will be applicable in the cases where $\sigma(t, x) = \sigma(x)$ and $\sigma(t, x) = \varsigma_{\lambda,c}(t)\sigma(x)$.
	To this end, by a scaling property, it is enough to study   $R_{\alpha,1}$ ($\lambda =1$) given by its expansion $E_{\alpha}(-t^{\alpha})$ where $E_{\alpha}$ is the {\em Mittag-Leffler function}.
	
	\medskip
	We will leverage the conducive class of completely monotone functions. Let us recall that a function \( \varphi : (0,+\infty) \to [0,+\infty) \)  is called a {\em completely monotone} (CM) function if it is non-negative, \( C^\infty \) (i.e. it is infinitely differentiable on \((0,+\infty)\)), and satisfies \( (-1)^n \varphi^{(n)}(t) \geq 0 \) for all \( n \in \mathbb{N} \) and \(t>0\).
	
	Crucially, ``Bernstein-Widder theorem''  \cite[Theorem 1.4]{SchillingSongVondracek2012} (see also \cite{Bernstein1929}) provides a necessary and sufficient condition a function \( \varphi: \mathbb{R}_+ \to \R \) to be CM. More specifically, \(\varphi\) is CM if it is a ( real valued) Laplace transform of a unique non-negative measure \( \mu \) on \( [0, \infty) \). 
	Futhermore, a result by Pollard \cite{SchillingSongVondracek2012}  state that a CM function can be obtained by composing a CM function with a Bernstein function\footnote{A function \( \psi: \R_+ \to \R \) is called a Bernstein function if it is of class \(\mathcal C^\infty\), is non-negative, and its derivative is CM.} 
	\subsection{$\alpha$-fractional kernels with $\alpha>0$}\label{sec:AlphaFract}
	
	The Mittag-Leffler function \( E_{\alpha}(z) \), with \( \alpha > 0 \), generalizes the exponential function (attained with \( \alpha = 1 \)). It is defined by a power series, which converges on the entire complex plane. 
	In particular, we are interested in the alternate Mittag-Leffler function reading:
	\[
	e_{\alpha}(t) := E_{\alpha}(-t^{\alpha}) = \sum_{k\ge 0} (-1)^k \frac {t^{\alpha k}}{\Gamma(\alpha k +1)}, \quad t\ge 0, \quad E_{\alpha}(z) := \sum_{n=0}^{\infty} \frac{z^n}{\Gamma(\alpha n + 1)}, \quad \alpha > 0, \quad z \in \mathbb{C}.
	\]
	In the limiting cases \( \alpha = 1 \) and \( \alpha = 2 \), \( e_{\alpha}(t) \) are elementary functions, namely \(	e_1(t) = e^{-t} \quad \text{and} \quad e_2(t) = \cos t.\)
	Integral representations of the Mittag-Leffler function \( E_{\alpha} \) were first established in~\cite{Pollard1948}, followed by further results in~\cite{GorMain2000}, where they were connected by the Laplace transform. For instance, (see (F.12) in ~\cite{GorMain2000}), the Laplace transform of \( E_{\alpha}(-at^{\alpha}) \), with \( a \in \mathbb{C} \), is given by:\\
	\(	L_{E_{\alpha}(-at^{\alpha})}(z) = \frac{z^{\alpha - 1}}{z^{\alpha} + a}, \quad z \in \mathbb{C}, \; \Re(z) > |a|^{1/\alpha}, \quad \alpha > 0.\)
	From this, we can deduce the Laplace transform of \( e_{\alpha} \), which is given by: \(	L_{e_{\alpha}}(z) = \frac{z^{\alpha - 1}}{z^{\alpha} + 1}, \quad z \in \mathbb{C}, \; \Re(z) > 1, \quad \alpha > 0.\)
	Here, we define \( z^{\alpha} := |z|^{\alpha} e^{i\alpha \arg(z)} \), where \( -\pi < \arg(z) < \pi \), that is in the complex
	z-plane cut along the negative real axis
	
	\subsubsection{$\alpha$-fractional kernels for $ \alpha\!\in \R^*_+ $}\label{subsec:alphafrac1}
	
	\begin{Proposition}\label{prop:representation} The followings hold for the  the alternate Mittag-Leffler function for any \(t\geq 0\):
		\begin{enumerate}
			\item If  $\alpha\!\in \R^*_+ \setminus \mathbb{N}^*$, \(e_{\alpha}(t) = F_{\alpha}(t) + G_{\alpha}(t)\) where 
			\(F_{\alpha}(t):=\int_0^{+\infty} e^{-tu} H_{\alpha}(u) \, du\) with \(\forall\,  u\!\in \R_+,\;\)\\\( H_{\alpha}(u)= \frac{\sin(\alpha\pi)}{\pi}\frac{u^{\alpha-1}}{u^{2\alpha}+2u^{\alpha}\cos(\pi \alpha)+1}\) and \(G_{\alpha}(t):= \frac{2}{\alpha} \sum_{n=0}^{\lfloor \frac{\alpha-1}{2} \rfloor} \exp\left[t \cos\left(\frac{(2n+1)\pi}{\alpha}\right)\right] \cos\left[t \sin\left(\frac{(2n+1)\pi}{\alpha}\right)\right]\)
			\item If $\alpha \in \mathbb{N}^*$, \(
			e_{\alpha}(t) = G_{\alpha}(t) = \frac{2}{\alpha} \sum_{n=0}^{\lfloor \frac{\alpha-1}{2} \rfloor} \exp\left[t \cos\left(\frac{(2n+1)\pi}{\alpha}\right)\right] \cos\left[t \sin\left(\frac{(2n+1)\pi}{\alpha}\right)\right]\)
		\end{enumerate}
	\end{Proposition}
	
	The result or representation of the above proposition \ref{prop:representation} is an extension of the case \( \alpha \in (0,2) \) studied in  \cite{GorenfloMainardi1997} to the general case $\alpha\!\in \R \setminus \mathbb{N}$. The second claim is straigthforward as the function $ H_{\alpha}$ vanishes identically if \(\alpha\) is an integer.
	
	\bigskip
	\noindent {\bf Proof.} 
	Based on  the inverse Laplace transform (Bromwich-Mellin formula \footnote{on the Bromwich path, i.e., a line \(\text{Re} \{ z \} = a\) with \(a \geq 1\), and \(\text{Im} \{ z \}\) running from \(-\infty\) to \(+\infty\).}
	), we have the below representation as a Laplace inverse integral: For \(\gamma \) larger than the real parts of all poles of the integrand,
	{\small
		\begin{equation}\label{eq:LpInverseInt}
			e_\alpha(t)= \frac{1}{2\pi i} \int_{\textit{Br}(\gamma, \infty)} e^{zt}\frac{z^{\alpha-1}}{z^\alpha + 1} \, dz, = \frac{1}{2i\pi} \int_{z=\gamma-i\cdot\infty}^{z=\gamma+i\cdot\infty}e^{zt}\frac{z^{\alpha-1}}{z^\alpha + 1}dz = \frac{1}{2\pi i} \lim_{R \to +\infty} \int_{\textit{Br}(\gamma, R)} e^{zt}\frac{z^{\alpha-1}}{z^\alpha + 1} \, dz.
		\end{equation}
	}
	Let \(J_{\alpha}(t, \cdot) : z \to e^{zt}\frac{z^{\alpha-1}}{z^\alpha + 1}  \) be the integrand of the above representaion.
	The relevant poles of \(J_{\alpha}(t, \cdot)\) or rather \( \frac{z^{\alpha-1}}{z^\alpha + 1} \) is the set \(\mathbb{S}:=\{z_n = \exp \left(i \frac{(2n + 1)\pi}{\alpha} \right), n = 0, \cdots , \lfloor \alpha-1 \rfloor\}\). \(J_{\alpha}(t, \cdot)\) is thus holomorphic/analytic on \(\mathbb{C} \setminus \mathbb{S}\).
	And since 0 is a brand-point of the integrand  \(J_{\alpha}(t, \cdot)\), we consider \( \Gamma_{\gamma, \delta, R} \) the Jordan contour (see Figure \ref{fig:contour_deformation}) defined as the union of the below-represented several distinct paths:
	\[\Gamma_{\gamma, \delta, R} = \textit{Br}(\gamma, R) \cup C^+ \cup C_R^+ \cup (-\textit{H}(\delta, \frac{1}{R})) \cup C_R^- \cup C^-,\]
	\begin{minipage}[c]{0.5\textwidth}
		\begin{itemize}
			\item 	\( \textit{H}(\delta, \frac{1}{R}) \) is the {\em Hankel Contour} given by 
			\(\textit{H}(\delta, \frac{1}{R}):=
			[-R + i\delta, -c + i\delta] \cup C_{\frac{1}{R}} \cup [-R - i\delta, -c - i\delta],
			\) where $C_{\frac{1}{R}}$ is the small circular arc  \( |s| = \frac1R \).\\ 
			\item \( \textit{Br}(\gamma, R) \), the {\em truncated Bromwich Path} i.e. \( \textit{Br}(\gamma, R):= [\gamma - iR, \gamma + iR] \), where \( \gamma \geq 1 \) and \(\text{Re} \{ z \} = \gamma\), with \( \text{Im} \{ z \} \in [-R, R] \).\\ 
			\item  {\em  \( C^+ :=[\gamma + iR, iR]\)} and {\em \( C^- :=[-iR, \gamma - iR]\)}.\\ 
			\item {\em \( C_R^+ \)} and {\em \( C_R^- \)} denote the upper and lower semicircular arcs, respectively, of a circle of radius \( R \); \( C_R^+ \) runs from \( iR \) to \( -R + i\delta \), and \( C_R^- \) from \( -R - i\delta \) to \( -iR \).
		\end{itemize}
	\end{minipage}
	\hfill
	\begin{minipage}[c]{0.5\textwidth}
		\centering
			\begin{tikzpicture}[scale=1.1]
				\def\gammaVal{1.5}
				\def\R{2.5}
				\def\deltaVal{0.15}
				\def\c{0.5}
				\def\r{1/\R}
				
				\draw[->] (-3, 0) -- (3, 0) node[right] {$\Re(z)$};
				\draw[->] (0, -3) -- (0, 3) node[above] {$\Im(z)$};
				
				\draw[blue, thick, ->] (\gammaVal, -\R) -- (\gammaVal, \R);
				\node at (\gammaVal + 0.3, 0.2) {\small \textcolor{blue}{$\textit{Br}(\gamma, R)$}};
				
				\draw[red, thick, ->] (\gammaVal, \R) -- (0, \R) node[midway, above] {\small $C^+$};
				\draw[red, thick, <-] (\gammaVal, -\R) -- (0, -\R) node[midway, below] {\small $C^-$};
				
				\draw[green, thick, ->] (0, \R) arc[start angle=90, end angle=175, radius=\R];
				\draw[green, thick, <-] (0, -\R) arc[start angle=270, end angle=185, radius=\R];
				\node at (-2.4, 2.1) {\small $C_R^+$};
				\node at (-2.4, -2.1) {\small $C_R^-$};
				
				\draw[orange, thick, <-] (-\R, \deltaVal) -- (-\c, \deltaVal);
				\draw[orange, thick, ->] (-\R, -\deltaVal) -- (-\c, -\deltaVal);
				\node at (-1.8, -0.35) {\small \textit{H}$(\delta, \tfrac{1}{R})$};
				
				\draw[orange, thick] (-\c, \deltaVal) arc[start angle=160, end angle=-155, radius=\r];
			\end{tikzpicture}
		
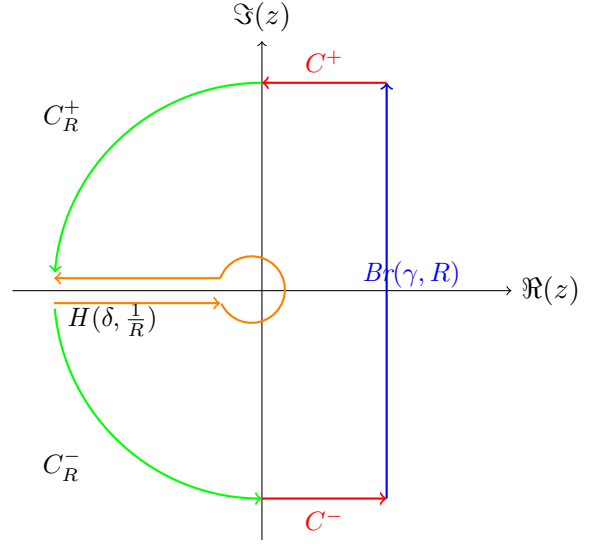
\captionof{figure}{Jordan contour \( \Gamma_{\gamma, \delta, R} \).}\label{fig:contour_deformation}
	\end{minipage}
	
	For small values of \( \delta \), large values of \( R \), and \( \gamma \geq 1 \), the Jordan contour \( \Gamma_{\gamma, \delta, R} \) encloses all poles of \( J_{\alpha}(t, \cdot) \). Therefore, by the {\em Jordan-Cauchy Residue Theorem}, we have:
	{\small
		\begin{align*}
			&\sum_{z \in \mathbb{C} \setminus \{-1\}: z^\alpha=-1} \text{Res}(J_{\alpha}(t, \cdot), z) 
			= \frac{1}{2\pi i} \oint_{\Gamma_{\gamma, \delta, R}} J_{\alpha}(t, z) \, dz 
			= \frac{1}{2\pi i}\int_{\textit{Br}(\gamma, R)} J_{\alpha}(t, z) \, dz + \frac{1}{2\pi i}\int_{C^+} J_{\alpha}(t, z) \, dz \\
			&\quad \hspace{1.5cm} + \frac{1}{2\pi i}\int_{C_R^+} J_{\alpha}(t, z) \, dz - \frac{1}{2\pi i}\int_{\textit{H}(\delta, \frac{1}{R})} J_{\alpha}(t, z) \, dz  + \frac{1}{2\pi i}\int_{C_R^-} J_{\alpha}(t, z) \, dz + \frac{1}{2\pi i}\int_{C^-} J_{\alpha}(t, z) \, dz.
		\end{align*}
	}
	Taking the limit as \(R \to \infty\) and \(\delta \to 0\), we may decompose \eqref{eq:LpInverseInt} as follows
	{\small
		\begin{align*}
			e_\alpha(t) &:= \frac{1}{2\pi i} \lim_{R \to +\infty} \int_{\textit{Br}(\gamma, R)} J_{\alpha}(t, z) \, dz =	\sum_{z \in \mathbb{C} \setminus \{-1\}: z^\alpha=-1} \text{Res}(J_{\alpha}(t, \cdot), z) 
			+ \frac{1}{2\pi i}\lim_{R \to +\infty}\lim_{\delta \to 0}\int_{\textit{H}(\delta, \frac{1}{R})} J_{\alpha}(t, z) \, dz \\
			& - \frac{1}{2\pi i} \left(\lim_{R \to +\infty}\int_{C^+} J_{\alpha}(t, z) \, dz 
			+ \lim_{R \to +\infty}\int_{C_R^+} J_{\alpha}(t, z) \, dz + \lim_{R \to +\infty}\int_{C_R^-} J_{\alpha}(t, z) \, dz + \lim_{R \to +\infty}\int_{C^-} J_{\alpha}(t, z) \, dz\right).
		\end{align*}
	}
	We now examine these six terms. The contribution from the Hankel path 
	is given by $\frac{1}{2\pi i}\int_{\textit{H}(\delta, \frac{1}{R})} J_{\alpha}(t, z) \, dz$, whose limit coincides with the usual contour representation of the Mittag-Leffler function for \(\alpha \in (0,1)\).
	\begin{equation}\label{eq:Laplace_transform}
		\frac{1}{2\pi i} \int_{\textit{H}(\delta, \frac{1}{R})} e^{zt}\frac{z^{\alpha-1}}{z^\alpha + 1} \, dz  \overset{R \to +\infty,\delta \to 0}{=} \int_0^{+\infty} e^{-tu} H_{\alpha}(u) \, du = L_{H_{\alpha}}(t) =:F_{\alpha}(t).
	\end{equation}
	where a synthetic formula was found for $H_{\alpha}$ in~\cite{GorMain2000} (see (F.22) p.31, see also~\cite{Mainardi2014} in the case $0<\alpha<1$).
	
	\begin{equation}\label{eq:Halpha}
		\forall\,  u\!\in \R_+,\; H_{\alpha}(u)=- \frac{1}{2\pi} \cdot 2\,\Im {\rm m}\Big( \frac{z^{\alpha-1}}{z^{\alpha}+1}\Big)_{|z= ue^{i\pi}}  = \frac{\sin(\alpha\pi)}{\pi}\frac{u^{\alpha-1}}{u^{2\alpha}+2u^{\alpha}\cos(\pi \alpha)+1}
	\end{equation}
	Note that this representation of \(F_{\alpha}\) in term of the Laplace transform of a non-negative Lebesgue integrable function (see Equation~\eqref{eq:Laplace_transform} above) was first established in~\cite{Pollard1948}.
	
	Also note that the function $ H_{\alpha}$ vanishes identically if \(\alpha\) is an integer.
	The  limit of the other integrals vanishes. In fact:
	\[\left \lvert ~ \int_{C_R^+} J_{\alpha}(t, z) \, \mathrm{d} z \, \right \rvert \leq \int_{\frac{\pi}{2}}^{\pi} R \lvert J_{\alpha}(t, e^{i\theta})\rvert  \, \mathrm{d} \theta \,\text{and} \, \lvert J_{\alpha}(t, e^{i\theta})\rvert  \, \leq \frac{R^{\alpha-1}}{R^{\alpha}-1}e^{tR\cos(\theta)} \leq \frac{R^{\alpha-1}}{R^{\alpha}-1}e^{tR(-\frac2\pi \theta +1)}\] 
	where in the last inequality, we used the trick $\cos(\theta) \leq -\frac2\pi \theta +1 \quad \forall \theta \in [\frac{\pi}{2},\pi]$.
	Consequently, 
	\[\left \lvert ~ \int_{C_R^+} J_{\alpha}(t, z) \, \mathrm{d} z \, \right \rvert \leq \frac{R^{\alpha-1}}{R^{\alpha}-1} \times \frac{\pi}{-2tR}\left[e^{tR(-\frac2\pi \theta +1)}\right]_{\theta = \frac\pi2}^{\theta=\pi} = \frac{\pi R^{\alpha}}{2t (R^{\alpha+1}-R)} (1-e^{-tR})  \stackrel{R \to \infty}{\longrightarrow} 0 \,\]
	Likewise $\lim_{R \to +\infty} \left \lvert ~ \int_{C_R^-} J_{\alpha}(t, z) \, \mathrm{d} z \, \right \rvert = 0$. Moreover:
	$\int_{C^+} J_{\alpha}(t, z) \, dz = \int_{\gamma+iR}^{iR} J_{\alpha}(t, z) \, dz = \int_{\gamma}^{0} J_{\alpha}(t, x+iR) \, dx$.
	Now, observe that:
	$ \lvert J_{\alpha}(t, x+iR)\rvert  \, \leq \frac{(x^2 + R^2)^{\frac{\alpha-1}{2}}}{(x^2 + R^2)^{\frac{\alpha}{2}}-1}e^{tx} \leq e^{tx} \frac{(\gamma^2 + R^2)^{\frac{\alpha-1}{2}}}{ R^\alpha-1}.$
	As a consequence, \(\left \lvert ~ \int_{C^+} J_{\alpha}(t, z) \, \mathrm{d} z \, \right \rvert \leq \frac{(\gamma^2 + R^2)^{\frac{\alpha-1}{2}}}{ R^\alpha-1}\int_{\gamma}^0 e^{tx}  \, \mathrm{d} x \, \stackrel{R \to \infty}{\longrightarrow} 0.\)
	Likewise for $\lim_{R \to +\infty}  \left \lvert ~ \int_{C^-} J_{\alpha}(t, z) \, \mathrm{d} z \, \right \rvert  = 0.$
	Finally,  
	\[
	G_{\alpha}(t) := \sum_{z \in \mathbb{C} \setminus \{-1\}: z^\alpha=-1} \text{Res}(J_{\alpha}(t, \cdot), z)  = \sum_{z_n \in \mathbb{S}} \text{Res}(J_{\alpha}(t, \cdot), z_n) =\sum_{n=0}^{\lfloor \alpha-1 \rfloor} e^{z_n t} \, \text{Res} \left[ \frac{z^{\alpha-1}}{z^\alpha + 1} \right]_{z_n} = \frac{1}{\alpha} \sum_{n=0}^{\lfloor \alpha-1 \rfloor}  e^{z_n t},
	\]
	Note that, \(e^{z_n t} + e^{\bar{z}_n t} = e^{\text{Re} \{ z_n \} t}\left(e^{\text{Im} \{ z_n \} t} + e^{-\text{Im} \{ z_n \} t}\right) = 2 e^{\text{Re} \{ z_n \} t} \cos\left(\text{Im} \{ z_n \} t\right) \) and \\ \(\sum_{z_n \in \mathbb{S}} \text{Res}(J_{\alpha}(t, \cdot), z_n) = \frac{1}{\alpha} \sum_{n=0}^{\lfloor \alpha-1 \rfloor}  e^{z_n t} = \frac{1}{\alpha} \sum_{n=0}^{\lfloor \frac{\alpha-1}{2} \rfloor} \left(e^{z_n t} + e^{\bar{z}_n t}\right).\)
	As a consequence, 
	{\small 
		\[
		G_{\alpha}(t) := \sum_{z \in \mathbb{C} \setminus \{-1\}: z^\alpha=-1} \text{Res}(J_{\alpha}(t, \cdot), z)  = \frac{1}{\alpha} \sum_{n=0}^{\lfloor \alpha-1 \rfloor}  e^{z_n t} = \frac{2}{\alpha} \sum_{n=0}^{\lfloor \frac{\alpha-1}{2} \rfloor} \exp\left[t \cos\left(\frac{(2n+1)\pi}{\alpha}\right)\right] \cos\left[t \sin\left(\frac{(2n+1)\pi}{\alpha}\right)\right]
		\]
	}
	\begin{figure}[H]
		\centering
		\includegraphics[width=1\textwidth]{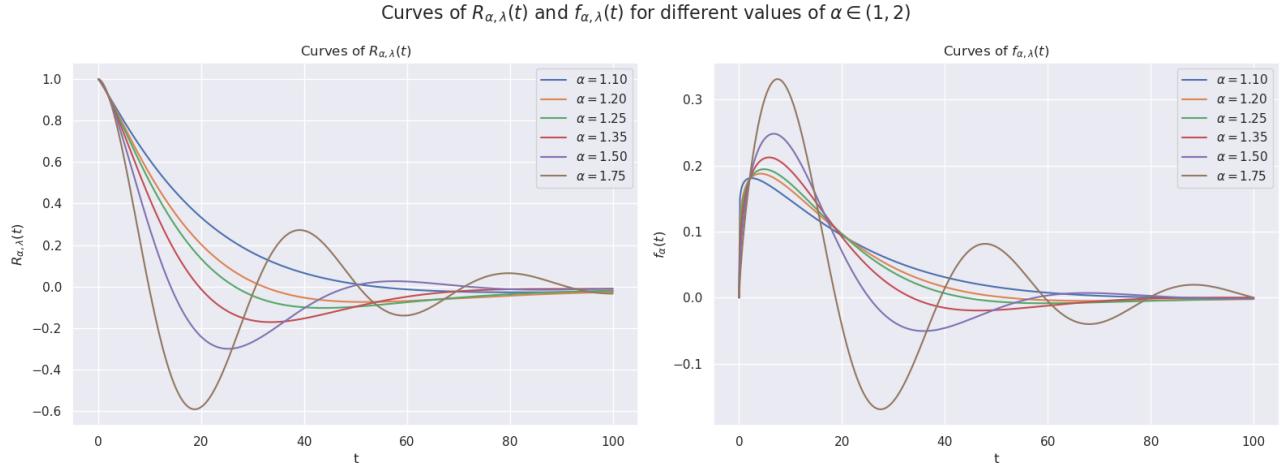} 
		\caption{Curves of $R_{\alpha,\lambda}(t)$ and $f_{\alpha,\lambda}(t)$ for different values of $\alpha \in (1,2)$}
		\label{fig:curves_alpha_more1}
	\end{figure}
	\noindent
	{\bf Remark:} 1. For \( 0 < \alpha < 1 \), there are no relevant poles since \( | \arg(z_k) | > \pi \), so \( G_{\alpha}(t) \equiv 0 \), and we obtain \(e_{\alpha}(t) = F_{\alpha}(t),\; \text{for}\quad 0 < \alpha < 1\), (see e.g. \cite[Figure 2]{EGnabeyeu2025}). For \( 1 < \alpha < 2 \), there are exactly two relevant poles, \( z_0 = \exp(i\pi/\alpha) \) and \( z_{-1} = \exp(-i\pi/\alpha) = \bar z_0 \), located in the left half-plane. In this case, we have $G_{\alpha}(t) = \frac{2}{\alpha} e^{t \cos \left( \frac{\pi}{\alpha} \right)} \cos \left( t \sin \left( \frac{\pi}{\alpha} \right) \right)$ and \(e_{\alpha}(t) = \int_0^{+\infty} e^{-tu} H_{\alpha}(u) \, du  + \frac{2}{\alpha} e^{t \cos \left( \frac{\pi}{\alpha} \right)} \cos \left( t \sin \left( \frac{\pi}{\alpha} \right) \right).\)
	It is clear that the function \( e_{\alpha}(t) \) oscillates in an evanescent manner to 0 as $t \to +\infty$.
	We note that this function exhibits oscillations with circular frequency and an exponentially decaying amplitude (see Figure \ref{fig:curves_alpha_more1}).
	Note that, the above expression of \(e_{\alpha}\) is the same for \( 2 < \alpha < 3 \) with the only difference that the two poles  are now located in  the right haft plane, and so providing amplified oscillations.\\
	\medskip
	2. In the case \( 2 < \alpha < +\infty \) ,
	however, certains poles are located in the right half plane, so providing amplified
	oscillations. This common instability for \(\alpha >2 \) is the reason why
	we will limit ourselves to consider \(\alpha\) in the range \( 0 < \alpha < 2 \) as highlighted by the below proposition.
	\begin{Proposition}\label{prop:main_general alpha}Let $\lambda>0$ and let $\alpha\!\in (1,+\infty) \setminus \mathbb{N}$.\\
		\smallskip
		\noindent $(a)$ The function \((-1)^{\lfloor \alpha \rfloor}F_{\alpha}\) is completely monotonic(thus convex), hence infinitely differentiable on $\mathbb{R}_+^*$.
		
		\smallskip
		\noindent $(b)$ 	The $\lambda$-resolvent $R_{\alpha, \lambda}$ satisfies $ R_{\alpha,1}= e_{\alpha}$ and $R_{\alpha, \lambda} = R_{\alpha,1}(\lambda^{1/\alpha}\cdot)$. The function $R_{\alpha, \lambda}$ is infinitely differentiable i.e. $\mathcal C^\infty$ on $(0, +\infty)$. Moreover $R_{\alpha, \lambda}(0)=1$, $R_{\alpha,\lambda}\!\in {\cal L}^{r}(\R_+)$ for every $r>\frac{1}{\alpha}$ and $\alpha<2$.
		
		$f_{\alpha, \lambda} (t):= -R'_{\alpha, \lambda} (t) $ is infinitely differentiable and satisfy : $\forall\, t>0, \quad f_{\alpha, \lambda} (t)=$
		{\small 
			\[
			\lambda^{\frac{1}{\alpha}}\left(\int_0^{+\infty} e^{-\lambda^{\frac{1}{\alpha}}tu}uH_{\alpha}(u)du - \frac{2}{\alpha} \sum_{n=0}^{\lfloor \frac{\alpha-1}{2} \rfloor} \exp\left[t \lambda^{\frac{1}{\alpha}} \cos\left(\frac{(2n+1)\pi}{\alpha}\right)\right] \cos\left[t \lambda^{\frac{1}{\alpha}} \sin\left(\frac{(2n+1)\pi}{\alpha}\right)- \frac{(2n+1)\pi}{\alpha}\right]\right)
			\]
		}
		so that, \( R_{\alpha,\lambda} \) converges to \( a \in [0,1) \) and $f_{\alpha,\lambda} \!\in {\cal L}^{2\beta}(\R_+)$ for every $\beta >0 $
		provided  $\alpha \in (1,2)$.
		
		\noindent $(c)$ 
		Furthermore, the function \( R_{\alpha,\lambda} \) satisfies assumption \(({\cal R}_\lambda)\) (i), namely \( R_{\alpha,\lambda} \) converges to \( 0 \), and the function \( f_{\alpha,\lambda} \) satisfies a $\vartheta$-H\"older regularity of the form~\eqref{eq:Regulf}
		 if and only if \( \alpha \in (1,2) \).
	\end{Proposition}
	\begin{Proposition}[$\alpha$-fractional kernels $1 <\alpha<2$]\label{prop:main}Let $\lambda>0$ and let $\alpha\!\in (1,2)$.\\
		\smallskip
		\noindent $(a)$  The $\lambda$-resolvent $R_{\alpha, \lambda}$ satisfies $ R_{\alpha,1}= e_{\alpha}$ and $R_{\alpha, \lambda} = R_{\alpha,1}(\lambda^{1/\alpha}\cdot)$. The function $e_{\alpha}$ and thus $R_{\alpha, \lambda}$ are infinitely differentiable i.e. $\mathcal C^\infty$ on $(0, +\infty)$) with:
		{\small
			\begin{equation}
				\forall k \in \N, \quad e^{(k)}_\alpha(t) = F^{(k)}_\alpha(t) + G^{(k)}_\alpha(t) \quad \text{where}\quad	F^{(k)}_\alpha(t) = \int_0^{+\infty} e^{-tu} H^{(k)}_{\alpha}(u) \, du \;
			\end{equation}
		}
		\begin{equation}\label{eq:derivG_k}
			H^{(k)}_{\alpha}(u) := (-1)^k \frac{\sin(\alpha \pi) }{ \pi } \frac{u^{\alpha - 1 + k}}{u^{2\alpha} + 2u^{\alpha} \cos(\alpha \pi) + 1}\; \text{and} \; G^{(k)}_\alpha(t) = \frac2\alpha e^{t\cos \left( \frac{\pi}{\alpha} \right)} \cos \left[ t \sin \left( \frac{\pi}{\alpha} \right) - \frac{k \pi}{\alpha} \right].
		\end{equation}
		
		Moreover $R_{\alpha, \lambda}(0)=1$, $R_{\alpha,\lambda}(t) \leq 1 \quad \forall t\geq0$,  $R_{\alpha,\lambda}$ converges  to $0$.
		$R_{\alpha,\lambda}\!\in {\cal L}^{r}(\R_+)$ for every $r>\frac{1}{\alpha}$ and 
		$f_{\alpha, \lambda}:= -R'_{\alpha, \lambda}$ is infinitely differentiable, bounded converges to $0$ and satisfy:
		\[
		\forall\, t>0, \quad f_{\alpha, \lambda} (t):= -R'_{\alpha, \lambda} (t) = \lambda^{\frac{1}{\alpha}}\left(\int_0^{+\infty} e^{-\lambda^{\frac{1}{\alpha}}tu}uH_{\alpha}(u)du - \frac2\alpha e^{t \lambda^{\frac{1}{\alpha}}\cos \left( \frac{\pi}{\alpha} \right)} \cos \left[ t \lambda^{\frac{1}{\alpha}} \sin \left( \frac{\pi}{\alpha} \right) - \frac{\pi}{\alpha} \right]\right).
		\]
		\noindent $(b)$ Moreover, if  $\alpha\!\in (1, 2)$,$f_{\alpha,\lambda} \!\in {\cal L}^{p}(\R_+)$ for every $p \in (0,+\infty] $ and for every $\vartheta\!\in\big(0,\alpha-1\big)$
		\[| f_{\alpha ,\lambda}(t+\delta)-f_{\alpha,\lambda}(t)  | \leq h_{\alpha,\lambda}(t) \delta^{\vartheta} \quad \text{with}\quad  h_{\alpha,\lambda} \in {\cal L}^1(\R_+) \cap {\cal L}^\infty(\R_+) \quad \text{and}\quad h_{\alpha,\lambda}*\varsigma^2_{\alpha,\lambda,c} \in {\cal L}^\infty(\R_+)\]
	\end{Proposition}
	\noindent{\bf Remark: } As a consequence of claim~$(b)$,
	for each $p \in \{1,2\}$ and every $\vartheta_p\!\in\big(0,\alpha-1\big)$ (we will actually prove the sharper estimate if $\alpha\!\in (\frac12, \frac32)$, for each $p\geq1$ and every $\vartheta_p\!\in\big(0,(\alpha-1+\frac1p)\wedge1\big)$), 
	there exists a real constant $C_{\vartheta_p, \lambda}>0$ such that
	\vspace{-.1cm}
	\begin{equation}\label{eq:thetaHolder}
		\forall\, \delta > 0, \quad 
		\left[\int_0^{+\infty} |f_{\alpha,\lambda}(t+\delta) - f_{\alpha,\lambda}(t) |^p \, dt \right]^{\frac1p} 
		\le C_{\vartheta_p,\lambda} \delta^{\vartheta_p}.
		\vspace{-.1cm}
	\end{equation}
		For clarity and conciseness, the proofs of Propositions \ref{prop:main_general alpha} and \ref{prop:main} are postponed to Appendix \ref{app:lemmata}.
	\begin{Theorem}\label{thm:mainFractional}
		Let \( \alpha \in \left( 1, \frac32\right) \), 
		let \( K(t) = K_{\alpha}(t) =  \frac{t^{\alpha-1}}{\Gamma(\alpha)} \), \( t > 0 \) the fractional kernel, let $\sigma(t,x):= \varsigma(t)\sigma(x)$ in~\eqref{eq:Volterrameanrevert}-\eqref{eq:Volterrameanrevert_}  with  \( \sigma \) a Lipschitz continuous function given by equation~\eqref{eq:sigmafakeI&II} with \([\sigma]_{\rm Lip}^2:=\kappa_2>0\), $\varsigma$ given by Equation~\eqref{eq:VolterraStabilizerCont}. Let  \( c \in \left( 0, \frac{\alpha-1}{8\kappa_2} \right) \), $\lambda >0$ and let \( X_0 \in L^p(\P) \) for some $p>\frac{2}{\alpha-1}>2$ such that \( \mathbb{E}[X_0] =m_0\) and \(Var(X_0) = v_0=\frac{c\sigma^2( m_0)}{1-c\kappa_2}.\)
		Then,
		\begin{enumerate}
			\item For fractional kernels \( K_{\alpha} \) with \( 1 < \alpha < \frac32 \), the solution $(X_t )_{t\geq0}$ to the Volterra equation~\eqref{eq:ConstMeanVar} or ~\eqref{eq:CondMean} starting from $X_0$ has a fake
			stationary regime of type I in the sense that:
			
			\centerline{$
				\forall\, t\ge 0, \qquad \E\, X_t =  m_0,  \quad{\rm Var}(X_t) = v_0=\frac{c\sigma^2( m_0)}{1-c\kappa_2} \mbox{ and }\quad \E\, \sigma^2(X_t) =  \bar \sigma_0^2 =\frac{\sigma^2( m_0)}{1-c\kappa_2}.
				$}
			\item If $m_0= (1-a) \frac{\mu_\infty}{\lambda}$ i.e. $\ell_\infty=0$, then for any starting random variable \( X_0^\prime  \in L^2(\P) \), a solution to the Volterra Equation~\eqref{eq:ConstMeanVar} or ~\eqref{eq:CondMean} starting from \( X_0^\prime \) satisfies \( \| X_t^\prime - X_t \|_2 \underset{t \to \infty}{\to} 0\)
			\item The family of shifted processes \( X_{t+\cdot}, t \geq 0 \), is \( C \)-tight as \( t \to +\infty \) and its (functional) limiting distributions are all \( L^2 \)-stationary processes with covariance function \( C_\infty \) given by~\eqref{eq:funclongRun_intrinsic} or\eqref{eq:funclongRun}.
		\end{enumerate}
	\end{Theorem}
	\noindent {\bf Proof.}	Let \( 1 < \alpha < \frac32 \).
	By Proposition~\ref{prop:main2}~$(c)$, we let \(0<\vartheta<\alpha-1\) and any \(	\beta>1\) by Assumption~\ref{ass:int_holregul}. We have
	that conditions~\eqref{eq:func_weak} of Theorem~\ref{Thm:funcWeak} are satisfied for \(\frac{2}{\alpha-1}<p < \frac1{4c\kappa_2}\) which we may assume, without loss of generality. In fact, in this case
	\[
	\lim_{\substack{\vartheta\to \alpha-1,
			\beta\to +\infty}}
	\Big[p\big(
	 \vartheta \wedge \frac{\beta-1}{2\beta}
	\big) \wedge \frac{p}{2}\big(
	\vartheta \wedge (1-\frac{1}{2\beta})
	\big)\Big]
	= \Big[ p\big(
	(\alpha-1) \wedge \frac{1}{2}
	\big) \wedge \frac{p}{2}\big(
	(\alpha-1) \wedge 1
	\big) \Big]=
	\frac{p(\alpha-1)}{2}
	>1.
	\]
	Therefore, theorems~\ref{Thm:Fake-funcWeak_Intrinsic} and~\ref{Thm:funcWeak} applies.
	The claim (2) is a consequence of~\cite[Proposition 4.4]{EGnabeyeu2025}.  \hfill $\square$
	\subsection{The function $\varsigma_{\alpha,\lambda,c}^2$ solution of the stabilizer equation when $\alpha \in (0,2)$}\label{subsec:sigma_Lmemory}
In this section, we derive an explicit representation of
$\varsigma_{\lambda,c}^{2}$ as a power series in $t^{k\alpha}$.
To motivate the form of this expansion, we first investigate the
small-time asymptotic behaviour of $\varsigma_{\lambda,c}^{2}$,
which naturally suggests an appropriate power-series ansatz.
	\subsubsection{Small-time asymptotics of the solution
		$\varsigma_{\alpha,\lambda,c}^{2}$ when $\alpha \in (0,2)$}
	In this section, we derive the leading-order asymptotic behaviour of
	$\varsigma_{\alpha,\lambda,c}^{2}$ as \(t\to0\) near the origin, assuming that
	\(\mu(t)\sim\mu(0)\) with \(\mu(0)/m_0\ge0\). For simplicity, we omit the dependence on \(\alpha\).
	To this end we rely on the Laplace version of the equation \textit{($E_{\lambda, c}$)} in ~\eqref{eq:VolterraStabilizer} satisfied by $\varsigma^2_{\lambda,c}$, namely \(\, c \lambda^2 \big(1- (\phi - f_{\lambda} * \phi)^2(t) \big) =  (f_{\lambda}^2 * \varsigma^2)(t)\; \forall\, t\ge 0 \). Applying an integration by parts, its Laplace transform is given by equation~\eqref{eq:Laplacesigma} 
	\begin{equation}\label{eq:Laplacesigma}\forall\, t>0, \quad  t\,L_{f^2_\lambda}(t).L_{\varsigma^2}(t)= - 2\,c\lambda^2 L_{(\phi - f_{\lambda} * \phi)(\phi - f_{\lambda} * \phi)^\prime}(t).
	\end{equation}
	Given the kernel $K_{\alpha}(u) = \frac{u^{\alpha-1}}{\Gamma(\alpha)}$ and the expansion of the resolvents $R_{\alpha,\lambda}$ and it derivative $-f_{\alpha,\lambda}$, \(\forall t\ge 0,\)
	{\small
		\begin{equation}\label{eq:e_alpha}
			R_{\alpha,\lambda}(t) = \sum_{k\ge 0} (-1)^k \frac{\lambda ^k t^{\alpha k}}{\Gamma(\alpha k+1)}= E_{\alpha}(-\lambda t^{\alpha} ), \; f_{\alpha, \lambda}(t)= \alpha\lambda t^{\alpha-1} E'_{\alpha}(-\lambda t^{\alpha})  = \lambda t^{\alpha-1}\sum_{k\ge 0}(-1)^k \frac{\lambda^kt^{\alpha k}}{\Gamma(\alpha (k+1))}.
		\end{equation}
	}
	Since \( \phi(t) - (f_{\lambda} * \phi)(t) = 1 - \frac{(f_{\lambda} * \mu)_t}{\lambda m_0}  \), we have \(\phi(t) - (f_{\lambda} * \phi)(t) \stackrel{0}{\sim} 1\) and owing to Assumption~\ref{ass:resolvent}~$(iii)$,
	\((\phi(t) - (f_{\lambda} * \phi)(t))^\prime\stackrel{0}{\sim} -\frac{\mu(0)}{\lambda m_0}f_{\lambda}(t) \). Together with Equation~\eqref{eq:e_alpha}, this yields
	 \[(\phi - f_{\lambda} * \phi)(\phi - f_{\lambda} * \phi)^\prime(t)\stackrel{0}{\sim} -\frac{\mu(0)}{\lambda m_0}\frac {\lambda t^{\alpha-1}}{\Gamma(\alpha)} \quad \mbox{ and}\quad f^2_{\lambda}(t)\stackrel{0}{\sim} \frac {\lambda^2 t^{2(\alpha-1)}}{\Gamma(\alpha)^2}.\]
	It follows that -- at least heuristically~(\footnote{We use here heuristically a dual version of the Hardy-Littlewood Tauberian theorem for Laplace transform, namely  $\varsigma^2(t)\stackrel{0}{\sim}Ct^{\gamma}$, $\gamma>-1$,  iff $L_{\varsigma^2}(t)\stackrel{+\infty}{\sim}C t^{-(\gamma+1)} \Gamma(\gamma+1).$ We refer to \cite{BiGoTe1989,DeHaanFerreira2006} for a general theory of regular variation.})~--
	\begin{equation}
	L_{(\phi - f_{\lambda} * \phi)(\phi - f_{\lambda} * \phi)^\prime}(t)\stackrel{+\infty}{\sim} -\lambda \frac{\mu(0)}{\lambda m_0}t^{-\alpha} \quad \mbox{ and}\quad L_{f^2_{\lambda}}(t)\stackrel{+\infty}{\sim} \frac {\lambda^2\Gamma(2\alpha-1) t^{-(2\alpha-1)}}{\Gamma(\alpha)^2}.
	\end{equation}
	This implies that \(L_{\varsigma^2}(t) \stackrel{+\infty}{\sim} 2\lambda \,c \frac{\mu(0)}{\lambda m_0} \frac{\Gamma(\alpha)^2}{\Gamma(2\alpha-1)} t^{-(2-\alpha)} \)
	owing to Equation~\eqref{eq:Laplacesigma}.
	This in turn suggests that 
	\begin{equation}\label{eq:varsigma}
		\varsigma^2(t) \stackrel{0}{\sim} \frac{2\lambda c \Gamma(\alpha)^2}{\Gamma(2\alpha-1) \Gamma(2-\alpha)} \frac{\mu(0)}{\lambda m_0} t^{1-\alpha} \quad \text{so that}
		\left\{
		\begin{array}{ll}
			(i) & \varsigma(0) = 0 \text{ if } \alpha < 1, \text{see e.g. ~\cite{Pages2024, EGnabeyeu2025}}\\
			(ii) & \lim_{t \to 0^+} \varsigma(t) = +\infty \text{ if } \alpha > 1.
		\end{array}
		\right.
	\end{equation}
	This suggests 
	to search $\varsigma^2(t)$ of the form ({\em Power Series Ansatz}):
	\begin{equation}\label{eq:expectedsigma2}
		\varsigma^2(t) = \varsigma_{\alpha,\lambda,c}^2(t):= 2\,\lambda \, c\,t^{1-\alpha}\sum_{k\ge 0} (-1)^k c_k\lambda^k t^{\alpha k} \quad \text{with}\quad  c_0 = \frac{\Gamma(\alpha)^2}{\Gamma(2\alpha-1)\Gamma(2-\alpha)}\frac{\mu(0)}{\lambda m_0}.
	\end{equation}
	
	\noindent {\bf Remarks :} 1.  At this point, it is crucial to emphasize that, for a fixed value of \( \alpha \), all functions \( \varsigma_{\alpha, \lambda, c}^2 \) from equation~\eqref{eq:expectedsigma2} are derived or generated from a common function, defined as
	\begin{equation}\label{eq:varsigma_reduit}
		\varsigma^2_{\alpha, \lambda, c}(t) = c \lambda^{2 - \frac{1}{\alpha}} \varsigma_\alpha^2\left( \lambda^{\frac{1}{\alpha}} t \right) \quad \text{with} \quad	\varsigma_{\alpha}^2(t) := 2\, t^{1-\alpha} \sum_{k \geq 0} (-1)^k c_k t^{\alpha k}.
	\end{equation}
	where the coefficients \( c_k \) depend on \( \alpha \). Thus, for simplicity in what follows, we will assume \( c = \lambda = 1 \).
	
	2. For the computation of the function $\varsigma_{\alpha,\lambda,c}^2$, we need to establish a recurrence formula satisfied by the coefficients $c_k$, which involves knowing the  form of the mean-reverting function \(\mu\). In practice, since this function is usually taken to be constant equal to \(\mu_0\), we are going in the next subsection to compute and study the function $\varsigma_{\alpha,\lambda,c}^2$ when $ \mu(t) = \mu_0=\lambda m_0 \quad \text{a.e.}$ and $\alpha \in (1,\frac32)$  bearing in mind that, the case when $\alpha \in (\frac12,1)$ have been intensively study in \cite{Pages2024, EGnabeyeu2025}.
	
	\subsubsection{Existence and computation of $\varsigma_{\alpha,\lambda,c}^2$ solution of the stabilizer equation when $\alpha \in (1, 2)$}\label{subsec:sigma_Lmemory}
	
	The recurrence formula satisfied by the coefficients $c_k$, which make possible the computation of the functions $\varsigma_{\alpha, \lambda,c}$ are established in the same manner as in \cite{Pages2024}.
	We consider the case where \( \mu(t) = \mu_0=\lambda m_0 \) a.e., so that \( \mu_\infty = \mu_0=\lambda m_0 \), and \( \phi \equiv 1 \). 
	 We then have the following proposition, whose proof is postponed to Appendix~\ref{app:lemmata}.
	\begin{Proposition}[Existence of the function $\varsigma_{\alpha,\lambda,c}^2$ for $\alpha \in (1,2)$]\label{prop:alphaFractKernel1}
		Let $\alpha \in (1,2)$: 
		\begin{enumerate}
			\item \(\lim_{t\to0} \varsigma^2_{\alpha,\lambda,c} = +\infty\), and \(
			\lim_{t \to +\infty} \varsigma^2_{\alpha,\lambda,c}(t) = \frac{c\lambda^2}{\|f_{\alpha,\lambda}\|^2_{{\cal L}^2(\R_+)}}.
			\)
			\item 
			\(\varsigma^2_{\alpha,\lambda, c}(t) = c \lambda^{2-\frac1\alpha}\varsigma_\alpha^2(\lambda^{\frac1\alpha} t)\) where \(\varsigma_{\alpha}^2(t):= 2\,t^{1-\alpha}\sum_{k\ge 0} (-1)^k c_k t^{\alpha k}\) and  the coefficients $(c_k)_{k\geq0}$ are defined as follows: \(
			c_0=\frac{\Gamma(\alpha)^2}{\Gamma(2\alpha-1)\Gamma(2-\alpha)} \quad \textit{ and for every} \quad k\ge 1,\)
			{\small
				\begin{equation}\label{eq:ck2}
					c_k = \frac{\Gamma(\alpha)^2\Gamma(\alpha(k+1))}{\Gamma(2\alpha-1)\Gamma(\alpha k+2-\alpha)}\left[ (a*b)_k- \alpha(k+1)\sum_{\ell=1}^kB\big(\alpha(\ell+2)-1,\alpha(k-\ell-1)+2\big) (b^{*2})_{\ell}  c_{k-\ell}  \right].
				\end{equation}
			} 
			where $
			a_k=\frac{1}{\Gamma(\alpha k+1)}$, $b_k=\frac{1}{\Gamma(\alpha(k+1))},\; k\ge0,$ and for two sequences of real numbers \( (u_k)_{k \geq 0} \) and \( (v_k)_{k \geq 0} \), the Cauchy product is defined as \( (u * v)_k = \sum_{\ell = 0}^k u_\ell v_{k - \ell} \) and \( B(a, b) = \int_0^1 u^{a-1}(1 -
			u)^{b-1} du\) denoting the beta function.
			
			\item   The convergence radius \( \rho_\alpha = \left( \liminf_k \left( |c_k|^{1/k} \right) \right)^{\frac{-1}{\alpha}} \) of the power series \( \sum_{k \geq 0} c_k t^{\alpha k} \), defined by the coefficients \( c_k \), is infinite. Specifically, there exist constants \( K \geq 1 \) and \( A \geq 2^{\alpha+2} \) such that for all \( k \geq 0 \), the following inequality holds: \(|c_k| \leq \frac{K A^k}{\Gamma(\alpha(k-1) + 2 )}.\)
			As a consequence, the expansion in equation \ref{eq:varsigma_reduit} converges for all \( t \in \mathbb{R}^+ \), and in fact, for all \( t \in \mathbb{R} \).
		\end{enumerate} 
	\end{Proposition}
	\noindent\textbf{Remark.} 
	The equation in~\eqref{eq:ck2}, which provides the coefficients of the expansion 
	for \( \varsigma_{\alpha,\lambda,c}^2 \) when \( \alpha \in (1,\tfrac{3}{2}) \), 
	closely resembles that obtained for \( \alpha \in (\tfrac{1}{2},1) \) in~\cite{Pages2024}, 
	although the properties of the two functions differ significantly. 
	By the scaling property~\eqref{eq:varsigma_reduit}, we may assume now that \( c=\lambda=1 \).

	\begin{figure}[H]
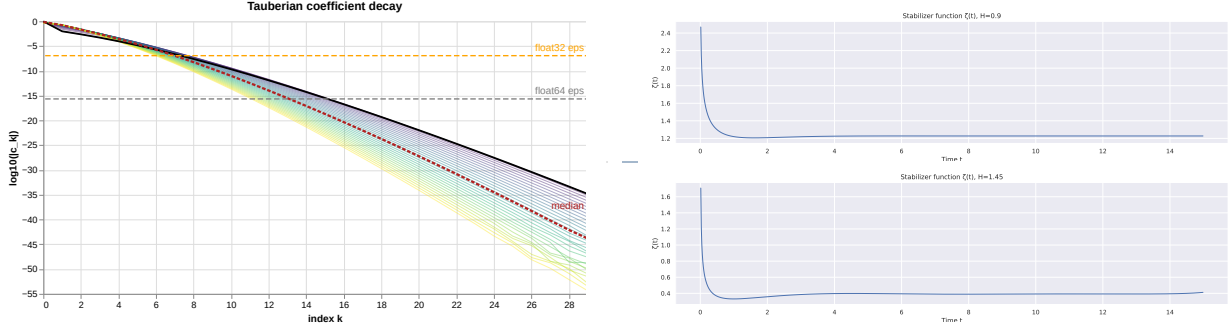

		\centering
		
		\begin{minipage}[c]{0.42\linewidth}
			\centering
			\includegraphics[width=1.2\linewidth]{Images/coeff_decay_lm.pdf}
		\end{minipage}
		\hfill
		\begin{minipage}[c]{0.55\linewidth}
			\centering
			\includegraphics[width=\linewidth]{Images/curves_stabilizer_H090_T15.pdf}
			\includegraphics[width=\linewidth]{Images/curves_stabilizer_H0145_T15.pdf}
		\end{minipage}
		\caption{\textit{
				Left: Tauberian series coefficients over the time interval
				\( [0,15] \) for Hurst exponents \( H\in(\frac12,1) \),
				with parameters \( \lambda=c=1 \) and \( n=600 \).
				Right: Stabilizer mappings
				\( t_k \mapsto \varsigma_{\alpha,\lambda,c}(t_k) \)
				over the time interval \( [0,15] \) for \( H=0.9 \) and \( H=1.45 \)
				with parameters \( \lambda=c=1 \) and \( n=600 \).
		}}
		\label{fig:coeff-and-stabilizers}
	\end{figure}
	The left panel displays $\log_{10}|c_k|$ as a function of $k$ for $\alpha \in [1.01,1.49]$.
	All curves fall below the float64 machine precision threshold, beyond which the values become indistinguishable from zero, well before $k=20$ in this long-memory regime. Therefore, truncating the series~\eqref{eq:varsigma_reduit}--\eqref{eq:ck2} at $n_{\max}=20$ (long-memory) is sufficient across the entire simulation range.
	
	The right panel depicts the stabilizing functions \(t \mapsto \varsigma_{\alpha,\lambda,c}(t)\) for different values of \(\alpha\in(1,\frac32)\).
		Note that, for \(\alpha\in(\frac12,1)\), these functions were studied in~\cite{Pages2024, EGnabeyeu2025} and were shown to be non-decreasing and typically concave. In the present long-memory setting, however, their behaviour appears to be markedly different.
		
	The numerical evidence provided by these figures suggests that the stabilizing function \(\varsigma^2_{\alpha,\lambda,c}\) is well-defined and strictly positive on \(\R_+\) (at least for \(\alpha\in(1,\frac32)\)). Since a full formal proof is rather involved, we instead prove in the sequel that the stabiliser $\varsigma^2_{\alpha,\lambda,c}$ exists as a non-negative function on a non-empty interval \( (0,T] \subseteq \mathbb{R}^+ \), for some \(T>0\).
	\begin{Proposition}[Existence of $\varsigma_{\alpha,\lambda,c}$ i.e. positivity computation of the function $\varsigma_{\alpha,\lambda,c}^2$ solution of the stabilizer equation for $\alpha \in (1,\frac32)$]\label{prop:alphaFractKernel2}
		Let $\alpha \in (1,\frac32)$ and consider the volterra equation of the first kind, 
		\begin{equation}\label{eq:stabil2}
			\kappa \, \left( 1 - R_{\alpha}^2(t) \right) = (f_{\alpha}^2 * g_{\alpha})(t), \quad \forall t \geq 0, \quad \kappa > 0.
		\end{equation}
		with \( R_{\alpha} : \mathbb{R}^+ \to \mathbb{R} \), \( f_{\alpha} := -R_{\alpha}' \) satisfy \( R_{\alpha}(0) = 1 \), \( \lim_{t\to +\infty}R_{\alpha} (t) = 0 \), and \( f_{\alpha}(0)=0 \), \( \lim_{t\to +\infty}f_{\alpha} (t) = 0\)
		\begin{enumerate}
			\item[(a)] Then equation ~\eqref{eq:stabil2} has at most one solution in \( {\cal L}^1_{\text{loc}}(\R_+) \) that converges to a finite limit.
			
			\item[(b)] If the equation ~\eqref{eq:stabil2} has a continuous solution \( g_{\alpha} \) defined on \( I \subseteq (0, +\infty) \) , then \( g_{\alpha} \geq 0 \) on \( I \subseteq \mathbb{R}^+ \), so that the function \( \sqrt{g_{\alpha}} \) is well-defined on \(I \subseteq \mathbb{R}_+ \).
		\end{enumerate}
	\end{Proposition}
	\noindent {\bf Proof.} The argument is similar to that of Proposition~\ref{prop:alphaFractKernel1_} further on, and is therefore left to the reader.

	\section {Applications to Exponential-Fractional Stochastic Volterra Equations}\label{sec:appl3}
	\vspace{-.2cm}
			Let consider the below {\em Gamma Fractional integration kernel} or {\em Exponential-Fractional integration kernel} defined in Example \ref{Ex:SolventGammaKernel}, where $\alpha = H + \frac{1}{2}$, with $H$ denoting the Hurst coefficient:
			
			\centerline{$
			K(t) = K_{\alpha, \rho}(t) = e^{-\rho t} \frac{u^{\alpha - 1}}{\Gamma(\alpha)} \mathbf{1}_{\mathbb{R}_+}(t), \quad  \text{with} \quad \alpha, \rho > 0.
			$}
			The purpose of this part is to extend the results of ~\cite{EGnabeyeu2025} to the general case of a gamma fractional integration kernel where $ \alpha \in (1, \frac{3}{2})$. Note that, this is a generalization of the exponential kernel and the fractional integration kernel.  
			The gamma kernel is often adopted in the Quadratic Rough Heston model (see, e.g., \cite{BourgeyGatheral2025}) due to its numerical convenience, flexibility, and the availability of a closed-form expression for its resolvent of the second kind. 
			
			\medskip
			We show that for such kernels $K_{\alpha, \rho}$, the resolvent $R_{\alpha,\rho,\lambda}$  satisfy our standing assumption $({\cal R}_\lambda)\; $ for all $\lambda > 0$, and that $f_{\alpha,\rho, \lambda} := -R_{\alpha,\rho,\lambda}$ exists and is square-integrable with respect to the Lebesgue measure on $\mathbb{R}_+$ whenever $1 < \alpha < \frac{3}{2}$ (``long memory volatility models'') and satisfies the \(\vartheta\)-H\"older regularity condition~\eqref{eq:Regulf}. As a result, the findings from  Sections \ref{subsec:fakestatioIntrin} and \ref{subsec:fakestatioStabil}, will be applicable in the cases where $\sigma(t, x) = \sigma(x)$ and $\sigma(t, x) = \varsigma_{\lambda,c}(t)\sigma(x)$.
			\vspace{-.2cm}
			\subsection{$\alpha-$ Exponential Fractional  kernels $1 <\alpha<\frac32$}\label{subsec:Gammaalphafrac}
			\begin{Proposition}\label{prop:main2} Let $\lambda>0$ and let $\alpha\!\in (1,2)$.
				\smallskip
				\noindent $(a)$ The $\lambda$-resolvent $R_{\alpha,\rho, \lambda} $ is infinitely differentiable 
				on $(0, +\infty)$.
				Moreover $R_{\alpha,\rho, \lambda}(0)=1$, $R_{\alpha,\rho, \lambda}$ converges  to $\frac{1}{1 + \lambda \rho^{-\alpha}} \in [0,1[$ and
				$R_{\alpha,\rho, \lambda}\!\in {\cal L}^{r}(\R_+)\;\forall\,r>\frac{1}{\alpha}$.  
				
				\smallskip
				\noindent $(b)$ $f_{\alpha,\rho, \lambda} := -R'_{\alpha, \rho, \lambda}$ is infinitely differentiable, bounded, converges to $0$, and satisfy :
				{\small
				\[ \forall\, t>0, \quad
				f_{\alpha,\rho, \lambda} (t):= e^{-\rho t}f_{\alpha, \lambda} (t) = \lambda^{\frac{1}{\alpha}}\int_0^{+\infty} e^{-(\rho+\lambda^{\frac{1}{\alpha}})tu}uH_{\alpha}(u)du - \frac2\alpha e^{t( \lambda^{\frac{1}{\alpha}}\cos \left( \frac{\pi}{\alpha} \right)-\rho)} \cos \left[ t \lambda^{\frac{1}{\alpha}} \sin \left( \frac{\pi}{\alpha} \right) - \frac{\pi}{\alpha} \right].
				\]
				}		
				\noindent $(c)$ Moreover, if  $\alpha\!\in (1, 2)$,  $ f_{\alpha,\rho, \lambda}$ is ${\cal L}^{p}$-integrable for every $p \in (0,+\infty] $ and for every $\vartheta\!\in\big(0,\alpha-1\big)$
				\[| f_{\alpha ,\rho,\lambda}(t+\delta)-f_{\alpha,\rho,\lambda}(t)  | \leq h_{\alpha,\rho,\lambda}(t) \delta^{\vartheta} \quad \text{with}\quad  h_{\alpha,\rho,\lambda} \in {\cal L}^1(\R_+) \cap {\cal L}^\infty(\R_+) \quad \text{and}\quad h_{\alpha,\rho,\lambda}*\varsigma^2_{\alpha,\rho,\lambda, c} \in {\cal L}^\infty(\R_+)\]
			\end{Proposition}
	\noindent For clarity and conciseness, the proof is postponed to Appendix \ref{app:lemmata}. 
	\begin{figure}[H]
		\centering
		\includegraphics[width=.95\textwidth]{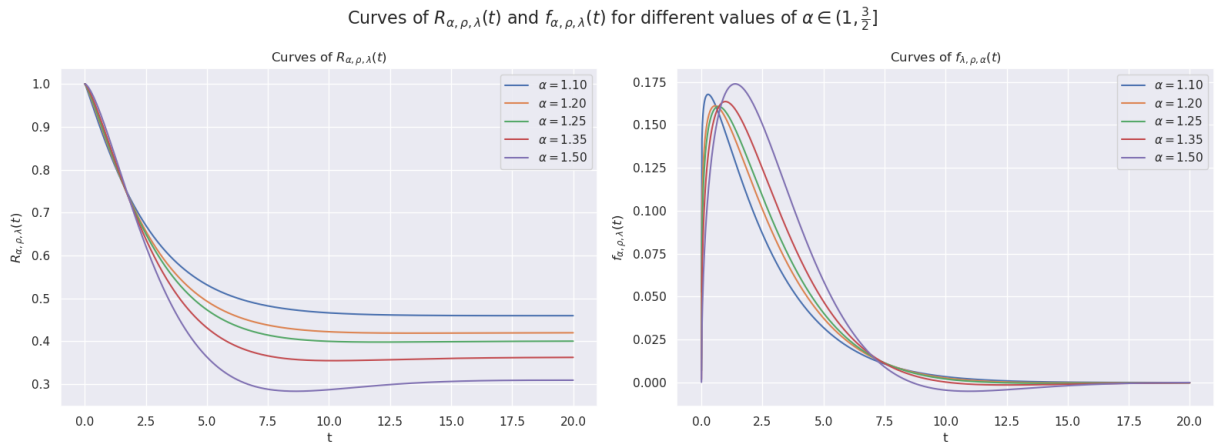} 
		\caption{Curves of $R_{\alpha,\rho,\lambda}(t)$ and $f_{\alpha,\rho, \lambda}(t)$ for different values of $\alpha \in (1,2)$, $\lambda= 0.25$ and
			$\rho = 0.2$}
		\label{fig:curves_alpha_more1}
	\end{figure}	
		\noindent{\bf Remark:} The remark following Proposition~\ref{prop:main} (or, equivalently,
		the claim~$(b)$ of the above proposition~\ref{prop:main2}) implies that if $\alpha\!\in (\frac12, \frac32)$, for each $p\geq1$ and every $\vartheta_p\!\in\big(0,(\alpha-1+\frac1p)\wedge1\big)$, 
		there exists  a constant $C_{\vartheta_p,\rho,\lambda}>0$ such that
		$f_{\alpha,\rho,\lambda}$ satisfies the following
		${\cal L}^p(\mathbb R_+)$-$\vartheta_p$-H\"older continuity estimate:
		\vspace{-.3cm}
		\begin{equation}\label{eq:thetaHolder2}
			\vspace{-.1cm}
			\forall\, \delta > 0, \quad 
			\left[\int_0^{+\infty} |f_{\alpha,\rho,\lambda}(t+\delta) - f_{\alpha,\rho,\lambda}(t) |^p \, dt \right]^{\frac1p} 
			\le C_{\vartheta_p,\rho,\lambda} \delta^{\vartheta_p}.
			\vspace{-.3cm}
		\end{equation}
		In fact, let $\delta >0$ and \(\vartheta_p\in(0,1]\),
		using that $0\le 1-e^{-v} \le (1-e^{-v})^{\vartheta_p} \le {v}^{\vartheta_p}$, for every $v\ge 0$, one may deduce
		{\small
			\begin{equation*}
				\begin{aligned}
					\left( \int_0^\infty \left| f_{\alpha,\rho,\lambda}(t+\delta) - f_{\alpha,\rho,\lambda}(t) \right|^p dt \right)^{1/p}
					&\leq e^{-\rho \delta}\left( \int_0^\infty \left| f_{\alpha, \lambda}(t+\delta) - f_{\alpha, \lambda}(t) \right|^p dt \right)^{\frac1p} + (\rho \delta )^{\vartheta_p} \left( \int_0^\infty \left| f_{\alpha, \lambda}(t) \right|^p dt \right)^{\frac1p} \\
					&\leq  e^{-\rho \delta} C_{\vartheta_p, \lambda}\delta^{\vartheta_p} + C_{f_\lambda}\delta^{\vartheta_p} := C_{\vartheta_p,\rho, \lambda}\delta^{\vartheta_p}.
				\end{aligned} 
			\end{equation*}
		}
		where in the last inequality, we used that \(f_{\alpha, \lambda} \in \mathcal{L}^p(\R_+)\) for each $p\geq1$ by/together with Proposition~\ref{prop:main}. 
	
	\medskip\noindent
	At this stage, we are in a position to state Theorem~\ref{thm:mainGamma}, whose proof is analogous to that of Theorem~\ref{thm:mainFractional} established for the fractional case and is therefore left to the reader.
			\subsection{Existence of the stabilizer $\varsigma_{\alpha,\rho,\lambda,c}$ i.e. positivity computation of $\varsigma_{\alpha,\rho,\lambda,c}^2$, 
				$\alpha \in (1,\frac32)$.}\label{subsec:sigma2Lmemory}
			\noindent In this section, we want to derive the asymptotic behaviour of $\varsigma_{\alpha,\rho,\lambda,c}$ near the origin, assuming that \(\mu(t) \sim \mu(0)\) as \(t\to 0\) with \(\frac{\mu(0)}{m_0}>0\).
			Following the same computations as in~\cite[Section 5.2]{EGnabeyeu2025},
			we are led to
			{\small
			\begin{equation}\label{eq:varsigma2}
				\varsigma^2_{\alpha,\rho,\lambda,c}(t) \stackrel{0}{\sim} \frac{2\lambda c \Gamma(\alpha)^2}{\Gamma(2\alpha-1) \Gamma(2-\alpha)} \frac{\mu(0)}{\lambda m_0} e^{-2 \rho t}t^{1-\alpha} \quad \text{so that}
			 \lim_{t \to 0^+} \varsigma_{\alpha,\rho,\lambda,c}(t) = +\infty.
			\end{equation}
			}
			This suggests 
		 that, there exists \(\eta\) small enough such that \(\forall t \in (0,\eta), \quad \varsigma_{\alpha,\rho, \lambda,c}^2(t) \approx e^{-2 \rho t} \varsigma_{\alpha, \lambda,c}^2(t).\)
	
	\medskip	
	\noindent {\bf Remark:\;} 
	1. 
	The computation of $\varsigma_{\alpha,\lambda,c}^2$ through a recurrence relation satisfied by the coefficients $(c_k)_{k\geq 0}$ as in section~\ref{subsec:sigma_Lmemory} is hindered by the need for an explicit expression of the mean-reverting function $\mu$.  We therefore compute $\varsigma_{\alpha,\lambda,c}^2$ by solving the corresponding functional equation numerically.\\
	2.  With that in mind, on a time grid \(t_k = k\frac{T}{n}, k = 0, . . . , n.\), we use the discretization \(\forall\, k\ge 1, \)
	\vspace{-.2cm}
					\begin{equation}\label{eq:TriangStabilScheme_}
						 c \lambda^2\Big(1-\big(1- \frac{(f_{\lambda} * \mu)((t_k))}{\lambda m_0}\big)^2 \Big) =  (f_{\alpha,\rho, \lambda}^2 * \varsigma_{\alpha,\rho, \lambda,c}^2)(t_k) = \sum_{j=0}^{k-1}\varsigma_{\alpha,\rho, \lambda,c}^2(t_{j+1})\int_{t_{k}-t_{j+1}}^{t_{k}-t_j} f_{\alpha,\rho, \lambda}^2(s)\, ds.
						\vspace{-.1cm}
					\end{equation}
					which we can solve recursively (Lower-Triangular system) to recover the values \(\varsigma_{\alpha,\rho, \lambda,c}^2(t_k),\, k\geq1.\)
					
			\smallskip
			\noindent From now on, we consider the case where \(\mu(t)=\mu_0=\lambda m_0\) for every \(t>0\), such that \(\phi \equiv 1\).
			\begin{Proposition}[Existence and Properties of the function $\varsigma_{\alpha,\rho,\lambda,c}^2$ for $\alpha \in  (1,\frac32)$]\label{prop:alphaFractKernel1_}
				Let $\alpha \in  (1,\frac32)$:
				\begin{enumerate}
					\item Consider the following equation for \(c,\lambda >0\):
					\begin{equation}\label{eq:stabil2}
						c \lambda^2 \left( 1 - R_{\alpha,\rho, \lambda}^2(t) \right) = (f_{\alpha,\rho, \lambda}^2 * g_{\alpha,\rho, \lambda})(t), \quad \forall t \geq 0.
					\end{equation}
					with \( R_{\alpha,\rho,\lambda} : \mathbb{R}^+ \to \mathbb{R} \) and \( f_{\alpha,\rho,\lambda} := -R_{\alpha,\rho,\lambda}' \) satisfy 
					\( \lim_{t\to +\infty}R_{\alpha,\rho,\lambda} (t) = a \),  \( \lim_{t\to +\infty}f_{\alpha,\rho,\lambda} (t) = 0 \).
					\begin{enumerate}
						\item[(a)] Then equation ~\eqref{eq:stabil2} has at most one solution in \( {\cal L}^1_{\text{loc}}(\R_+) \) that converges to a finite limit.
						
						\item[(b)] If the equation ~\eqref{eq:stabil2} has a continuous solution \( g_{\alpha,\rho,\lambda} \) defined on \( I \subseteq (0, +\infty) \), then \( g_{\alpha,\rho,\lambda} \geq 0 \) on \( I \subseteq \mathbb{R}^+ \), so that the function \( \sqrt{g_{\alpha,\rho,\lambda}} \) is well-defined on \( I \subseteq \mathbb{R}_+ \).

					\end{enumerate}
				\item The stabilizer \( \varsigma^2_{\alpha,\rho,\lambda,c} \) exists as a non-negative function on \(  I \subseteq (0, +\infty) \) and
				\vspace{-.2cm}
				\begin{equation} \lim_{t\to0} \varsigma_{\alpha,\rho, \lambda,c} = +\infty
				\quad \text{and} \quad 
				\lim_{t \to +\infty} \varsigma_{\alpha,\rho,\lambda,c}(t) = \frac{\sqrt{c(1-a^2)}\lambda}{\|f_{\alpha,\rho,\lambda}\|_{{\cal L}^2(\R_+)}}, \quad a= \frac{1}{1 + \lambda \rho^{-\alpha}} .
				\vspace{-.1cm}
				\end{equation}
				\end{enumerate} 
			\end{Proposition}
\noindent {\bf Proof.} Claim 1(a) comes from~\cite[Lemma 3.9]{EGnabeyeu2025}.
 Claim \((2)\) follows from 1(b), equation~\eqref{eq:varsigma2} and~\cite[Lemma 3.9]{EGnabeyeu2025}. The proof of 1(b) is postponed to Appendix \ref{app:lemmata}.
            \begin{figure}[H]
            	\centering
            	\begin{minipage}{0.48\linewidth}
            		\centering
            		\includegraphics[width=1\linewidth]{Images/Stabilizer_H072_T10.pdf}
            		\caption{ Graph of the stabilizer $ t \to \varsigma_{\alpha,\rho,\lambda,c}(t)$  over time interval [0, 10], for a value of the Hurst esponent $H=0.7$,  $\lambda = 0.2$, \( \rho = 1.2 \), c = 0.36.}
            	\end{minipage}
            	\hfill
            	\begin{minipage}{0.48\linewidth}
            		\centering
            		\includegraphics[width=1\linewidth]{Images/curves_stabilizer_H080Gamma_T10.pdf}
            		\caption{Graph of the stabilizer $ t \to \varsigma_{\alpha,\rho,\lambda,c}(t)$  over time interval [0, 10], for a value of the Hurst esponent $H=0.8$,  $\lambda = 0.2$, \( \rho = 1.2 \), c = 0.36.}
            	\end{minipage}
            	\vspace{-.3cm}
            \end{figure}
            These figures suggests the stabilizing function \(\varsigma^2_{\alpha,\rho,\lambda,c}\) is well-defined and positive on a non-empty interval \( (0,T] \subseteq \mathbb{R}^+ \), \(T>0\) (at least for \(\alpha\in(1,\frac32)\)).
            \vspace{-.3cm}
			\section{Numerical experiments with (stabilized) quadratic  diffusion coefficient
				}\label{Sec:Num}
				\vspace{-.2cm}
				In this section we specify a family of scaled volterra equations~\eqref{eq:Volterra} where $b(t,x) = \mu(t)-\lambda \, x$ for \( \lambda > 0 \) and \(\sigma(t,x) = \varsigma(t)\,\sigma(x)\) with \(\varsigma\) given by~\eqref{eq:VolterraStabilizerCont} and the state-dependent \(\sigma\) be a squared trinomial diffusion coefficient
				 of the form \ref{eq:sigmafakeI&II} given by:
				 \vspace{-.2cm}
				\begin{equation}\label{eq:trinom2}
					\sigma(x) = \sqrt{ \kappa_0 +\kappa_1\,(x-\frac{\mu_0}{\lambda})+\kappa_2\,(x-\frac{\mu_0}{\lambda})^2}\quad \mbox{ with }\quad  \kappa_i\ge 0,\;i=0,2, \;\kappa^2_1 \le 4\kappa_2\kappa_0.
					\vspace{-.1cm} 
				\end{equation}
			 In this case with affine drift, the equation~\eqref{eq:Volterra}-\eqref{eq:CondMean} reads with \(\phi\) as in~\eqref{eq:notcst_phi} reads:  
			 \vspace{-.2cm}
				\begin{equation}\label{eq:Gamma_lm_SVIE}
					X_t= g_0(t) +  \frac{1}{\lambda}\int_0^t f_{\alpha,\rho, \lambda}(t-s)\varsigma(s)\sigma(X_{s})dW_s,\;\; g_0(t) = X_0 - \frac{\big(X_0 - m_0\big)}{\lambda m_0} \int_0^t f_{\alpha,\rho, \lambda}(t-s) \mu(s)\dd s. 
					\vspace{-.1cm}
				\end{equation}
			\smallskip\noindent
			To simulate the volterra process, in contrast to multifactor approximation approaches, which are naturally
			suited to completely monotone kernels, we introduce a $f_{\alpha,\rho, \lambda}$-integrated discrete time Euler-Maruyama scheme providing a flexible and generic simulation framework
			for Volterra-type equations, including those driven by long-memory kernels.
			It is defined by Equation~\eqref{eq:VarianceModel} on the discrete  grid \((t^n_k)_{0\leq k\leq n}\) $t^n_k =\frac{kT}{n}, k=0, \dots, n$, and for every $k=1,\ldots,n$,   which write recursively:
			\vspace{-.2cm}
			\begin{equation}\label{eq:VarianceModel}
				\overline X_{t_{k}}^{n} 
				=  g_0(t_{k}) + \frac{1}{\lambda}\sum_{\ell=0}^{k-1}  \varsigma(t_{\ell+1})\sigma( \overline{X}^{n}_{t_{\ell}}) \int_{t_{\ell}}^{t_{\ell+1}} f_{\alpha,\rho, \lambda}(t^n_k-s) \,dW_s \qquad k=1,\ldots,n.
				\vspace{-.1cm}
			\end{equation}
			One has to deal with both the deterministic and stochastic integrals in the discretization.  Let us denote by $G= (G_{k\ell})_{ k=1:n+1, \ell=1: n}$ the  $(n+1)\times n$ matrix  involving the random terms $I^{n,\ell}_k:=\int_{t_{\ell}}^{t_{\ell+1}} f_{\alpha,\lambda}(t^n_k-s) \,dW_s$ reading column by column
				\vspace{-.2cm}
			\begin{equation}\label{eq:VCV}
				(G_{k\ell})_{k=1:n+1,\ \ell=1:n}= \left[\left( \begin{array}{l}
					\Delta W_{t_{\ell}}\,\mbox{\bf 1}_{\{k=\ell\},\, 
					}\\
					\int_{t_{\ell-1}}^{t_{\ell}} f_{\alpha,\lambda}(t_{k}-s)\,dW_s\,\mbox{\bf 1}_{\{k > \ell\}}
				\end{array} \right)_{k=1:n+1}\right]_{\ell=1:n}.
			\end{equation}
					where the first line has been introduced in order to be able to perform a consistent joint simulation of $\bar X^{n}$ and  the Euler-Maruyama scheme of any Markovian process depending on $\bar X^{n}$ and $W$ (see e.g.~\cite{Gnabeyeu2026a, Gnabeyeu2026b,dro2026optimal} 
					).
			
					\smallskip \noindent
					Then the following relation holds: $\;\overline X^{n}_{t_0}=g_0(t_{0})$ and for every $k=1,\dots,n$
						\begin{equation}\label{eq:VarianceModel2}
								 \overline X^{n}_{t_{k}}= g_0(t_{k}) +G_{k+1,\cdot}\big(\varsigma_{\alpha,\lambda,c}(t_{\ell+1})\sigma( \overline{X}^{n}_{t_{\ell}}))_{\ell=1:k}
							\end{equation}
					\noindent which is simulated on the discrete  grid \((t^n_k)_{0\leq k\leq n}\) by generating an independent sequence of gaussian vectors \( G_{\cdot,\ell}, \ell=0 \cdots n-1\)
					using an alternate extended and stable version of Cholesky decomposition of a well-defined covariance matrix \(C\).
					The reader is referred to \cite[Appendix A]{EGnabeyeu2025, GnabeyeuPages2026} for further details about the simulation of the Gaussian stochastic integrals terms in the semi-integrated Euler scheme~\eqref{eq:VarianceModel}-\eqref{eq:VarianceModel} introduced in this context for Equations~\eqref{eq:CondMean}-\eqref{eq:ConstMeanVar}. Theoretical guarantees for the convergence of this numerical scheme, as well as the convergence rate are established 
					for more general Kernels and path-dependent coefficients in~\cite{GnabeyeuPages2026} (see also~\cite{BRAS2025}).
				The reader is invited to take a look to the captions of the differents figures for the numerical values of the parameters of the Stochastic Volterra equation.
			\vspace{-.3cm}
			
			\subsection{Numerical illustration of intrinsic fake stationary SVIE with $\alpha$-fractional kernels 
			}
			We consider an $\alpha$-fractional kernel for $ \alpha \in (1, \frac32)$ (``long memory models ")  and a squared trinomial diffusion coefficient \(\sigma(t,x) = \sigma(x)\) with
			\(\sigma\) of the form \ref{eq:sigmafakeI&II} and given by \(\sigma(x) = \sqrt{ \kappa_0+\kappa_2\,(x-m_0)^2}\). 
			\begin{figure}[H]
				\centering
				\includegraphics[width=0.886\linewidth]{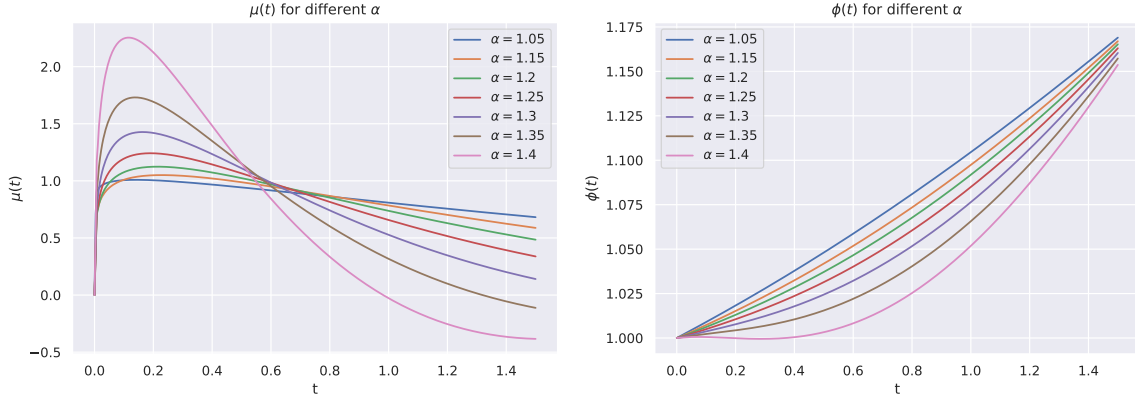}
				\caption{\textit{Graph of \( t \mapsto \mu(t) \) (left) and \( t \mapsto \phi(t) \) (right) for the fractional kernel (Example~\ref{Ex:SolventGammaKernel}~$-2.$) with different values of \(\alpha\) over the time interval \( [0, T] \), \( T = 1.5 \), \( \lambda = 0.2 \), \( m_0 = 10 \) and \(c=1.95356\).}}\label{fig:Mu_phi_autonome}
			\end{figure}
			\begin{figure}[H]
				\centering
				\includegraphics[width=0.856\linewidth]{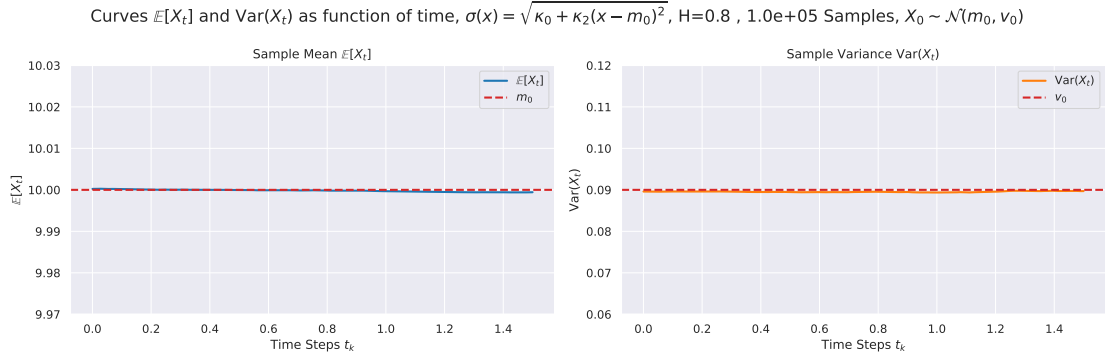}
				\caption{\textit{Graph of \( t_k \mapsto \mathbb{E}[X_{t_k}] \) and \( t_k \mapsto \text{Var}(X_{t_k}, M) \) for the process~\eqref{eq:ConstMeanVar} with \(\rho=0\) i.e. fractional kernel (Example~\ref{Ex:SolventGammaKernel}~$-2.$) over the time interval \( [0, T] \), \( T = 1 \), \( H = 0.8 \), \( \lambda = 0.2 \), \( m_0 = 10 \), \( v_0 = 0.09 \), \(c=1.95356\), \(\kappa_0 =0.01151\) and \(\kappa_2 =0.384\). Number of steps: \( n = 800 \), Simulation size: \( M = 160000 \).}}\label{fig:autonome_}
				\vspace{-.4cm}
			\end{figure}
			The Figure~\ref{fig:autonome_} illustrates that the Volterra process $(X_t)_{t\in[0,T]}$ in Equation~\eqref{eq:ConstMeanVar} is designed to ensure constant marginal
			mean and variance over time. The reader is referred to~\cite[Appendix~$A$]{EGnabeyeu2025} for
			further details about the semi-integrated Euler scheme  introduced in this context for the simulation of the stochastic integrals terms in Equation~\eqref{eq:ConstMeanVar}.
			\subsection{Numerical illustration of fake stationary SVIE with $\alpha$-fractional kernel $\alpha \in (1,\frac32)$} 
			We consider an $\alpha$-fractional kernel (\(\rho= 0\)) for $ \alpha \in (1, \frac{3}{2})$ (``Long Memory'')  and the squared trinomial diffusion coefficient~\eqref{eq:trinom2}. 	Let c be such that $c[\sigma]^2_{Lip} < 1$. 
			\begin{figure}[H]
				\centering
				\begin{minipage}{0.5\linewidth}
					\centering
					\includegraphics[width=.98\linewidth]{Images/curves_stabilizer_H080_T10.pdf}
					\caption{ Graph of the stabilizer $ t \to \varsigma_{\alpha,\lambda,c}(t)$  over [0,10]. Hurst esponent $H=0.8$,  $\lambda = 0.2$, c = 0.3.}
				\end{minipage}%
				\hfill
				\begin{minipage}{0.5\linewidth}
					\centering
					\includegraphics[width=0.9\linewidth]{Images/curves_trajectoriesConfluence_scheme1_H080.pdf}
					\caption{Confluence from a [0,30]-Uniform Distribution, T=60, $H=0.8$,  $\lambda = 0.2$, c = 0.36.}\label{fig:confluent1}
				\end{minipage}
			\end{figure}%
			\noindent
			Figure \ref{fig:confluent1} shows $L^2$-confluence (see~\cite[Proposition 4.4]{EGnabeyeu2025}) of the process's marginals for different initial values as time increases.
			\begin{figure}[H]
				\centering
				\includegraphics[width=0.9\linewidth]{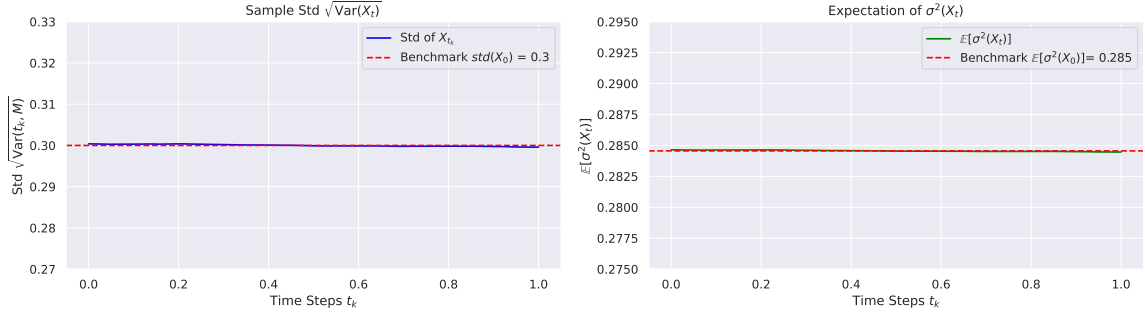}
				\caption{Graph of \( t_k \mapsto \text{StdDev}(t_k, M) \) and \( t_k \mapsto \mathbb{E}[\sigma^2(X_{t_k},M)] \) over \( [0, T] \), \( T = 1 \), \( H = 0.8 \), \( \mu_0 = 2 \), \( \lambda = 0.2 \), \( v_0 = 0.09 \), and \( \text{StdDev}(X_0) = 0.3 \). Number of steps: \( n = 800 \), Simulation size: \( M = 100000 \).}
			\end{figure}
			In this fractional case, the stabilizing function $\varsigma_{\alpha, \lambda,c}$, solution of the associated functional equations~\eqref{eq:VolterraStabilizer2} is computed using the power series expansion in Proposition~\ref{prop:alphaFractKernel1}.
			\subsection{A numerical illustration with $\alpha$-exponential fractional kernel for $\alpha \in (1,\frac32)$}
			\noindent
			We consider an $\alpha$-exponential fractional kernel for $ \alpha \in (1, \frac{3}{2})$ (``Long Memory'')  and the squared trinomial diffusion coefficient~\eqref{eq:trinom2}. 
		
		\medskip\noindent
			Note that, in this exponential-fractional case, the stabilizing function $\varsigma_{\alpha,\rho, \lambda,c}$, solution of the associated functional equations~\eqref{eq:VolterraStabilizer2}--\eqref{eq:VolterraStabilizer} is computed using the scheme~\eqref{eq:TriangStabilScheme_}.
			\begin{figure}[H]
				\centering
				\includegraphics[width=0.83\linewidth]{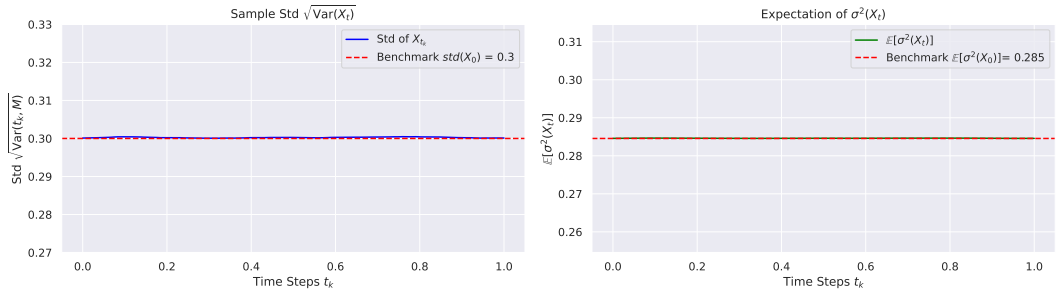}
				\caption{Graph of \( t_k \mapsto \text{StdDev}(t_k, M) \) and \( t_k \mapsto \mathbb{E}[\sigma^2(X_{t_k},M)] \) over the time interval \( [0, 1] \), \( H = 0.8 \), \( \mu_0 = 2 \), \( \lambda = 0.2 \), \( v_0 = 0.09 \) and \( \rho = 1.2 \). Number of steps: \( n = 800 \), Simulation size: \( M = 10^5 \).}
				\vspace{-0.5cm}
			\end{figure}
	
			\smallskip
			\noindent {\bf Acknowledgement:}  The authors thank J-F. Chassagneux, M. Rosenbaum and Saad Souilmi
			 for insightful discussions, help and comments.
			\vspace{-0.5cm}
			\bibliographystyle{abbrv}
			\bibliography{StationarityVolterraEquations}
			
			\normalsize  
			\appendix
			\vspace{-0.1cm}
			\section{Auxiliary results and Technical Proofs}\label{app:lemmata}
				\begin{Lemma}[Expansions]\label{lm:expansion}
		We have the following inequalities:
		\begin{enumerate}
			\item $0\le 1-e^{-v} \le (1-e^{-v})^{\vartheta} \le {v}^{\vartheta}$, for every $v\ge 0$, and $\vartheta\!\in (0,1]$.
			\item $\sin(v) \le {v}^{\vartheta}$, for every $v\ge 0$, and $\vartheta\!\in (0,1]$.
		\end{enumerate}
	\end{Lemma}
	\noindent {\bf Proof.} The claim (1) is straightforward since $\vartheta\!\in (0,1)$, while for the second claim, we have:
	\begin{itemize}
		\item if $0\le v \le 1$, then $\sin(v) \le v \le {v}^{\vartheta}$, for every $\vartheta\!\in (0,1]$.
		\item  if $v \ge 1$, then  ${v}^{\vartheta} \ge 1 \ge\sin(v) $, for every $\vartheta\!\in (0,1]$.
	\end{itemize}
	\bigskip
	\noindent {\bf Proof of Proposition \ref{prop:main_general alpha}:} 
	\medskip
	\noindent  {\sc Step~1.} 
	As \( \forall \alpha\!\in \R^*_+ \setminus \mathbb{N} \) , \( (-1)^{\lfloor \alpha \rfloor}\sin(\alpha \pi) > 0 \), we have the inequality:
	{\small
		\begin{equation}\label{eq:ineqHalpha}
			u^{2\alpha} + 2u^{\alpha} \cos(\pi \alpha) + 1 \ge 1 - \cos^2(\alpha \pi) = \sin^2(\alpha \pi) > 0 \quad (\textit{or} \quad \ge (u^{\alpha} -1)^2 > 0 ).
		\end{equation}
	}
	i.e., \( (-1)^{\lfloor \alpha \rfloor}H_{\alpha}(u) \) is non-negative for all \( u \) in the integral \ref{eq:Laplace_transform}. Therefore, \( (-1)^{\lfloor \alpha \rfloor}F_{\alpha}(t) \) is the Laplace transform of a non-negative Lebesgue integrable function \( (-1)^{\lfloor \alpha \rfloor}H_{\alpha} : \mathbb{R}_+ \to \mathbb{R}_+ \), and, by the ''Bernstein theorem'', \( (-1)^{\lfloor \alpha \rfloor}F_{\alpha}(t) \) is \textit{completely monotone} (CM) in the real line,in the sense that $(-1)^n (-1)^{\lfloor \alpha \rfloor}F^{(n)}_\alpha(t) \ge 0$ at every order $n\ge 0$.
	However, the CM property of \( (-1)^{\lfloor \alpha \rfloor}F_\alpha(t) \) can also be seen as a consequence of the result by Pollard \cite{SchillingSongVondracek2012} because the transformation \( x = t^\alpha \) is a Bernstein function for \( \alpha \in (0, 1) \). 
	
	\noindent  {\sc Step~2.} Moreover as $H_\alpha$ is continuous on $(0,+\infty)$, $
	H_\alpha(u)\stackrel{0}{\sim}  u^{\alpha-1}\frac{\sin(\pi \alpha)}{\pi}
	\quad \mbox{ and }\quad H_\alpha(u)\stackrel{+\infty}{\sim} \frac{\sin(\pi \alpha)}{\pi u^{\alpha+1}}.$\\
	It is clear that $H_\alpha \!\in {\cal L}_{\R_+}^1(\R_+)$ and that both functions \(u\mapsto uH_\alpha(u) \) and \(u\mapsto u^{\alpha+1}H_\alpha(u)\)  are bounded on $\R_+$.
	Thus, for every $t >0$, $\int_0^{+\infty} e^{-t u}uH_\alpha (u)du <+\infty$ so that owing to a Lebesgue-type condition for differentiation under the integral sign, $F_\alpha$ is differentiable  on $(0, +\infty)$ with 
	{\small
		\begin{equation}\label{eq:derivF}
			F'_\alpha(t) = -\int_0^{+\infty} e^{-tu} uH_\alpha(u) du, \quad t>0.
		\end{equation}
	}
	The same rule applied $k$ times shows that $F_\alpha$ is $\mathcal C^k$ for $k \in \N$, hence is infinitely differentiable and
	{\small
		\begin{equation}\label{eq:derivF_k}
			F^{(k)}_\alpha(t) = \int_0^{+\infty} e^{-tu} H^{(k)}_{\alpha}(u) \, du \; \text{with} \; H^{(k)}_{\alpha}(u) := (-1)^k u^{k} H_{\alpha}(u) = (-1)^k \frac{\sin(\alpha \pi) }{ \pi } \frac{u^{\alpha - 1 + k}}{u^{2\alpha} + 2u^{\alpha} \cos(\alpha \pi) + 1}.
		\end{equation}
	}
	$G_\alpha(t)$ is infinitely differentiable($\mathcal C^k$ for $k \in \N$) as product of such functions and by recurrence, we have:
	{\small
		\begin{equation}\label{eq:derivG}
			\forall k \in \N, \quad G^{(k)}_\alpha(t) =\frac{2}{\alpha} \sum_{n=0}^{\lfloor \frac{\alpha-1}{2} \rfloor} \exp\left[t \cos\left(\frac{(2n+1)\pi}{\alpha}\right)\right] \cos\left[t \sin\left(\frac{(2n+1)\pi}{\alpha}\right)- \frac{k (2n+1)\pi}{\alpha}\right].
		\end{equation}
	}
	Claim $(b)$ follows from the fact that $R_{\alpha,\lambda}= e_{\alpha}(\lambda^{1/\alpha}\cdot)  = R_{\alpha,1}(\lambda^{1/\alpha}\cdot) $, hence infinitely differentiable  on $(0, +\infty)$ from \ref{eq:derivF_k} and \ref{eq:derivG}. 
	The representation of $f_{\alpha,\lambda}$ follows from \ref{eq:derivF} and \ref{eq:derivG}.
	
	It follows  from~\eqref{eq:Halpha} and~\eqref{eq:ineqHalpha}  that \(|	H_\alpha(u)|  \le \frac{u^{\alpha-1}|\sin(\pi \alpha)|}{\pi \sin^2(\pi \alpha)} =  \frac{u^{\alpha-1}}{\pi |\sin(\pi \alpha)|}.\)
	Hence, for every $t\ge 0$,  
	{\small
		\begin{align*}
			|R_{\alpha,1}(t)| = |e_{\alpha}(t)| &= |F_\alpha(t) +G_\alpha(t)| < \frac{1}{\pi |\sin(\pi \alpha)|}\int_{0}^{+\infty}e^{-tu}u^{\alpha-1}du + \frac{2}{\alpha} \sum_{n=0}^{\lfloor \frac{\alpha-1}{2} \rfloor} \exp\left[t \cos\left(\frac{(2n+1)\pi}{\alpha}\right)\right]\\
			&\leq \frac{\Gamma(\alpha)}{\pi |\sin(\pi \alpha)|}t^{-\alpha} +\frac{2}{\alpha} \sum_{n=0}^{\lfloor \frac{\alpha-1}{2} \rfloor} \exp\left[t \cos\left(\frac{(2n+1)\pi}{\alpha}\right)\right]\leq \frac{\Gamma(\alpha)}{\pi |\sin(\pi \alpha)|}t^{-\alpha} +\frac{\lfloor \alpha+1 \rfloor}{\alpha} e^{t \cos\left(\frac{\pi}{\alpha}\right)}
		\end{align*}
	}
	\noindent
	where the last inequality comes from the fact that \(\cos(x) \) is non-increasing on \([0, \pi]\)
	so that $R_{\alpha,1}\!\in {\cal L}^{\gamma}(\R_+)$ for every $\gamma>\frac{1}{\alpha}$ where $\alpha$ is such that $\cos \left( \frac{\pi}{\alpha} \right)<0$ i.e. $\alpha \in (0,2]$ . This extends to $R_{\lambda, \alpha}$ by scaling.
	For the $L^{2\beta}$-integrability of $f_{\alpha, \lambda}$, once noted that $f_{\alpha,\lambda}= \lambda^{1/\alpha} f_{\alpha,1}(\lambda^{1/\alpha} \cdot)$ so that $\int_0^{+\infty} f_{\alpha, \lambda}^{2\beta}(t)dt= \lambda^{\frac{2\beta-1}{\alpha}}\int_0^{+\infty} f_{\alpha,1}^{2\beta}(t)dt$, it is clear that it is enough to prove that $ f_{\alpha,1}$ is ${\cal L}^{2\beta}$-integrable.
	
	\smallskip By the same argument as above, it follows  from~\eqref{eq:derivF}  and ~\eqref{eq:derivG} that for every $t\geq0$
	{\small
		\[|f_{\alpha, 1}(t)| < \frac{1}{\pi |\sin(\pi \alpha)|} \int_0^{+\infty} e^{-tu}u^{\alpha}  du + \frac{2}{\alpha} \sum_{n=0}^{\lfloor \frac{\alpha-1}{2} \rfloor} \exp\Big[t \cos\Big(\frac{(2n+1)\pi}{\alpha}\Big)\Big] \leq \frac{\Gamma(\alpha+1)}{t^{\alpha+1}\pi |\sin(\pi \alpha)|} + \frac{\lfloor \alpha+1 \rfloor}{\alpha} e^{t \cos\Big(\frac{\pi}{\alpha}\Big)}.\]
	}
	
	Thus $f_{\alpha, 1}\!\in {\cal L}^{2\beta}([1,+\infty), {\rm Leb}_1) \quad \forall \beta >0$ provided that $\cos \Big( \frac{\pi}{\alpha} \Big)<0$ i.e. $\alpha \in (0,2)$.
	On the other hand  \(f_{\alpha, \lambda}(t) = - R'_{\alpha, \lambda}(t) = \alpha\lambda t^{\alpha-1} E'_{\alpha}(-\lambda t^{\alpha})  = \lambda t^{\alpha-1}\sum_{k\ge 0}(-1)^k\lambda^k \frac{t^{\alpha k}}{\Gamma(\alpha (k+1))}\)
	so that \(f_{\alpha,1}(t)\stackrel{0}{\sim} 
	\frac{t^{\alpha-1}}{\Gamma(\alpha)}.\)
	As $t\mapsto \frac {1}{t^{1-\alpha}}\!\in {\cal L}^{2\beta}((0,1],{\rm Leb}_1)$ for any $\beta\!\in \big( \frac{1}{2(1-\alpha)}, +\infty\big)$ , we conclude that $f_{\alpha,1}\!\in {\cal L}^{2\beta}(\R_+) \quad \forall \beta >0$ provided that $\cos \Big( \frac{\pi}{\alpha} \Big)<0$ i.e. $\alpha \in (0,2)$.
	
	\medskip
	\noindent  {\sc Step~3.}
	As for the $\vartheta$-H\"older continuity of $f_{\alpha, \lambda}$, we may again assume without loss of generality that $\lambda =1$. Let $\delta >0$, \(\vartheta\in(0,1]\) and $\alpha\!\in (1,+\infty) \setminus \mathbb{N}$. One has
	\vspace{-.2cm}
	{\small 
		\begin{equation*}
		f_{\alpha ,1}(t+\delta)-f_{\alpha,1}(t) = \Big(F'_\alpha(t)-F'_\alpha(t+\delta)\Big) + \Big(G'_\alpha(t)-G'_\alpha(t+\delta)\Big) =  \Big(F'_\alpha(t)-F'_\alpha(t+\delta)\Big) + \sum_{n=0}^{\lfloor \frac{\alpha-1}{2} \rfloor} \Big(G'^n_\alpha(t)-G'^n_\alpha(t+\delta)\Big)
		\vspace{-.2cm}
		\end{equation*}
	}
	However, bearing in mind that  \( 0 \leq \frac{\pi}{\alpha} \leq \frac{(2n+1)\pi}{\alpha} \leq \pi \) for $\alpha\!\in (1,+\infty) \setminus \mathbb{N}$ and $0 \leq n \leq \lfloor \frac{\alpha-1}{2} \rfloor$ ,  we have:
	{\small 
		\begin{align*}
			&G'^n_\alpha(t)-G'^n_\alpha(t+\delta)= \\
			& \frac2\alpha e^{t \cos \Big( \frac{(2n+1)\pi}{\alpha} \Big)} \Big( \cos \Big[ t \sin \Big( \frac{(2n+1)\pi}{\alpha} \Big) - \frac{(2n+1)\pi}{\alpha} \Big]-e^{\delta \cos \Big( \frac{(2n+1)\pi}{\alpha} \Big)} \cos \Big[ (t+\delta) \sin \Big( \frac{(2n+1)\pi}{\alpha} \Big) - \frac{(2n+1)\pi}{\alpha} \Big]\Big) \\
			&= \frac{2}{\alpha} e^{t \cos \Big( \frac{(2n+1)\pi}{\alpha} \Big)} \Big( \cos \Big[ t \sin \Big( \frac{(2n+1)\pi}{\alpha} \Big) - \frac{(2n+1)\pi}{\alpha} \Big] - \cos \Big[ (t+\delta) \sin \Big( \frac{(2n+1)\pi}{\alpha} \Big) - \frac{(2n+1)\pi}{\alpha} \Big] \Big) \\
			&\quad\hspace{4cm} + \frac{2}{\alpha} e^{t \cos \Big( \frac{(2n+1)\pi}{\alpha} \Big)} \Big( 1 - e^{\delta \cos \Big( \frac{(2n+1)\pi}{\alpha} \Big)} \Big) \cos \Big[ (t+\delta) \sin \Big( \frac{(2n+1)\pi}{\alpha} \Big) - \frac{(2n+1)\pi}{\alpha} \Big]\\
			&= \frac4\alpha e^{t \cos \Big( \frac{(2n+1)\pi}{\alpha} \Big)} \sin \Big[ \frac{\delta}{2} \sin \Big( \frac{(2n+1)\pi}{\alpha} \Big) \Big] \sin \Big[ -(t+\frac{\delta}{2}) \sin \Big( \frac{(2n+1)\pi}{\alpha} \Big) + \frac{(2n+1)\pi}{\alpha} \Big]
			\\
			&\quad\hspace{4cm} + \frac2\alpha e^{t \cos \Big( \frac{(2n+1)\pi}{\alpha} \Big)} \Big( 1 - e^{\delta \cos \Big( \frac{(2n+1)\pi}{\alpha} \Big)} \Big) \cos \Big[ (t+\delta) \sin \Big( \frac{(2n+1)\pi}{\alpha} \Big) - \frac{(2n+1)\pi}{\alpha} \Big]
		\end{align*}
		Consequently,
		
		\begin{align*}
			&|G'^n_\alpha(t)-G'^n_\alpha(t+\delta)|\leq \frac2\alpha e^{t \cos \Big( \frac{(2n+1)\pi}{\alpha} \Big)} \Big( 2\sin \Big[ \frac{\delta}{2} \sin \Big( \frac{(2n+1)\pi}{\alpha} \Big) \Big] + (1-e^{\delta \cos \Big( \frac{(2n+1)\pi}{\alpha} \Big)} )\Big) \\
			&\hspace{3cm}\leq \frac2\alpha e^{t \cos \Big( \frac{(2n+1)\pi}{\alpha} \Big)} \Big( 2\sin \Big[ \frac{\delta}{2} \sin \Big( \frac{(2n+1)\pi}{\alpha} \Big) \Big] + (1-e^{\delta \cos \Big( \frac{(2n+1)\pi}{\alpha} \Big)} )\Big) \\
			&\hspace{2cm}\leq \frac2\alpha e^{t \cos \Big( \frac{\pi}{\alpha} \Big)} \Big( 2\Big(\frac{\delta}{2} \pi \Big)^{\vartheta} + (1-e^{-\delta } ) \Big) \leq \frac2\alpha e^{t \cos \Big( \frac{\pi}{\alpha} \Big)} \Big( 2^{1-\vartheta} \pi^\vartheta \delta^\vartheta + \delta^\vartheta \Big)= \frac2\alpha e^{t \cos \Big( \frac{\pi}{\alpha} \Big)} \Big( 2^{1-\vartheta} \pi^\vartheta  + 1 \Big)\delta^\vartheta. 
		\end{align*}
	}
	The penultimate inequality follows from the fact that \( 0 \leq \frac{\pi}{\alpha} \leq \frac{(2n+1)\pi}{\alpha} \leq \pi \), which leads to two key observations. On the first hand, as we always have \(\cos \Big( \frac{(2n+1)\pi}{\alpha} \Big)\geq-1\)  we have \( 1 - e^{\delta \cos \Big( \frac{(2n+1)\pi}{\alpha} \Big)} \leq 1 - e^{-\delta},\)
	and on the other hand, by applying Lemma \ref{lm:expansion} (2), we obtain the following inequality:
	\vspace{-0.2cm}
	\begin{equation}
		\sin \Big[ \frac{\delta}{2} \sin \big( \frac{(2n+1)\pi}{\alpha} \big) \Big] \leq \Big( \frac{\delta}{2} \sin \big( \frac{(2n+1)\pi}{\alpha} \big) \Big)^{\vartheta} \leq \Big( \frac{\delta}{2} \big( \frac{(2n+1)\pi}{\alpha} \big) \Big)^{\vartheta} \leq \big( \frac{\delta}{2} \pi \big)^{\vartheta}.
	\vspace{-0.2cm}
	\end{equation}
	where the final inequality follows from Lemma \ref{lm:expansion} (1).
	Consequently, H\"older regularity with exponent \( \vartheta \) for the function \( f_{\alpha,\lambda} \)  of the form~\eqref{eq:Regulf} can be achieved provided that \( \cos \Big( \frac{\pi}{\alpha} \Big) < 0 \), i.e., \( \alpha \in (1,2) \).
	
	\bigskip
	\noindent Now, about the $\alpha$-fractional kernels with $1 <\alpha<2$, it follows from Proposition\ref{prop:representation}, (see also \cite{GorenfloMainardi1997}) that:
	\vspace{-0.2cm}
	\begin{equation}
		e_{\alpha}(t) = F_{\alpha}(t) + G_{\alpha}(t) = \int_0^{+\infty} e^{-tu} H_{\alpha}(u) \, du  + \frac{2}{\alpha} e^{t \cos \Big( \frac{\pi}{\alpha} \Big)} \cos \Big( t \sin \Big( \frac{\pi}{\alpha} \Big) \Big), \quad 1 < \alpha < 2, \quad t \geq 0. \label{eq:main_relation}
		\vspace{-0.2cm}
	\end{equation}
	Note that, in this case ($1 <\alpha<2$), the function \(H_{\alpha}(u) \) is negative for all u (and thus -F is completely monotone and hence infinitely differentiable on $\R_+^+$)
	since \( 1 < \alpha < 2 \) implies \( \sin(\alpha \pi) < 0 \) and we have the following inequality: $
	u^{2\alpha} + 2u^{\alpha} \cos(\pi \alpha) + 1 \ge 1 - \cos^2(\alpha \pi) = \sin^2(\alpha \pi) > 0 \; (\textit{or} \; \ge (u^{\alpha} -1)^2 > 0 ).$ \hfill $\square$
	
	\bigskip
	\noindent {\bf Proof of Proposition \ref{prop:main}.}
	\noindent  {\sc Step~1.}
	$(a)$ follows from the first claim of Proposition \ref{prop:main_general alpha} since $R_{\alpha,\lambda}= e_{\alpha}(\lambda^{1/\alpha}\cdot)  = R_{\alpha,1}(\lambda^{1/\alpha}\cdot) $, hence infinitely differentiable  on $(0, +\infty)$ from \ref{eq:derivF_k} and \ref{eq:derivG}. All will extend to $R_{\alpha, \lambda}$ by scaling.
	It follows  from~\eqref{eq:Halpha} and~\eqref{eq:ineqHalpha} that \(|H_\alpha(u)| \le \frac{u^{\alpha-1}|\sin(\pi \alpha)|}{\pi \sin^2(\pi \alpha)} =  \frac{u^{\alpha-1}}{\pi |\sin(\pi \alpha)|}.\)
	Hence, for every $t\ge 0$,
	{\small  
	\vspace{-0.2cm}
	\begin{equation}
	|R_{\alpha,1}(t)| = | F_\alpha(t) +G_\alpha(t) |\le \frac{1}{\pi |\sin(\pi \alpha)|}\int_{0}^{+\infty}e^{-tu}u^{\alpha-1}du + \frac2\alpha e^{t\cos \Big( \frac{\pi}{\alpha} \Big)}= \frac{\Gamma(\alpha)}{\pi |\sin(\pi \alpha)|}t^{-\alpha} +\frac2\alpha e^{t\cos \Big( \frac{\pi}{\alpha} \Big)}.
	\vspace{-0.2cm}
    \end{equation}
    }
	so that $R_{\alpha,1}\!\in {\cal L}^{\gamma}(\R_+)$ for every $\gamma>\frac{1}{\alpha}$ as $\cos \Big( \frac{\pi}{\alpha} \Big)<0,\,\forall\,\alpha \in (1,2)$ and in particular $R_{\alpha,1}(t) \leq 1 \quad \forall t\geq0$ since $\sin(\pi \alpha) \leq 0$ . The representation of $f_{\alpha,\lambda}$ in $(b)$  follows from \eqref{eq:derivF} and \eqref{eq:derivG_k}.
	
	\medskip
	\noindent  {\sc Step~2.}
	 $(b)$ Let us prove the $L^{p}$-integrability of $f_{\alpha, \lambda}$. Once noted that $f_{\alpha,\lambda}= \lambda^{1/\alpha} f_{\alpha,1}(\lambda^{1/\alpha} \cdot)$ so that $\int_0^{+\infty} f_{\alpha, \lambda}^{p}(t)dt= \lambda^{\frac{p-1}{\alpha}}\int_0^{+\infty} f_{\alpha,1}^{p}(t)dt$, it is clear that it is enough to prove that $ f_{\alpha,1}$ is ${\cal L}^{p}$-integrable.
	
	\medskip
	\noindent It follows from~\eqref{eq:derivF}  and ~\eqref{eq:derivG_k} that for every $t>0$,
	{\small
	\vspace{-0.2cm}
	\begin{equation}
	|f_{\alpha, 1}(t)| =|-F'_\alpha(t) -G'_\alpha(t)| \le \frac{1}{\pi |\sin(\pi \alpha)|} \int_0^{+\infty} e^{-tu}u^{\alpha}  du + \frac2\alpha e^{t\cos \Big( \frac{\pi}{\alpha} \Big)}= \frac{\Gamma(\alpha+1)}{t^{\alpha+1}\pi \sin(\pi \alpha)} + \frac2\alpha e^{t\cos \Big( \frac{\pi}{\alpha} \Big)}.
	\vspace{-0.2cm}
	\end{equation}
	}
	Thus $f_{\alpha, 1}\!\in {\cal L}^{p}([1,+\infty), {\rm Leb}_1) \; \forall p >0$ and is bounded in \([1,+\infty)\).
	On the other hand  \(f_{\alpha, \lambda}(t) = \lambda t^{\alpha-1}\sum_{k\ge 0}(-1)^k\lambda^k \frac{t^{\alpha k}}{\Gamma(\alpha (k+1))}\)
	so that \(f_{\alpha,1}(t)\stackrel{0}{\sim} 
	\frac{t^{\alpha-1}}{\Gamma(\alpha)},\) \(f_{\alpha,1}(0)=0\) and \(f_{\alpha,1}\) is bounded on \([0,1]\). Moreover, since
	$t\mapsto \frac {1}{t^{1-\alpha}}\!\in {\cal L}^{p}((0,1],{\rm Leb}_1)$ for any $p\!\in \big( \frac{1}{(1-\alpha)}, +\infty\big) \cap \R_+^* = \R_+^* $, we conclude that $f_{\alpha,1}\!\in {\cal L}^{p}(\R_+) \quad \forall p \in (0,+\infty]$. Another consequence is that, for every \( t \geq 1 \),
	\(	|R_{\alpha, 1}(t)| = |e_{\alpha}(t)| = \int_t^{+\infty} | f_{\alpha, 1}(s)| \, ds \leq C_{\alpha}^\prime \, t^{-\alpha} + C_{\alpha}^{\prime\prime} \, e^{t\cos \Big( \frac{\pi}{\alpha} \Big)},\)
	so that \( R_{\alpha, 1} \in {\cal L}^2(\R_+) \).
	
	\medskip
	\noindent  {\sc Step~3.}
	As for the $\vartheta$-H\"older continuity of $f_{\alpha, \lambda}$, one may again assume w.l.g. that $\lambda =1$. Let $\delta >0$ and \(\vartheta\in(0,1]\). One has \(f_{\alpha ,1}(t+\delta)-f_{\alpha,1}(t) = \Big(F'_\alpha(t)-F'_\alpha(t+\delta)\Big) + \Big(G'_\alpha(t)-G'_\alpha(t+\delta)\Big)\) and following the same reasoning as above while bearing in mind that $\cos \Big( \frac{\pi}{\alpha} \Big) \leq 0 $, $\sin \Big( \frac{\pi}{\alpha} \Big) \geq 0 $ for $\alpha \in (1,2)$, we have:
	\vspace{-0.2cm}
	{\small
		\begin{align*}
			&G'_\alpha(t)-G'_\alpha(t+\delta) = \frac2\alpha e^{t \cos \Big( \frac{\pi}{\alpha} \Big)} \Big( \cos \Big[ t \sin \Big( \frac{\pi}{\alpha} \Big) - \frac{\pi}{\alpha} \Big]-e^{\delta \cos \Big( \frac{\pi}{\alpha} \Big)} \cos \Big[ (t+\delta) \sin \Big( \frac{\pi}{\alpha} \Big) - \frac{\pi}{\alpha} \Big]\Big) \nonumber\\
			&= \frac2\alpha e^{t \cos \Big( \frac{\pi}{\alpha} \Big)} \Big( (\cos \Big[ t \sin \Big( \frac{\pi}{\alpha} \Big) - \frac{\pi}{\alpha} \Big] -\cos \Big[ (t+\delta) \sin \Big( \frac{\pi}{\alpha} \Big) - \frac{\pi}{\alpha} \Big])+ (1-e^{\delta \cos \Big( \frac{\pi}{\alpha} \Big)} )\cos \Big[ (t+\delta) \sin \Big( \frac{\pi}{\alpha} \Big) - \frac{\pi}{\alpha} \Big]\Big)\nonumber\\
			&= \frac2\alpha e^{t \cos \Big( \frac{\pi}{\alpha} \Big)} \Big( 2\sin \Big[ \frac{\delta}{2} \sin \Big( \frac{\pi}{\alpha} \Big) \Big] \sin \Big[ -(t+\frac{\delta}{2}) \sin \Big( \frac{\pi}{\alpha} \Big) + \frac{\pi}{\alpha} \Big]+ (1-e^{\delta \cos \Big( \frac{\pi}{\alpha} \Big)} )\cos \Big[ (t+\delta) \sin \Big( \frac{\pi}{\alpha} \Big) - \frac{\pi}{\alpha} \Big]\Big) 
		\end{align*}
		Consequently,
		\vspace{-0.2cm}
		\begin{align}
			|G'_\alpha(t)-G'_\alpha(t+\delta)| &\leq \frac2\alpha e^{t \cos \Big( \frac{\pi}{\alpha} \Big)} \Big( 2\sin \Big[ \frac{\delta}{2} \sin \Big( \frac{\pi}{\alpha} \Big) \Big] + (1-e^{\delta \cos \Big( \frac{\pi}{\alpha} \Big)} )\Big) \leq \frac2\alpha e^{t \cos \Big( \frac{\pi}{\alpha} \Big)} \Big( 2\Big(\frac{\delta}{2} \frac{\pi}{\alpha} \Big)^{\vartheta} + (1-e^{-\delta } ) \Big) \nonumber\\
			&\leq \frac2\alpha e^{t \cos \Big( \frac{\pi}{\alpha} \Big)} \Big( 2^{1-\vartheta} (\frac{\pi}{\alpha})^\vartheta \delta^\vartheta + \delta^\vartheta \Big)= \frac2\alpha e^{t \cos \Big( \frac{\pi}{\alpha} \Big)} \Big( 2^{1-\vartheta} (\frac{\pi}{\alpha})^\vartheta  + 1 \Big)\delta^\vartheta.\label{eq:EstimG'}
		\end{align}
	}
	Where the penultimate inequality follows from the fact that \( \frac{\pi}{2} \leq \frac{\pi}{\alpha} \leq \pi \), so that \(1 - e^{\delta \cos \Big( \frac{\pi}{\alpha} \Big)} \leq 1 - e^{-\delta},\)
	and \(
	\sin \Big[ \frac{\delta}{2} \sin \Big( \frac{\pi}{\alpha} \Big) \Big] \leq \Big( \frac{\delta}{2} \sin \Big( \frac{\pi}{\alpha} \Big) \Big)^{\vartheta} \leq \Big( \frac{\delta}{2} \Big( \frac{\pi}{\alpha} \Big) \Big)^{\vartheta}
	\) owing to Lemma \ref{lm:expansion} (2).
	The final inequality follows from Lemma \ref{lm:expansion} (1).
	Moreover, for the term \(F'_\alpha(t)-F'_\alpha(t+\delta) := \int_0^{+\infty} e^{-tu}(1-e^{-\delta u} ) uH_{\alpha}(u)du\), we may write
	\vspace{-0.2cm}
	\begin{equation}\label{eq:EstimF'}
		| F'_\alpha(t)-F'_\alpha(t+\delta) | \leq \int_0^{+\infty} e^{-tu}(1-e^{-\delta u} )^{\vartheta} u |H_{\alpha}(u)|du \leq  \int_0^{+\infty} e^{-tu} \delta^{\vartheta} u^{1+\vartheta} | H_{\alpha}(u)|du.
		\vspace{-0.2cm}
     \end{equation}
	Finally, combining the estimates \eqref{eq:EstimG'} and \eqref{eq:EstimF'}, we obtain:
	\vspace{-0.2cm}
		\begin{equation}\label{eq:estim_hLambda}| f_{\alpha ,1}(t+\delta)-f_{\alpha,1}(t)  | \leq h_{\alpha,1}(t) \delta^{\vartheta} \; \text{where}\; h_{\alpha,1}(t):= \frac2\alpha e^{t \cos \big( \frac{\pi}{\alpha} \big)} \big( 2^{1-\vartheta} (\frac{\pi}{\alpha})^\vartheta  + 1 \big) + \int_0^{+\infty} e^{-tu} u^{1+\vartheta} | H_{\alpha}(u)|du
			\vspace{-0.2cm}
		\end{equation}
	It remains to verify that \( h_{\alpha,1} \in {\cal L}^1(\R_+) \cap {\cal L}^\infty(\R_+) \) and \( h_{\alpha,1}*\varsigma^2 \in {\cal L}^\infty(\R_+).\) Since $\alpha \in (1,2)$, we have
	\(\cos\!\Big(\frac{\pi}{\alpha}\Big) < 0\).
	Hence the first term in $h$ decays exponentially as $t \to \infty$, and therefore belongs to
	\({\cal L}^1(\mathbb{R}_+) \cap {\cal L}^\infty(\mathbb{R}_+).\)
	For the integral term, we derive form~\eqref{eq:Halpha} that: \(|H_\alpha(u)|\stackrel{0}{\sim} \frac{|\sin(\pi \alpha)|}{\pi} u^{\alpha-1}
	\; \mbox{ and }\; |H_\alpha(u)|\stackrel{+\infty}{\sim} \frac{|\sin(\pi \alpha)|}{\pi}u^{-(\alpha+1)}.\)
	Consequently
	\vspace{-0.2cm}
	\begin{equation}\label{eq:Int_Hu}
		u^{\vartheta} | H_\alpha(u)|\stackrel{0}{\sim} \frac{|\sin(\pi \alpha)|}{\pi} u^{\alpha-1+\vartheta}
		\quad \mbox{ and }\quad u^{\vartheta} |H_\alpha(u)|\stackrel{+\infty}{\sim} \frac{|\sin(\pi \alpha)|}{\pi} u^{-(1 +\alpha-\vartheta)},
		\vspace{-0.2cm}
	\end{equation}
	which implies that \(\int_{(0, +\infty)}\!\!\!  u^{\vartheta} |H_{\alpha}(u)|\,du<+\infty\) and \(\int_{(0, +\infty)}\!\!\!  u^{1+\vartheta} |H_{\alpha}(u)|\,du<+\infty\) if and only if \( \vartheta < \alpha \wedge (\alpha-1).\) Consequently, the second term in $h_{\alpha,1}$ also belongs to
	\({\cal L}^1(\mathbb{R}_+) \cap {\cal L}^\infty(\mathbb{R}_+).\)
		Now, we show that $h_{\alpha,1}*\varsigma^2 \in {\cal L}^\infty(\mathbb{R}_+)$.
		We write, for $t\geq\epsilon>0$,
		\vspace{-0.2cm}
		\begin{align*}
		(h_{\alpha,1}*\varsigma^2)(t)
		=
		\int_0^t h_{\alpha,1}(t-s)\varsigma^2(s)\,ds
		&=\int_0^{\epsilon} h_{\alpha,1}(t-s)\varsigma^2(s)\,ds + \int_{\epsilon}^t h_{\alpha,1}(t-s)\varsigma^2(s)\,ds\\
		&\leq \|h_{\alpha,1}\|_{{\cal L}^\infty(\mathbb{R}_+)}\int_0^\epsilon \varsigma^2(s)\,ds +  \|\varsigma^2\|_{{\cal L}^\infty([\epsilon,\infty))}\int_0^{t-\epsilon} h_{\alpha,1}(s)\,ds <\infty
		\end{align*}
		owing to~\cite[Lemma 3.9]{EGnabeyeu2025}, the local asymptotic~\eqref{eq:varsigma} for \(\alpha\in(1,2)\) and the fact that \(h_{\alpha,1}\in {\cal L}^1(\mathbb{R}_+) \cap {\cal L}^\infty(\mathbb{R}_+)\). This yields \(\sup_{t\ge0}(h_{\alpha,1}*\varsigma^2)(t)<\infty\).
	

	\medskip
	\noindent{\sc Step~3.} In this step, we provide refined ${\cal L}^p(\R_+)$-$\vartheta_p$ global H\"older continuity estimates for every \(p\geq1\), improving the ranges of admissible H\"older exponents stated in the remark following Proposition~\ref{prop:main}.
	
	\smallskip
	\noindent
	1. Owing to Fubini-Tonelli's theorem to interwind the order of integration, we have sucessively:
	\begin{align*}
		&\int_0^{+\infty} \big|F'_{\alpha }(t+\delta)-F'_{\alpha}(t) \big| dt  \le \int_{(0, +\infty)}\!\!\! u^{1+\vartheta}|H_{\alpha}(u)|\int_0^{+\infty} \!\!e^{-tu} dt\,du\, \delta^{\vartheta} =  \Big[ \int_{(0, +\infty)}\!\!\! u^\vartheta |H_{\alpha}(u)|\,du\,\Big]\delta^{\vartheta}\\
	&\int_0^{+\infty} \big|G'_{\alpha}(t)-G'_{\alpha }(t+\delta) \big| dt \le \frac2\alpha  \Big( 2^{1-\vartheta} (\frac{\pi}{\alpha})^\vartheta  + 1 \Big)\delta^\vartheta   \int_0^{+\infty} \!\!e^{t \cos \Big( \frac{\pi}{\alpha} \Big)} dt\,=  \Big[\frac{-2}{\alpha \cos \Big( \frac{\pi}{\alpha} \Big)}  \Big( 2^{1-\vartheta} (\frac{\pi}{\alpha})^\vartheta  + 1 \Big) \Big]\delta^{\vartheta}. 
	\end{align*}
	It follows that, \(\int_0^{+\infty} \big|f_{\alpha ,1}(t+\delta)-f_{\alpha,1}(t) \big| dt  \le\Big[ \int_{\R_+}\!\!\! u^\vartheta |H_{\alpha}(u)|\,du\, + \frac{2}{\alpha}  \Big( 2^{1-\vartheta} (\frac{\pi}{\alpha})^\vartheta  + 1 \Big)\Big]\delta^{\vartheta}:=C\delta^{\vartheta}\) since \(\int_{(0, +\infty)}\!\!\!\;  u^{\vartheta} |H_{\alpha}(u)|\,du<+\infty\) owing to~\eqref{eq:Int_Hu} and provided that \(\vartheta\in(0,\alpha\wedge1)=(0,1)\).
	
	\medskip
	\noindent 2. Secondly, as: \(\left(
	f_{\alpha ,1}(t+\delta)-f_{\alpha,1}(t)\right)^2 \leq 2\left( \left(F'_\alpha(t)-F'_\alpha(t+\delta)\right) \right)^2 + 2 \left( \left(G'_\alpha(t)-G'_\alpha(t+\delta)\right) \right)^2 \)
	with:
	{\small
		\begin{align*}
			\int_0^{+\infty} \,&\big(F'_{\alpha }(t+\delta)-F'_{\alpha}(t) \big)^2 dt  \le \int_0^{+\infty}  \int_0^{+\infty} e^{-tu} \delta^{\vartheta} u^{1+\vartheta} |H_{\alpha}(u)|du \int_0^{+\infty} e^{-tv} \delta^{\vartheta} v^{1+\vartheta} |H_{\alpha}(v)|dv \\
			&\leq \int_{(0, +\infty)^2}\!\!\! (uv)^{1+\vartheta}|H_{\alpha}(u)H_{\alpha}(v)|\int_0^{+\infty} \!\!e^{-t(u+v)} dt\,du\,dv\, \delta^{2\vartheta} =   \int_{(0, +\infty)^2}\!\!\! \frac{(uv)^{1+\vartheta}}{u+v}|H_{\alpha}(u)H_{\alpha}(v )|\,du\,dv\, \delta^{2\vartheta} \\
			& \le  \tfrac 12 \int_{(0, +\infty)^2}\!\!\!  (uv)^{\frac 12 +\vartheta} |H_{\alpha}(u)H_{\alpha}(v )|\,du\,dv\, \delta^{2\vartheta} = \tfrac 12 \Big[\int_{(0, +\infty)}\!\!\!  u^{\frac 12 +\vartheta} |H_{\alpha}(u)|\,du\Big]^2\, \delta^{2\vartheta}.
		\end{align*}
	}
	\noindent
	where we used  Fubini-Tonelli's theorem in the first   line to interwind the order of integration and the elementary inequality $\sqrt{uv} \le \frac 12 (u+v)$ when $u,\, v\ge 0$ in the penultimate line. Furthermore,
	{\small
		\[
		\int_0^{+\infty} \big(G'_{\alpha }(t+\delta)-G'_{\alpha}(t) \big)^2 dt \le \frac{4}{\alpha^2}  \left( \left( 2^{1-\vartheta} (\frac{\pi}{\alpha})^\vartheta  + 1 \right)\delta^\vartheta \right)^2  \int_0^{+\infty} \!\!e^{2t \cos \left( \frac{\pi}{\alpha} \right)} dt\,=  \left[\frac{-2}{\alpha^2 \cos \left( \frac{\pi}{\alpha} \right)}  \left( 2^{1-\vartheta} (\frac{\pi}{\alpha})^\vartheta  + 1 \right)^2 \right]\delta^{2\vartheta}. 
		\]
	}
	It follows that , \(\int_0^{+\infty} \big(f_{\alpha ,1}(t+\delta)-f_{\alpha,1}(t) \big)^2 dt  \le \left[ \Big(\int_{(0, +\infty)}\!\!\!  u^{\frac 12 +\vartheta} |H_{\alpha}(u)|\,du\Big)^2 +\frac{4}{\alpha^2} \left( 2^{1-\vartheta} (\frac{\pi}{\alpha})^\vartheta  + 1 \right)^2\right]\, \delta^{2\vartheta}.\)
	
	\noindent Now, we derive, still form~\eqref{eq:Halpha} that:
	$
	 | H_\alpha(u)|\stackrel{0}{\sim} \frac{|\sin(\pi \alpha)|}{\pi} u^{\alpha-1}
	\; \mbox{ and }\; |H_\alpha(u)|\stackrel{+\infty}{\sim} \frac{|\sin(\pi \alpha)|}{\pi}u^{-(\alpha+1)}.
	$
	Consequently
	$$
	u^{\frac 12+\vartheta} |H_\alpha(u)|\stackrel{0}{\sim} \frac{|\sin(\pi \alpha)|}{\pi} u^{\alpha-\frac 12+\vartheta}
	\quad \mbox{ and }\quad u^{\frac 12+\vartheta} |H_\alpha(u)|\stackrel{+\infty}{\sim} \frac{|\sin(\pi \alpha)|}{\pi} u^{-(-\frac 12 +\alpha-\vartheta)},
	$$
	which implies that \(\int_{\R_+}\!\!\!  u^{\frac 12 +\vartheta} |H_{\alpha}(u)|\,du<+\infty\, \mbox{iff}\, \vartheta < (\alpha-\tfrac 12) \wedge 1. 
	\).
	
	\smallskip
	\noindent	3. {\em General case \(p\geq1\):}  Global H\"older estimate. From the estimates~\eqref{eq:estim_hLambda}, for every $0<\vartheta<1$,
	\[
	\|f_{\alpha,1}(\cdot+\delta)-f_{\alpha,1}\|_{{\cal L}^p(\R_+)}
	\le \|h_{\alpha,1}\|_{{\cal L}^p(\R_+)}\delta^\vartheta,\; 
	\]
	It remains to show that $h_{\alpha,1}\in {\cal L}^p(\R_+)$. The first term in the definition of $h_{\alpha,1}$ clearly belongs to ${\cal L}^p(\R_+)$. Moreover, Minkowski's integral inequality gives
		\begin{equation*}
			\begin{aligned}
				\left\|\int_0^\infty e^{-\cdot u}u^{1+\vartheta}|H_\alpha(u)|\,du\right\|_{{\cal L}^p(\R_+)}
				&\le \int_0^\infty u^{1+\vartheta}|H_\alpha(u)|
				\|e^{-u\cdot}\|_{{\cal L}^p(\R_+)}\,du= p^{-\frac1p}\int_0^\infty u^{1+\vartheta-\frac1p}|H_\alpha(u)|\,du ,
			\end{aligned}
		\end{equation*}
	Now, we derive, still form~\eqref{eq:Halpha}, the asymptotic estimates (see also~\eqref{eq:Int_Hu})
	$$
	u^{1-\frac1p+\vartheta} |H_\alpha(u)|\stackrel{0}{\sim} \frac{|\sin(\pi \alpha)|}{\pi} u^{\alpha-\frac 12+\vartheta}
	\quad \mbox{ and }\quad u^{\frac 12+\vartheta} |H_\alpha(u)|\stackrel{+\infty}{\sim} \frac{|\sin(\pi \alpha)|}{\pi} u^{-(\frac1p +\alpha-\vartheta)},
	$$ 
	Consequently, the last integral is finite whenever $\vartheta<\alpha-1+\frac1p$. Therefore, for every $0<\vartheta<\alpha-1+\frac1p$, there exists a constant $C_{p,\vartheta}$ such that
	$\|f_{\alpha,1}(\cdot+\delta)-f_{\alpha,1}\|_{{\cal L}^p(\R_+)}
	\le C_{p,\vartheta} \delta^\vartheta$ for every $\delta>0$.
	One concludes the general case when \( \lambda > 0 \) by scaling.\hfill $\Box$
	\begin{Lemma}\label{lem:bound}
		Let \(\alpha \in (1, \frac32)\). For every \(k \geq 1\),
		\begin{enumerate}
			\item \(\forall l \geq 1 \quad \forall a \geq 1, \quad B(\alpha\ell, \alpha(k-\ell+a)) \geq \frac{1}{(\alpha(k+a) - 1) 2^{\alpha k+2(a-1)}} 
			\geq \frac{1}{\alpha(k+a) 2^{\alpha k+2(a-1)}}.\)
			\item \(
			(a * b)_k \leq \frac{2^{\alpha k}}{\Gamma(\alpha(k+1))} \left( 1 + (k + 1)(1 + \log k) \right) .
			\)
			\item  \((b^{*2})_k \leq \frac{(\alpha(k+2) - 1)(k + 1) 2^{\alpha k+2}}{\Gamma(\alpha(k+2))} .\)
		\end{enumerate}
	\end{Lemma}		
	\noindent {\bf Proof.} 1.\(\forall l \geq 1 \quad \forall a \geq 1,\) we have:
	{\small
		\begin{align*}
			B(\alpha\ell, \alpha(k-\ell+a)) &= \int_0^1 u^{\alpha\ell - 1} \left( 1 - u \right)^{\alpha(k-\ell+a) - 1} \, du \geq \int_0^{\frac{1}{2}} u^{\alpha(k+a) - 2} \, du + \int_{\frac{1}{2}}^1 \left( 1 - u \right)^{\alpha(k+a) - 2} \, du \\
			&= 2 \int_0^{\frac{1}{2}} u^{\alpha(k+a) - 2} \, du \geq \frac{1}{(\alpha(k+a) - 1) 2^{\alpha(k+a)-2}} \geq \frac{1}{(\alpha(k+a) - 1) 2^{\alpha k+2(a-1)}}.
		\end{align*}
	}
	Where the last inequality comes from the fact that \(\alpha < 2 \).
	
	2. Using the identity :
	\(
	\forall a, b > 0 \quad \Gamma(a+1) = a \Gamma(a), \quad B(a, b) := \frac{\Gamma(a)\Gamma(b)}{\Gamma(a + b)}
	\), we have for every \(k \geq 1\)
	\vspace{-0.3cm}
	{\small
		\begin{align*}
			(a * b)_k &= \sum_{\ell=0}^{k} \frac{1}{\Gamma(\alpha\ell + 1) \Gamma(\alpha(k - \ell + 1))} = \frac{1}{\Gamma(1)\Gamma(\alpha(k+1))}+\sum_{\ell=1}^{k} \frac{1}{\alpha \ell\Gamma(\alpha \ell) \Gamma(\alpha(k - \ell + 1))} \\
			&=\frac{1}{\Gamma(\alpha(k+1))}\left[1+\frac1\alpha\sum_{\ell=1}^{k} \frac{1}{ \ell}\frac{1}{B(\alpha\ell, \alpha(k-\ell+1))} \right] \leq \frac{2^{\alpha k}}{\Gamma(\alpha(k+1))} \left( 1 + (k + 1)(1 + \log k) \right) .
		\end{align*}
	}
	where the last inequality comes from Lemma \ref{lem:bound}(1) for \(a=1\) and the fact that \(\frac{1}{2^{\alpha }} \leq 1\). 
	
	3. Likewise,
	\vspace{-0.3cm}
	{\small
		\begin{align*}
			(b^{*2})_k &= \sum_{\ell=0}^{k} \frac{1}{\Gamma(\alpha (\ell + 1)) \Gamma(\alpha(k - \ell + 1))} =\frac{1}{\Gamma(\alpha(k+2))}\sum_{\ell=0}^{k} \frac{1}{B(\alpha(\ell + 1), \alpha(k-\ell+1))}  \\
			&\leq \frac{(\alpha(k+2) - 1)(k + 1)}{\Gamma(\alpha(k+2))} 2^{\alpha k+2}.
		\end{align*}
	}
	Still owing to Lemma \ref{lem:bound} (1), now for \(a=2\). \hfill $\Box$
	
	\bigskip
	\noindent {\bf Proof of Proposition \ref{prop:alphaFractKernel1}.}
	\medskip
	\noindent  {\sc Step~1.} (1) comes from equation~\eqref{eq:varsigma} and ~\cite[Lemma 3.9]{EGnabeyeu2025}.
	
	\noindent  {\sc Step~2.} To establish statement (2), following the approach in \cite{Pages2024}, it is useful (though not strictly necessary) to transition to Laplace transforms. For simplicity, and as indicated in remark ~\eqref{eq:varsigma_reduit}, we assume \( c = \lambda = 1 \) and proceed by rewriting the series expansions in~\eqref{eq:e_alpha}. We define \( R_{\alpha} := R_{\alpha, 1} \) and \( f_{\alpha} := f_{\alpha, 1} \), as follows: 
	\vspace{-0.2cm}
	\begin{equation*}R_{\alpha}(t) = \sum_{k \geq 0} (-1)^k a_k t^{\alpha k},  \; f_{\alpha}(t) = t^{\alpha - 1} \sum_{k \geq 0} (-1)^k b_k t^{\alpha k} \quad \text{with} \quad a_k = \frac{1}{\Gamma(\alpha k + 1)}, \; b_k = \frac{1}{\Gamma(\alpha(k + 1))}, \; k \geq 0.
	\vspace{-0.2cm}
	\end{equation*}
	Now, using the Cauchy product of two series
	and the fact that \( L_{u^{\gamma}}(t) = t^{-(\gamma+1)} \Gamma(\gamma + 1) \), we obtain the following Laplace transforms: \(	L_{R_{\alpha} f_{\alpha}}(t) = t^{-\alpha} \sum_{k \geq 0} (-1)^k (a * b)_k t^{-\alpha k} \Gamma(\alpha(k + 1))\)  and\\ \(L_{f_{\alpha}^2}(t) = t^{-2\alpha + 1} \sum_{k \geq 0} (-1)^k (b^{*2})_k t^{-\alpha k} \Gamma(\alpha(k + 2) - 1)\),
	where for two sequences of real numbers \( (u_k)_{k \geq 0} \) and \( (v_k)_{k \geq 0} \), the Cauchy product is defined as \( (u * v)_k = \sum_{\ell = 0}^k u_\ell v_{k - \ell} \). We define the sequences
	\vspace{-0.1cm}
	\begin{equation*}
	\widetilde{b}_k = (b^{*2})_k \Gamma(\alpha(k + 2) - 1) \quad \text{and} \quad \widetilde{c}_k = c_k \Gamma(\alpha(k - 1) + 2), \; k \geq 0.
	\vspace{-0.1cm}
	\end{equation*}
	Assuming that \( \varsigma_{\alpha}^2(t) \) (for \( c = \lambda = 1 \)) takes the expected form~\eqref{eq:varsigma_reduit}, we have:
	\vspace{-0.1cm}
	\begin{equation*}
	L_{\varsigma_{\alpha}^2}(t) = 2 \sum_{k \geq 0} (-1)^k c_k t^{-(\alpha(k - 1) + 2)} \Gamma(\alpha(k - 1) + 2) = 2 t^{\alpha - 2} \sum_{k \geq 0} (-1)^k \widetilde{c}_k t^{-\alpha k}.
	\vspace{-0.2cm}
	\end{equation*}
	Thus, by equating the coefficients from both sides of equation~\eqref{eq:Laplacesigma}, we obtain the condition:
	\centerline{$
		\forall k \geq 0, \quad (\widetilde{b} * \widetilde{c})_k = (a * b)_k \Gamma(\alpha(k + 1)).
		$}
	Simple computations yield \(	c_0 = \frac{\Gamma(\alpha)^2}{\Gamma(2\alpha - 1) \Gamma(2 - \alpha)},\)
	and for every \( k \geq 1 \),
	\vspace{-0.3cm}
	{\small
		\begin{equation}\label{eq:c0}
			c_k = \frac{\Gamma(\alpha)^2}{\Gamma(\alpha(k - 1) + 2) \Gamma(2\alpha - 1)} \left[ \Gamma(\alpha(k + 1)) (a * b)_k - \sum_{\ell = 1}^k \Gamma(\alpha(\ell + 2) - 1) \Gamma(\alpha(k - \ell - 1) + 2) (b^{*2})_\ell c_{k - \ell} \right].
		\end{equation}
	}
	\normalsize
	Using standard identities such as \( \Gamma(a) \Gamma(b) = \Gamma(a + b) B(a, b) \) for \( a, b > 0 \), where \( B(a, b) = \int_0^1 u^{a - 1}(1 - u)^{b - 1} \, du \), and \( \Gamma(a + 1) = a \Gamma(a) \), we arrive at the formulation of the \( c_k \)'s provided in the proposition, which is more suitable for numerical computations.
	
	\medskip
	\noindent  {\sc Step~3.} Using standard methods, as in \cite{CalGraPag21} 
	or Appendix~A of \cite{Pages2024} (in the case $\alpha \in (\frac12, 1)$), we show that the radius of convergence \( \rho_\alpha \) of the power series defined by the coefficients \( c_k \) is infinite. 
	Firstly, let us prove by induction that there exists \(A > 2^{\alpha+2}\) and \(K > 1\) such that,
	\begin{equation}\label{eq:bound_onC}
		\forall k \geq 0 \quad |c_k| \leq \frac{K A^k}{ \Gamma(\alpha (k - 1) + 2 )}.
	\end{equation}
	By the triangle inequality, we get the bound :
	\vspace{-0.2cm}
	{\small
		\begin{equation}
			\left| c_k \right| \leq \frac{\Gamma(\alpha)^2 \Gamma(\alpha(k + 1))}{\Gamma(\alpha (k - 1) + 2 ) \Gamma(2\alpha - 1)} \left[ (a * b)_k + \alpha(k + 1) \sum_{\ell = 1}^k B\left( \alpha(\ell + 2) - 1, \alpha(k - \ell - 1) + 2 \right) (b^{*2})_\ell |c_{k-\ell}| \right].  \label{eq:second_inequality}
			\vspace{-0.2cm}
		\end{equation}
	}
	{\em Initialisation:} 
	For \( k = 0 \), \( c_0 = \frac{\Gamma(\alpha) ^2}{\Gamma(2-\alpha)\Gamma(2\alpha-1)}\leq\frac{K}{\Gamma(2 - \alpha)}  \) since \(K > 1\) and by log-convexity \( \frac{\Gamma(\alpha)^2}{\Gamma(2\alpha - 1)} < 1 \). 
	{\em Heredity:}
	Now let \( k \geq 1 \) and assume that \(c_\ell\) satisfies the inequality~\eqref{eq:bound_onC} for every \(\ell = 0, \dots, k-1\). Then, for every \(\ell = 1, \dots, k\),
	\vspace{-0.2cm}
	\begin{align*}
		& \,B\left( \alpha(\ell + 2) - 1, \alpha(k - \ell - 1) + 2 \right) (b^{*2})_\ell |c_{k-\ell}| \leq \frac{\Gamma(\alpha (\ell +2) -1 )\Gamma(\alpha (k-\ell - 1) + 2 )}{\Gamma(\alpha (k+ 1) + 1)\;\Gamma(\alpha (k-\ell - 1) + 2 )} \times K A^{k-\ell}(b^{*2})_\ell \\
		&\leq \frac{ K A^{k-\ell}\Gamma(\alpha (\ell +2) -1 )}{\Gamma(\alpha (k+ 1) + 1)}\frac{(\alpha(l+2) - 1)(l + 1) 2^{\alpha l+2}}{\Gamma(\alpha(l+2))} \leq \frac{ K A^{k-\ell}}{\Gamma(\alpha (k+ 1) + 1)}\frac{(\alpha(l+2) - 1)(l + 1) 2^{\alpha l+2}}{(\alpha (\ell +2) -1 )}  \\
		&\leq K \frac{ (l + 1) 2^{\alpha l+2} A^{k-\ell}}{\alpha (k+ 1)  \Gamma(\alpha (k+ 1) )}. \quad \text{Inserting this bound into the inequality~\eqref{eq:second_inequality} for \(c_k\) gives:}
	\end{align*}
		\vspace{-0.2cm}
		\begin{equation}
	\left| c_k \right| \leq \frac{\Gamma(\alpha)^2}{\Gamma(\alpha (k - 1) + 2 ) \Gamma(2\alpha - 1)} \left[ \Gamma(\alpha(k + 1))(a * b)_k + K\,A^{k} \frac{1}{\Gamma(\alpha(k + 1))} \sum_{\ell = 1}^k (\ell + 1) \rho^\ell \right].
		\vspace{-0.2cm}
	\end{equation}
	\noindent
	where we set \(
	\rho = \rho(A) := \frac{2^{\alpha+2}}{A}\).  Next, dividing the above inequality by $ K A^{k}$  and using the upper bound for $ (a * b)_k $ from Lemma \ref{lem:bound}(2):
	\vspace{-0.2cm}
	\begin{equation}
	\frac{\left| c_k \right|}{ K A^{k}}  \leq \frac{ 1}{\Gamma(\alpha (k - 1) + 2 ) } \frac{\Gamma(\alpha)^2}{\Gamma(2\alpha - 1)} \left[ \frac{\rho^k}{K}\left( 1 + (k + 1)( 1 + \log k) \right) + \frac{1}{(1-\rho)^2} \right].
	\vspace{-0.2cm}
	\end{equation}
	Owing to the elementary inequality: \(\forall \, \rho \in (0, 1), \quad \sum_{l \geq 1} l \rho^{l-1} \leq \frac{1}{(1 - \rho)^2}\).
	Let \( \epsilon > 0 \) and let \( A = A_{\epsilon} \) be large enough so that
	\(\sup_{k \geq 1} \left( \rho^k + \rho^k(k + 1)( 1 + \log k) \right) < \epsilon
	\quad \text{and} \quad \frac{1}{(1 - \rho)^2} < 1 + \epsilon.\)
	Due to the log-convexity of the Gamma function, \(
	\log \Gamma(\alpha) \le \tfrac{1}{2} \log \Gamma(2\alpha - 1) + \log \Gamma(1) = \tfrac{1}{2} \log \Gamma(2\alpha - 1)
	\), so that \( \frac{\Gamma(\alpha)^2}{\Gamma(2\alpha - 1)} < 1 \). Thus, it is possible to choose \( \epsilon \) small enough and \( K \) large enough such that:
	\vspace{-0.2cm}
	\begin{equation}	\frac{\Gamma(\alpha)^2}{\Gamma(2\alpha - 1)} \left[ \frac{\rho^k}{K} \left( 1 + (k + 1)(1 + \log k) \right) + \frac{1}{(1-\rho)^2} \right] \leq \frac{\Gamma(\alpha)^2}{\Gamma(2\alpha - 1)} \left( \frac{\epsilon}{K} + 1 + \epsilon \right) < 1.
	\vspace{-0.2cm}
	\end{equation}
	Consequently, \(|c_k| \leq \frac{K A^k}{ \Gamma(\alpha (k - 1) + 2 )}.\)
	And thus the Cauchy-Hadamard's formula for the radius of convergence together with Stirling's formula gives:
	\vspace{-0.2cm}
	\begin{equation}
	\limsup_{k \to \infty} |c_k|^{\frac{1}{k}} \leq \limsup_{k \to \infty} \left( \frac{K\,A^k}{ \Gamma(\alpha k + 2 - \alpha)} \right)^{1/k} \sim \lim_{k \to \infty}  A \frac{K^\frac1k}{ e^{-\alpha}(\alpha (k-1) + 2)^\alpha} = 0.
	\vspace{-0.2cm}
	\end{equation}
	This completes the proof.	 \hfill $\square$
	
		\bigskip
	\noindent {\bf Proof of Proposition \ref{prop:main2}.}
	$(a)$ 	We consider the function $R_{\alpha, \rho, \lambda}(t) = 1 + \sum_{k \geq 1} (-1)^k \frac{\lambda^k }{\Gamma(k\alpha)}I_k(t) $ where \(I_k(t) = \int_0^t e^{-\rho s} s^{k\alpha -1} \,ds\). Given that for all \( k \geq 1\), the function \( s \mapsto e^{-\rho s} s^{k\alpha -1} \) is continuous on \((0,+\infty)\), 
	\(I_k\) is differentiable on \((0,+\infty)\).
	Moreover, the series of derivatives \(\sum_{k \geq 1} (-1)^k \frac{\lambda^k}{\Gamma(k\alpha)} e^{-\rho t} t^{k\alpha -1} \)
	converges absolutely and locally uniformly on \((0,+\infty)\). Therefore, by the theorem on term-by-term differentiation of uniformly convergent series,	
	\( R_{\alpha, \rho, \lambda}(t) \) is differentiable for \( t > 0 \), with its derivative given by \(R'_{\alpha, \rho, \lambda}(t) = \sum_{k \geq 1} (-1)^k \frac{\lambda^k }{\Gamma(k\alpha)} e^{-\rho t} t^{k\alpha -1} =: f_{\alpha, \rho, \lambda}(t)\).
	One could argue similarly to show that \( R_{\alpha, \rho, \lambda} \) is infinitely differentiable, i.e., \( \mathcal{C}^\infty \) on \( (0, +\infty) \). Alternatively, observe that for all \( t > 0 \), we have \(f_{\alpha, \rho, \lambda}(t) = e^{-\rho t} f_{\alpha, \lambda}(t),\)
	which is \( \mathcal{C}^\infty \) as the product of such functions, by virtue of Proposition~\ref{prop:main}~$(a)$.
	
	\smallskip
	\noindent $(b)$ The representation of $f_{\alpha,\rho, \lambda}$ follows by definition and from the claim $(b)$ of Proposition \ref{prop:main2}.
	
	\smallskip
	\noindent $(c)$  Let us prove the $L^{p}$-integrability of $f_{\alpha,\rho,  \lambda}$. Once noted that $f_{\alpha,\rho,  \lambda} = e^{-\rho t}f_{\alpha,\lambda}$ so that $\int_0^{+\infty} f_{\alpha,\rho, \lambda}^{p}(t)dt= \int_0^{+\infty} e^{-p\rho t} f_{\alpha,\lambda}^{p}(t)dt \leq   \int_0^{+\infty} f_{\alpha,\lambda}^{p}(t)dt,$
	it is clear that it is enough to have that $ f_{\alpha,\lambda}$ is ${\cal L}^{p}$-integrable.
	It follows from Proposition \ref{prop:main}  that $ f_{\alpha,\rho, \lambda}$ is ${\cal L}^{p}$-integrable $\forall p \in (0,+\infty]$.
	
		\medskip
	\noindent 
	As for the $\vartheta$-H\"older continuity of $f_{\alpha,\rho,\lambda}$, let $\delta >0$ and \(\vartheta\in(0,1]\). One has
	\begin{equation}\label{eq:Boundf_Gamma}
	\left| f_{\alpha,\rho,\lambda}(t+\delta) - f_{\alpha,\rho,\lambda}(t) \right|
	\leq e^{-\rho(t+\delta)}\left| f_{\alpha, \lambda}(t+\delta) - f_{\alpha, \lambda}(t) \right| + e^{-\rho t} \left|f_{\alpha, \lambda}(t) \right| (1-e^{-\rho \delta}) .
    \end{equation}
	Then, 
	 using Lemma~\ref{lm:expansion}, one may deduce
	\[| f_{\alpha ,\rho,\lambda}(t+\delta)-f_{\alpha,\rho,\lambda}(t)  | \leq h_{\alpha,\rho,\lambda}(t) \delta^{\vartheta} \quad \text{with}\quad h_{\alpha,\rho,\lambda}:= e^{-\rho\delta} h_{\alpha,\lambda} + \left|f_{\alpha, \lambda} \right|
	\]
	Consequently, owing to Proposition~\ref{prop:main}, \(h_{\alpha,\rho,\lambda} \in {\cal L}^1(\R_+) \cap {\cal L}^\infty(\R_+)\).
	Arguing as in the proof of Proposition~\ref{prop:main},
	and using~\cite[Lemma 3.9]{EGnabeyeu2025}, the local asymptotic~\eqref{eq:varsigma2} for \(\alpha\in(1,2)\), and the fact that \(h_{\alpha,\rho,\lambda}\in {\cal L}^1(\mathbb{R}_+) \cap {\cal L}^\infty(\mathbb{R}_+)\), we obtain \(\sup_{t\ge0}(h_{\alpha,\rho,\lambda}*\varsigma^2_{\alpha,\rho,\lambda, c})(t)<\infty\).
 This completes the proof.	 \hfill $\square$
	
		\bigskip
	\noindent {\bf Proof of Proposition \ref{prop:alphaFractKernel1_}.}
	We have \( \lim_{t\to 0^+}g_{\alpha,\rho,\lambda} = +\infty \) and \(\lim_{t\to +\infty}g_{\alpha,\rho,\lambda} >0\).
	Hence, there exists \(t_0, t_1 >0\) such that \(g_{\alpha,\rho,\lambda} \geq 0\) at least on the small intervals \((0,t_0) \cup (t_1, +\infty)\) with $t_0 = \inf\{t \,:\, g_{\alpha,\rho,\lambda}(t) < 0\}$ and $t_1 = \sup\{t \,:\, g_{\alpha,\rho,\lambda}(t) < 0\}$. 
	By continuity of $g_{\alpha,\rho,\lambda}$ it is clear that $g_{\alpha,\rho,\lambda}(t_0)=g_{\alpha,\rho,\lambda}(t_1) = 0$ and $g_{\alpha,\rho,\lambda} \ge 0$ on $[0, t_0] \cup [t_1, +\infty)$.
	While numerical computations suggest that \(g_{\alpha,\rho,\lambda}\) is positive on \(\R_+\) (i.e. \(t_0=t_1=\infty\)), establishing this positivity analytically turns out to be quite challenging.
	We shall, however, establish that if \(T^{\alpha,\lambda,\rho}\) is the first zero of the resolvent \(R_{\alpha,\rho,\lambda}\) (see \cite[Proposition 3.13.]{Gorenflo2020MittagLeffler} for all zeros of the functions $E_\alpha$), then, since \(R_{\alpha,\rho,\lambda}^2\) decreases strictly on \((0, T^{\alpha,\lambda,\rho})\), the function \(g_{\alpha,\rho,\lambda}\) remains non-negative over that interval. \\
	
	Let's assume that \(t_0 \in (0, T^{\alpha,\lambda,\rho})\) and thus $g_{\alpha,\rho,\lambda} \le 0$ on a small interval $[t_0, t_0 + \eta] \subset(0, T^{\alpha,\lambda,\rho})$ for some $\eta > 0$. 
	Then, for every $t \in (t_0, t_0 + \eta]$, there exists \(\tau >0\) such that \(t=t_0 + \tau\).
	Let \(\delta \in (0,\frac\tau2)\), and set \(c:=-\max_{s \in [t_0+\delta,t_0+\tau]}g_{\alpha,\rho,\lambda}(s)\).
	By continuity $c > 0$ and 
	$g_{\alpha,\rho,\lambda}(s) \le -c$ for all $s \in [t_0+\delta,t_0+\tau]$. 
	For simplification, we set \(f_{\alpha,\rho,\lambda}\equiv f_{\alpha}\) and \(R_{\alpha,\rho,\lambda}\equiv R_{\alpha}, g_{\alpha,\rho,\lambda}\equiv g_{\alpha}\). Then, we have:
	{\small
		\begin{align*}
			&\,(f_{\alpha}^2* g_{\alpha})(t_0+\tau)-(f_{\alpha}^2* g_{\alpha})(t_0)= \int_{0}^{t_0} \underbrace{(f_{\alpha}^2(t_0+\tau-s)-f_{\alpha}^2(t_0-s))}_{\geq \approx 0}\, \underbrace{g_{\alpha}(s)}_{\geq0}\,ds  + \int_{t_0}^{t_0+\delta} f_{\alpha}^2(t_0+\tau-s)\, \underbrace{g_{\alpha}(s)}_{\leq0}\,ds\\
			&\qquad\hspace{4cm} + \int_{t_0+\delta}^{t_0+\tau} f_{\alpha}^2(t_0+\tau-s)\, \underbrace{g_{\alpha}(s)}_{\leq0}\,ds\leq I_1 - I_2 -c \left(\int_{0}^{\tau-\delta}f_{\alpha}^2(u)\,du\right).\; 
		\end{align*}
	}
	where \(I_2:=-\int_{t_0}^{t_0+\delta} f_{\alpha}^2(t_0+\tau-s)\, g_{\alpha}(s)\,ds\geq0\) and \( I_1:=\int_{0}^{t_0} (f_{\alpha}^2(t_0+\tau-s)-f_{\alpha}^2(t_0-s))\, g_{\alpha}(s)\,ds\geq 0\)
	
	However, as \(I_1\) is nonnegative and close to zero, for an adequate choice of \(\delta \in (0,\frac\tau2)\), the upper bound above is strictly negative.
	On the other hand, 
	$(f_{\alpha}^2* g_{\alpha})(t_0+\tau)-(f_{\alpha}^2* g_{\alpha})(t_0) = c\lambda^2 (R^2_{\alpha}(t_0) - R^2_{\alpha}(t_0+\tau))  > 0$, 
	which yields a contradiction.
	Hence, for every large enough $n \ge 0$, there exists 
	$t_n^+ \in (t_0, t_0 + \frac{1}{n}]$ such that $g_{\alpha}(t_n^+) > 0$. 
	On the other hand, by the very definition of $t_0$, 
	there exists a sequence $t_n^- > t_0$, $n \ge 1$, such that $g_{\alpha}(t_n^-) < 0$. 
	One then builds by induction a sequence $(\tau_n)_{n \ge 1}$ such that 
	$g_{\alpha}(\tau_{2n+1}) < 0$ and $g_{\alpha}(\tau_{2n}) > 0$, with $\tau_n \to t_0$ as $n \to +\infty$, $\tau_n > t_0$. 
	In turn this implies, by the intermediate value theorem, 
	the existence of a sequence $(\tilde{\tau}_n)_{n \ge 1}$ such that 
	$\tilde{g}_{\alpha}(\lambda\tilde{\tau}^\alpha_n) = g_\alpha(\tilde{\tau}_n) = 0$, 
	$\lambda\tilde{\tau}_n^\alpha > \lambda t_0^\alpha$ and 
	$\lambda\tilde{\tau}_n^\alpha \to \lambda t_0^\alpha$ by the continuity of $g_{\alpha}$. 
	As $\tilde{g}_{\alpha}$ is analytic, it implies that $\tilde{g}_{\alpha}$ is everywhere zero. 
	Hence a contradiction since $\tilde{g}_{\alpha}(0) > 0$.

	From the above steps, we have \( \forall t \geq 0 \quad  g_{\alpha,\rho,\lambda}(t) \geq 0  \) on an interval \(  I \subseteq (0, +\infty) \) so that the function \( \sqrt{g_{\alpha,\rho,\lambda}} \)  is well-defined on \( I \).     This completes the proof and we are done.	 \hfill $\square$
	
	\medskip	
	\noindent {\bf Proof of Theorem~\ref{thm:ExistenceVolterraL2}}
	We follow~\cite[Theorem 2.7]{GnabeyeuPages2026} to prove the existence and uniqueness of a strong solution to~\eqref{eq:Volterrameanrevert},
	using a contraction mapping principle.
	Let $p\geq2$. For $T\ge0$, consider the family of processes on $[0,T]$ 
	\vspace{-0.2cm}
	\[
	\Hcal_{p,T}
	:=
	\left\{
	X=(X_s)_{s\in[0,T]}, \text{ }\mathbb{R}\text{-valued, }\mathbb{F}\text{-adapted processes such that } \|X\|_{p,T}<\infty \text{ and is $L^p$-continuous\footnotemark}
	\right\}
	\]
	\footnotetext{In the sense that
		$\lim_{s\to t}\E[|X_s-X_t|^p]=0$ for every $t\in[0,T]$.}
	where $\|X\|_{p,T} = \big(\sup_{t \leq T} \E[|X_t|^p]\big)^{\frac1p}$.
	We still denote by $\Hcal_{p,T}$ the space of all such $X$, modulo the equivalence relation obtained by identifying processes that are versions of each other. One readily checks that $(\Hcal_{p,T},\|\fdot\|_{p,T})$ is a Banach space. Thanks to \cite[Proposition 3.21]{PeszatZabczyk2007}, every element $X\in\Hcal_{p,T}$ admits a predictable modification, 
	again denoted by $X$. Below we always work with such modifications. 
	We first prove the existence of a unique solution to~\eqref{eq:Volterra} in $\mathcal H_{p,T}$. To this end we consider the following family of norms on $\mathcal H_{p,T}$:
	\vspace{-0.2cm}
	\begin{equation}
	\|X\|_{p,c,T} :=  \big(\sup_{t \leq T} \E[|e^{-c t} X_t|^p]\big)^{\frac1p}, \quad c >0.
	\vspace{-0.2cm}
	\end{equation}
	It is readily seen that the norms $\|\cdot\|_{p,c,T}$ and $\|\cdot\|_{p,T}$ are equivalent, since $\forall\,X\in \Hcal_{p,T}$, $e^{-cT}\|X\|_{p,T}\le \|X\|_{p,c,T}\le \|X\|_{p,T}$.
	The core of the argument is to show that
	the map \(\mathcal{T}_c : \mathcal{H}_{p,T} \rightarrow \mathcal{H}_{p,T} , \, \quad \text{defined by}\)
	\vspace{-0.2cm}
	\begin{equation}
	\forall X \in \mathcal{H}_{p,T} 
	\quad \mathcal{T}_c(X) =
	\Big(x_0(t) + \int_{0}^{t } K(t-s) b(s, X_s) \, ds + \int_{0}^{t} K(t-s) \sigma(s, X_s) \, dW_s
	\Big)_{t \in [0,T]}.
	\vspace{-0.2cm}
	\end{equation}
	i.e. for $X \in \mathcal{H}_{p,T}$, \(	\mathcal{T}_c(X) := (\mathcal{T}_c(X)_t)_{0 \leq t \leq T}\) with  
	\vspace{-0.2cm}
	\begin{equation}
	(\mathcal{T}_c(X))_t := x_0(t) + \int_{0}^{t } K(t-s) b(s, X_s) \, ds + \int_{0}^{t} K(t-s) \sigma(s, X_s) \, dW_s.
	\vspace{-0.2cm}
	\end{equation}
	is Lipschitz continuous with Lipschitz coefficient strictly less than one when $c$ is sufficiently large.
	
	\noindent {\sc Step~1} ({\em Well-definedness of $\mathcal{T}_c$ }).
	We need to prove that $\|\mathcal{T}_c(X)\|_{p,c,T} < \infty$  so that $\mathcal{T}_c(X) \in \mathcal{H}_{p,T}$. Since the norms $\|\cdot\|_{p,T}$ and $\|\cdot\|_{p,c, T}$ are equivalent, it is enough to show that $\|\mathcal{T}_c(X)\|_{p,T} < \infty$ . By the triangle inequality,
	\vspace{-0.2cm}
	\begin{align}\label{eq:norm_X}
		\| \mathcal{T}_c(X) \|_{p,T} 
		&\leq \| x_0 \|_{p,T} + \Big\| \int_{0}^{\cdot } K(\cdot-s) b(s, X_s) \, ds \Big\|_{p,T} + \Big\| \int_{0}^{\cdot} K(\cdot-s) \sigma(s, X_s) \, dW_s \Big\|_{p,T}\nonumber\\
		&\leq \| x_0 \|_{p,T} + \big(T^\frac12+ C_p^{\text{BDG}}\big)\Big( \int_0^{T} \ell_{b,\sigma}^{\frac{2\beta}{\beta-1}}(s) ds\Big)^{\frac{\beta-1}{2\beta}}\Big( \int_0^{T} K(s)^{2\beta} ds\Big)^{\frac{1}{2\beta}}\Big(1+ \| X \|_{p,T}\Big)
	\end{align}
	where we argue exactly as in the proof of Lemma~\ref{L:Holder_bound}, using Jensen's inequality\footnote{For two measurable real-valued functions \( f \), \( g \), and \( p \geq 1 \), we have
		\begin{equation}
			\Big| \int f(s)g(s) \, ds \Big|^p 
			\leq \Big| \int |f(s)|^{1 - \frac{1}{p}} \cdot |f(s)|^{\frac{1}{p}} g(s) \, ds \Big|^p 
			\leq \Big( \int |f(s)| \, ds \Big)^{p - 1} \cdot \int |f(s)||g(s)|^p \, ds.\label{eq:JensenIneq}
		\end{equation}
		\label{fn:first}}, the Burkholder--Davis--Gundy inequality then H\"older's inequality, and the linear growth properties of $b$ and $\sigma$ in assumption~\ref{assum:CoefsVolterra}~$(ii)$.
	As \(K\in{\cal L}^{2\beta}_{loc}(\R_+)\), \(\ell_{b,\sigma}\in{\cal L}^{{\frac{2\beta}{\beta-1}}}_{loc}(\R_+)\) for some \(\beta>1\) and \(\| x_0 \|_{p,T}<+\infty\) owing to condition~\eqref{eq:Hol_Intergral_Cond} and assumption~\ref{assum:CoefsVolterra}~$(iii)$, Equation~\eqref{eq:norm_X}
	implies that $\|\Tcal_c(X)\|_{p,T}<\infty$ whenever \(\| X \|_{p,T}<+\infty\).
	The fact that $\Tcal_c(X)$ is continuous in $L^p$ follows from Lemma~\ref{L:Holder_bound}. 
	Thus $\mathcal{T}_c(X)$ lies in $\mathcal H_{p,T}$. 
	Since the norms $\|\cdot\|_{p,T}$ and $\|\cdot\|_{p,c, T}$ are equivalent, we have $\|\mathcal{T}_c(X)\|_{p,c, T} < \infty$ and thus $\mathcal{T}_c(X) \in \mathcal{H}_{p,T}$ whenever \( X \in  \mathcal{H}_{p,T} \).
	
	\smallskip
	\noindent {\sc Step~2} ({\em Lipschitzianity of $\mathcal{T}_c$}).
	For $0\leq t \leq T$ and $c > 0$, 
	We want to show that \( \mathcal{T}_c \) is Lipschitz with respect to the norm \(\|X\|_{p,c,T}\).
	Then, for \((X,Y)\in \mathcal{H}_{p,T}\times \mathcal{H}_{p,T}\), let us estimate \( \|\mathcal{T}_c(X) - \mathcal{T}_c(Y)\|_{p,c,T} \). One has
	\begin{align*}
		|e^{-c t}((\mathcal{T}_c(X))_t - (\mathcal{T}_c(Y))_t)|^p  &\leq  2^{p-1} \left|\int_0^t e^{-c (t-s)}K(t-s)e^{-c s}(b(s,X_s)-b(s,Y_s))ds\right|^p\\
		&\quad  + 2^{p-1} \left|\int_0^t e^{-c (t-s)} K(t-s)e^{-c s}((\sigma(s, X_s)-\sigma(s, Y_s))dW_s\right|^p.
	\end{align*} 
	Let us denote by {\bf(A)} and {\bf(B)} the two terms of the sum on the right hand side of the above equality.
	For the first and deterministic term, using the triangle and 
	Jensen's inequality~\eqref{eq:JensenIneq} in~\ref{fn:first}, then the Lipschitz assumption on \(b\)  
	and taking expectations yields
	\begin{equation} \label{bound_A}
		\E[\,{\bf A}\,] \le T^{\frac{p}{2}}\Big( \int_0^{T} h_{b,\sigma}^{\frac{2\beta}{\beta-1}}(s) ds\Big)^{\frac{p(\beta-1)}{2\beta}}\Big( \int_0^{t}  e^{-2c\beta (t-s)} K(t-s)^{2\beta} ds\Big)^{\frac{p}{2\beta}} \sup_{0\le s\le T} \E[\big(e^{-c s}|X_s-Y_s|\big)^p ].
	\end{equation}
	Similarly, for the martingale terms, using the Burkholder-Davis-Gundy (BDG) inequality applied to the continuous local martingale $\big\{\int_0^r K(t-s)\sigma(s,X_s)dW_s\colon r\in[0,t]\big\}$ and
	standard arguments taking advantage of the Lipschitz property of \(\sigma\), yield
	\begin{align} \label{bound_B}
		&\E\Big[ \,{\bf B}\, \Big]
		\le (C_p^{\text{BDG}})^p\,\E\Big[ \Big( \int_0^t e^{-2c (t-s)} K(t-s)^2e^{-2c s}((\sigma(s, X_s)-\sigma(s, Y_s))^2ds\Big)^{\frac{p}{2}} \Big] \nonumber\\
		&\;\;\le (C_p^{\text{BDG}})^p\,\Big( \int_0^{T} h_{b,\sigma}^{\frac{2\beta}{\beta-1}}(s) ds\Big)^{\frac{p(\beta-1)}{2\beta}}\Big( \int_0^{t}  e^{-2c\beta (t-s)} K(t-s)^{2\beta} ds\Big)^{\frac{p}{2\beta}} \sup_{0\le s\le T} \E[  \big(e^{-c s}|X_s-Y_s|\big)^p ].
	\end{align}
	Combining \eqref{bound_A}--\eqref{bound_B} 
	and taking the sup norm  leads ( we set for short $K_{p,T}:=T^{\frac{1}{2}}+ C_p^{\text{BDG}}$) to
	\begin{equation}\label{eq:lipschtbound} 
		\|\mathcal{T}_c(X)-\mathcal{T}_c(Y)\|_{p,c,T}
		\leq K_{p,T} \Big( \int_0^{T} h_{b,\sigma}^{\frac{2\beta}{\beta-1}}(s) ds\Big)^{\frac{(\beta-1)}{2\beta}} \,\Big( \int_0^Te^{-2c \beta s} K(s)^{2\beta} ds\Big)^{\frac{1}{2\beta}}
		\|X-Y\|_{p,c,T}.
	\end{equation} 
	\noindent {\sc Step~3} ({\em Contraction of $\mathcal{T}_c$}).
	Now, it is enough to find $c>0$ such that the $\mathcal{T}_c$ defines a contraction on $(\mathcal H_{p,T},\|\cdot\|_{p,c, T})$. That is, we look for $c>0$ and $\rho<1$ such that
	\begin{equation}\label{eq:contract temp}
		\| \mathcal{T}_c(X) - \mathcal{T}_c(Y) \|_{p,c, T} \leq \rho \| X -  Y \|_{p,c, T}, \quad  X,Y \in \mathcal H_{p,T}.	
	\end{equation} 
	Note that for any \(\beta\in (1,+\infty)\), \( \int_0^T e^{-2c \beta s}|K(s)|^{2\beta}\,ds
	\xrightarrow{}0\) as \(c\to\infty\)
	by the dominated convergence theorem,
	so that one can choose $c>0$ in~\eqref{eq:lipschtbound} sufficiently large such that \eqref{eq:contract temp} holds.
	Consequently, $\mathcal{T}_c$ is a contraction on $(\mathcal H_{p,T},\|\cdot\|_{p,c,T})$.
	The Banach fixed-point theorem therefore yields a unique fixed point $X\in\mathcal H_{p,T}$ satisfying \(\mathcal{T}_c(X)=X.\)
	Moreover, the following estimate holds
	true:
	\begin{equation}\label{eq:moment bound}
		\sup_{t \leq T} \E[|X_t|^p] \leq C\,\big(1+ \sup_{t \leq T}\E[|x_0(t)|^p] ) \leq C\,\big(1+ \E[\sup_{t \leq T}|x_0(t)|^p] ) < \infty
	\end{equation}
	In fact, the estimate~\eqref{eq:moment bound} is obtained, by taking \( X = 0 \) in~\eqref{eq:norm_X} and \( Y = 0 \), in inequality~\eqref{eq:lipschtbound} and owing in this latter to the fact that \(\mathcal{T}_c (X) = X\) since \(X\) is a fixed point of the application \(\mathcal{T}_c\). The result being true for every \(c\geq 0.\)
	
	\medskip
	Let $X$ be the unique fixed point in $\mathcal H_{p,T}$ of the map $\mathcal{T}_c$. By Lemma~\ref{L:Holder_bound}, $X$ admits a continuous version. This version is a strong solution of \eqref{eq:Volterra} on $[0,T]$. Conversely, by virtue of Lemmas~\ref{L:Holder_bound} and~\ref{eq:moment bound}, every continuous solution of~\eqref{eq:Volterra} on $[0,T]$ belongs to $\mathcal H_{p,T}$ and satisfies $ \mathcal{T}_c(X)=X$. Hence every continuous solution is a fixed point of the map $\mathcal{T}_c$, and uniqueness of the fixed point implies uniqueness of the solution.
	The inequalities~\eqref{eq:Lpincrements}, ~\eqref{eq:Holderpaths} and~\eqref{eq:supLpbound}
	follow from Lemma~\ref{L:Holder_bound} (see Equation~\ref{eq:FinTightness}) together with the estimate~\eqref{eq:moment bound} and Jensen's inequality.
	This completes the proof.	 \hfill $\square$
	
	\medskip
	\noindent{\bf Remark:} If the diffusion coefficient $\sigma$ satisfy a Lipschitz-linear growth assumption uniformly in time with some real constants $ C_{\sigma,T}$ and $ C^\prime_{\sigma,T}$, then Equation~\eqref{eq:norm_X} and~\eqref{eq:lipschtbound} are rather:
	\begin{align}
		&\| \mathcal{T}_c(X) \|_{p,T} \leq \| x_0 \|_{p,T} + \Big(\Big( \int_0^{T} \ell_{b}^2(s) ds\Big)^{\frac{1}{2}}+C_{\sigma,T}^\prime C_p^{\text{BDG}}\Big)\Big( \int_0^{T} K(s)^2 ds\Big)^{\frac{1}{2}}\Big(1+ \| X \|_{p,T}\Big)\\
		&\|\mathcal{T}_c(X)-\mathcal{T}_c(Y)\|_{p,c,T}
		\leq \Big(\Big( \int_0^{T} h_{b}^2(s) ds\Big)^{\frac{1}{2}}+C_{\sigma,T} C_p^{\text{BDG}}\Big) \,\Big( \int_0^Te^{-2c s} K(s)^2 ds\Big)^{\frac{1}{2}}
		\|X-Y\|_{p,c,T}.
	\end{align}
Therefore, the well-posedness result remains valid under the sole assumption that \(K,\ell_b,h_b \in {\cal L}_{loc}^2(\R_+)\).
	\begin{Lemma}\label{L:Holder_bound}
		Let  Assumptions~\eqref{assum:convolKernel} and~\eqref{assum:CoefsVolterra} be in force and consider the Volterra process~\eqref{eq:Volterra}. Let $T\ge0$ and $p>\max\{2,\frac1\delta\vee\frac{1}{\widehat\theta}\vee\frac1\theta\}$ be such that $\sup_{t\le T} \E[ |X_t|^p ]$ is finite. Then $X$ admits a version which is H\"older continuous on $[0,T]$ of any order $a<\delta\wedge\widehat\theta\wedge\theta-\frac1p$. Denoting this version again by $X$, one has
		\begin{equation} \label{eq:L:Holder_bound}
			\E\left[ \left( \sup_{s\neq t\in[0,T]} \frac{|X_t-X_s|}{|t-s|^a} \right)^p \right] \le C_{a,p, K, b,\sigma,T}\,\Big[\big(1+ \E[\sup_{0\leq u\le T} |x_0(u)|^{p} ]\big)\vee \big(1+\sup_{0\leq u\le T} \E[  |X_u|^{p} ]\big)\Big].
		\end{equation}
		for all $a\in[0,\delta\wedge\widehat\theta\wedge\theta-\frac1p)$ and $C_{a,p, K, b,\sigma,T}$, a positive constant that only depends on $a$, $p$, $K$, $b$,$\sigma$ and $T$.
	\end{Lemma}
	\noindent{\bf Proof:}
	For any $p\ge2$ and any $s<t\le T<\infty$ we have
	\begin{align*}
		|X_t - X_s|^p &\le  5^{p-1} \Big(|x_0(t) - x_0(s)|^p + \Big| \int_s^t K(t-u)b(u,X_u) du \Big|^p  + \Big| \int_0^s \left( K(t-u) - K(s-u) \right) b(u,X_u) du \Big|^p\Big) \\
		&\qquad5^{p-1} \Big( \Big| \int_s^t K(t-u)\sigma(u,X_u)dW_u \Big|^p +  \Big| \int_0^s \left( K(t-u) - K(s-u) \right)\sigma(u,X_u) dW_u \Big|^p\Big)  \\
		&= 5^{p-1}\left( {\bf C} + {\bf I} + {\bf II}+  {\bf III} + {\bf IV}  \right).
	\end{align*}
	By Jensen's inequality 
	in~\ref{fn:first}, then taking expectations together with a change of variables, we obtain 
	\begin{equation}\label{bound_I_new}
		\E[\,{\bf I}\,]\le (t-s)^\frac{p}{2} \Big( \int_s^{t} \ell_{b,\sigma}^{\frac{2\beta}{\beta-1}}(u) du\Big)^{\frac{p(\beta-1)}{2\beta}} \, \Big( \int_s^{t} K(t-u)^{2\beta}du\Big)^{\frac{p}{{2\beta}}} \sup_{0\leq u\le T} \E[  \big(1+|X_u|\big)^p ].
	\end{equation}
	In a similar manner,
	\begin{equation}\label{bound_II_new}
		\E[\,{\bf II}\,] \le T^\frac{p}{2} \Big( \int_0^{s} \ell_{b,\sigma}^{\frac{2\beta}{\beta-1}}(u) du\Big)^{\frac{p(\beta-1)}{2\beta}} \,\Big( \int_0^s (K(u+t-s)-K(u))^{2\beta}  du\Big)^{\frac{p}{{2\beta}}} \sup_{0\leq u\le T} \E[  \big(1+|X_u|\big)^p ].
	\end{equation}
	Analogous calculations relying on the BDG inequalities together with Jensen's inequality~\eqref{eq:JensenIneq}
	yield
	\begin{align} 
			\E\left[ \,{\bf III}\, \right]
			&\le (C_p^{\text{BDG}})^p\,\E\Big[ \Big( \int_s^t K(t-u)^2 \,\ell_{b,\sigma}^2(u)\, \big(1+|X_u|\big)^2\, du \Big)^{\frac{p}{2}} \Big] \nonumber\\
			&\le (C_p^{\text{BDG}})^p\,\Big( \int_s^{t} \ell_{b,\sigma}^{\frac{2\beta}{\beta-1}}(u) du\Big)^{\frac{p(\beta-1)}{2\beta}} \,\Big( \int_s^{t} K(t-u)^{2\beta} du\Big)^{\frac{p}{{2\beta}}} \sup_{0\leq u\le T} \E[  \big(1+|X_u|\big)^p ].\label{bound_III_new}
	\end{align}
	and
	\begin{equation} \label{bound_IV_new}
		\E\left[ \,{\bf IV}\, \right] \le ( C_p^{\text{BDG}})^p \big( \int_0^{s} \ell_{b,\sigma}^{\frac{2\beta}{\beta-1}}(u) du\big)^{\frac{p(\beta-1)}{2\beta}} \Big( \int_0^s (K(u+t-s)-K(u))^{2\beta}  du\Big)^{\frac{p}{{2\beta}}} \sup_{0\leq u\le T} \E[ \big(1+|X_u|\big)^p ].
	\end{equation}
	Combining \eqref{bound_I_new}--\eqref{bound_IV_new} with conditions~\eqref{eq:contKtilde} and \eqref{eq:Kcont} and Assumption~\ref{assum:CoefsVolterra}~$(iii)$ for \({\bf C}\) leads to
	\begin{equation}\label{eq:FinTightness}
		\E\left[|X_t - X_s|^p\right] \le C'\Big[\big(1+ \E[\sup_{0\leq u\le T} |x_0(u)|^{p} ]\big)\vee \big(1+\sup_{0\leq u\le T} \E[  |X_u|^{p} ]\big)\Big]\,(t-s)^{p(\delta\wedge\widehat\theta\wedge\theta)},
	\end{equation}
	with 
	$C':=C_{T,p}+\big(\widehat \kappa_{T,\beta}+ \kappa_{T,\beta}\big)\Big(T^\frac{p}{2}+(C_p^{\text{BDG}})^p\Big) \Big( \int_0^{s} \ell_{b,\sigma}^{\frac{2\beta}{\beta-1}}(u) du\Big)^{\frac{p(\beta-1)}{2\beta}}$, a constant that only depends on $p$, $K$, $b$, $\sigma$, $T$, but not on $s$ or $t$. Existence of a continuous version as well as bound \eqref{eq:L:Holder_bound} now follow from the Kolmogorov continuity theorem; see \cite[Theorem~I.2.1]{RevuzYor}. This completes the proof.	 \hfill $\square$
	
	\medskip
	\noindent{\bf Remark:} If $x_0:=X_0\phi$ as in~\eqref{eq:Volterrameanrevert}, where the function $\phi$ satisfies the H\"older continuity condition in~\ref{assum:CoefsVolterra}-$(iii)$, using~\eqref{eq:moment bound} and mimicking the end of the proof of~\cite[Theorem 1.1]{JouPag22} (see see also~\cite[Theorem 2.10]{GnabeyeuPages2026}), the bounds~\eqref{eq:FinTightness} and~\eqref{eq:L:Holder_bound} can be extended to every \(p > 0\) such that \(\lVert X_0\rVert_p < +\infty\).
	
\paragraph{An alternative proof under stronger regularity assumptions:}	 We conclude this section with an alternative proof of Theorem~\ref{thm:ExistenceVolterraL2}. Assuming additional regularity of the kernel, the proof proceeds by estimating the supremum norm directly, thereby avoiding the Hölder continuity estimates used above.

Here, the Kernel $K$ satisfies \(K(0)=0 \) and its derivative $K^\prime$ together with $h_{b,\sigma}$, $\ell_{b,\sigma}$
are locally integrable with an exponent strictly larger than 2, namely,
\vspace{-.2cm}
\begin{equation}\label{eq:Hol_Intergral_Cond_}
	\int_0^T\big(\ell_{b,\sigma}(u)^{{\frac{2\beta}{\beta-1}}} + h_{b,\sigma}(u)^{{\frac{2\beta}{\beta-1}}}
	+K^\prime(u)^{2\beta}\big)du<+\infty, \; T>0, \; \text{for some} \; \beta>1.
	\vspace{-.3cm}
\end{equation}
\noindent{\bf Remarks :} 1. Note that, if $h_{b,\sigma}$ and $\ell_{b,\sigma}$ 	
are constant functions, we only require \(K^\prime \in {\cal L}^2_{loc}(\R_+) \), which is sufficient to prove the existence result.\\
\noindent 2. One readily checks that the derivative $K^\prime$ of the kernel \(K\) in Examples~\ref{Ex:SolventGammaKernel}~$2.$ and~$3.$ satisfy~\ref{eq:Hol_Intergral_Cond_}
for every \(\alpha>\frac32\) with any \(\beta\in[1,\frac{1}{2(\alpha-2)^-})\), still with the convention $1/0 = +\infty$ and \(x^-:=\max(-x, 0)\).

\medskip	
\noindent {\bf Alternative proof of Theorem~\ref{thm:ExistenceVolterraL2}}
		We follow~\cite[Remark on Theorem~2.7]{GnabeyeuPages2026} (see also \cite{Protter1985}) to prove the existence and uniqueness of a strong solution to~\eqref{eq:Volterrameanrevert},
		using a Picard iteration principle.
		Let $p\geq2$. For $T\ge0$, consider the family of processes on $[0,T]$ 
		\[
		\Hcal_{p,T}
		:=
		\left\{
		X=(X_s)_{s\in[0,T]},\mathbb{R}\text{-valued, }\mathbb{F}\text{-adapted, $L^p$-continuous process such that } \mathbb{E}\!\left[\|X\|_T^p\right]<\infty
		\right\},
		\]
		where \(\|x\|_t:=\sup_{u\in[0,t]}|x_u|,
		\; t\in[0,T],\; x\in \Hcal_{p,T}\). 
		We set \(\|X\|_{p,T}:=\big\|\|X\|_T\big\|_{p}=\big(\mathbb{E}\!\left[\|X\|_T^p\right]\big)^\frac1p\).
		We still denote by $\Hcal_{p,T}$ the space of all such $X$, modulo the equivalence relation obtained by identifying processes that are versions of each other. One readily verifies that $(\Hcal_{p,T},\|\fdot\|_{p,T})$ is a Banach space. Thanks to \cite[Proposition 3.21]{PeszatZabczyk2007}, every element $X\in\Hcal_{p,T}$ admits a predictable representative, again denoted by $X$. Below we always work with such representatives.
		We first prove the existence of a unique solution to~\eqref{eq:Volterra} in $\mathcal H_{p,T}$.
		For all \(X\in \mathcal{H}_{p,T}\), we define 
		the process \(	\mathcal{T}(X) := (\mathcal{T}(X)_t)_{0 \leq t \leq T}\)
		by
		\[
		(\mathcal{T}(X))_t := x_0(t) + \int_{0}^{t } K(t-s) b(s, X_s) \, ds + \int_{0}^{t} K(t-s) \sigma(s, X_s) \, dW_s = x_0(t) + \int_{0}^{t } K(t-s) dZ_s .
		\]
		Under our assumptions, for any \(X\in \mathcal{H}_{p,T}\),
		\(\int_0^t \left( \int_0^s \left| \partial_u K(u-s) \right|^2 d\langle Z\rangle_u \right)^{\frac12} ds < \infty \quad \P-\text{a.s.}.\) Consequently,
		we have that \(Y_t := \int_{0}^{t } K(t-s) dZ_s \) is an \((\mathcal{F}_t)_{t \in [0,T]}\)-semimartingale. By applying the fundamental theorem of calculus and stochastic Fubini's theorem \cite[ Theorem 2.6]{Walsh1986}, we can represent \(Y_t\) as:
		\[
		Y_t = \int_0^t K(0) \, dZ_s + \int_0^t \Big( \int_0^s \partial_s K(s-u) \, dZ_u \Big) ds, \quad \text{where} \quad \partial_s K(s-u) := \Big. \frac{\partial}{\partial t} K(t-u) \Big|_{t = s}=K^\prime(s-u).
		\]
		Furthermore, for any \( T >0\), combining generalized Minkowski inequality, H\"older's inequality and the $L^p-BDG$ (Burkholder-Davis-Gundy) inequality yields
		{\small
			\begin{equation}\label{eq:bounds_Sup2}
				\frac{\big\|\|Y\|_T\big\|_{p}}{T^{1-\frac1p}} 
				\leq  \Big( \int_0^T \mathbb{E} \Big[ \Big( \int_0^s | K^\prime(s-u) b(u, X_u)| du \Big)^{p} \Big] ds\Big)^\frac1p
				+ C_p^{BDG}\Big( \int_0^T \mathbb{E} \Big[ \Big( \int_0^s | K^\prime(s-u) \sigma(u, X_u)|^2 du \Big)^{\frac{p}{2}} \Big] ds \Big)^\frac1p
			\end{equation}
		}
		Consequently, we can directly prove the existence and uniqueness of a solution in this setting, using the supremum norm, via the Picard iteration method.
		
		\noindent {\sc Step~1} ({\em Well-definedness of $\mathcal{T}$ }).
		We need to prove that for any $X \in \mathcal{H}_{p,T}$, $\|\mathcal{T}(X)\|_{p,T} < \infty$  so that $\mathcal{T}(X) \in \mathcal{H}_{p,T}$. By the triangle inequality,
		\begin{align}\label{eq:norm_X2}
			&\,\| \mathcal{T}(X) \|_{p,T} 
			\leq \| x_0 \|_{p,T} + \Big\| \int_{0}^{\cdot } K(\cdot-s) b(s, X_s) \, ds \Big\|_{p,T} + \Big\| \int_{0}^{\cdot} K(\cdot-s) \sigma(s, X_s) \, dW_s \Big\|_{p,T}\nonumber\\
			&\leq \| x_0 \|_{p,T} + (2T)^{1-\frac1p}\big(\sqrt{T}+C_p^{\text{BDG}}\big)\sup_{0\le s\le T}\big( \int_0^{s} K^\prime(s-u)^2 \ell_{b,\sigma}^2(u) du\big)^{\frac{1}{2}}\big(\int_0^T\big(1+ \| X \|_{p,s}^p\big)ds\big)^\frac1p
		\end{align}
		where we used Equation~\eqref{eq:bounds_Sup2} togehter with Jensen's inequality~\eqref{eq:JensenIneq}
		in~\ref{fn:first}, and the linear growth properties of $b$ and $\sigma$ in assumption~\ref{assum:CoefsVolterra}~$(ii)$.
		As \(\int_0^T\big(\ell_{b,\sigma}(u)^{{\frac{2\beta}{\beta-1}}}
		+K^\prime(u)^{2\beta}\big)du<+\infty \; \text{for some} \; \beta>1\) owing to  Equation~\eqref{eq:Hol_Intergral_Cond} of  Assumption~\ref{assum:convolKernel}
	and	\(\| x_0 \|_{p,T}<+\infty\) by assumption~\ref{assum:CoefsVolterra}~$(ii)$ and~$(iii)$, Equation~\eqref{eq:norm_X2}
		implies that $\|\Tcal(X)\|_{p,T}<\infty$ whenever \(\| X \|_{p,T}<+\infty\).
		 The fact that $\Tcal(X)$ is continuous in $L^p$ follows from Lemma~\ref{L:Holder_bound}. 
		Thus $\mathcal{T}(X)$ lies in $\mathcal H_{p,T}$. 
	
		\smallskip
		\noindent {\sc Step~2} ({\em An integral recursive stability estimate for $\mathcal{T}$}).
		For $0\leq t \leq T$, we want to establish a key integral estimate for
		$\mathcal{T}$ with respect to the norm \(\|X\|_{p,t}\) needed for the Picard
		iteration.  For \((X,Y)\in \mathcal{H}_{p,T}\times \mathcal{H}_{p,T}\), let us estimate \( \|\mathcal{T}(X) - \mathcal{T}(Y)\|_{p,t} \). We have
		using Equation~\eqref{eq:bounds_Sup2}, the Lipschitz assumption on \(b\) and \(\sigma\) 	together with 
		Jensen's inequality~\eqref{eq:JensenIneq} in~\ref{fn:first}: 
			\begin{align}\label{eq:Lipschitz_norm}
			\|\mathcal{T}(X) &\,- \mathcal{T}(Y)\|_{p,t} 
			\leq  \Big\| \int_{0}^{\cdot } K(\cdot-s) \big(b(s, X_s)-b(s, Y_s)\big) \, ds \Big\|_{p,t} + \Big\| \int_{0}^{\cdot} K(\cdot-s) \big(\sigma(s, X_s)-\sigma(s, Y_s)\big) \, dW_s \Big\|_{p,t}\nonumber\\
			&\leq (2t)^{1-\frac1p}\big(\sqrt{t}+C_p^{\text{BDG}}\big)\sup_{0\le s\le t}\Big( \int_0^{s} K^\prime(s-u)^2 h_{b,\sigma}^2(u) du\Big)^{\frac{1}{2}}\Big(\int_0^t \| X-Y \|_{p,s}^pds\Big)^\frac1p
		\end{align}
		Besides, by Equation~\eqref{eq:Hol_Intergral_Cond} of Assumption~\ref{assum:convolKernel}, \(	\int_0^T\big(h_{b,\sigma}(u)^{{\frac{2\beta}{\beta-1}}}
		+K^\prime(u)^{2\beta}\big)du<+\infty \; \text{for some} \; \beta>1\). Then,
		combining \eqref{eq:Lipschitz_norm} with assumption~\ref{assum:convolKernel} leads to the existence of some constant $C_{p,T} > 0$ that depends only on $p$, $T$, $\beta$, $b$, $\sigma$ and the kernel $K$ such that
		\begin{equation}\label{eq:lipschtbound2} 
			\|\mathcal{T}(X)-\mathcal{T}(Y)\|_{p,t}^{\,p}
			\leq C_{p,T}\,\int_0^t \| X-Y \|_{p,s}^pds,
		\end{equation} 
		\noindent {\sc Step~3} ({\em Picard Iteration Algorithm}). 
		Moreover, for every integer $n\ge1$,
		\begin{align*}
			\|\mathcal{T}^n(X)-\mathcal{T}^n(Y)\|_{p,t}^{\,p}
			&\le C_{p,T}\int_0^t
			\|\mathcal{T}^{\,n-1}(X)-\mathcal{T}^{\,n-1}(Y)\|_{p,s}^{\,p}\,ds\\
			&\le (C_{p,T})^2
			\int_0^t\int_0^{s}
			\|\mathcal{T}^{\,n-2}(X)-\mathcal{T}^{\,n-2}(Y)\|_{p,r}^{\,p}\,dr\,ds\\
			&\le (C_{p,T})^n
			\int_{\{0\le s_n\le\cdots\le s_1\le t\}}
			\|X-Y\|_{p,s_n}^{\,p}\,
			ds_n\cdots ds_1\le (C_{p,T})^n
			\|X-Y\|_{p,t}^{\,p}\,
			\frac{t^n}{n!}.
		\end{align*}
		Now for $Y \in \mathcal{H}_{p,T}$, let $X^0:=Y$ and define recursively $X^{n+1}:=\mathcal{T}(X^n),\; n\ge0.$
		Then, it follows that
		\[
		\|X^n-X^{n+1}\|_{p,t}^{\,p}
		\le
		(C_{p,T})^n
		\|Y-\mathcal{T}(Y)\|_{p,t}^{\,p}
		\frac{t^n}{n!},\;\text{and consequently}\; \sum_{n\ge0}
		\|X^{n+1}-X^n\|_{p,T}
		<\infty.
		\]
		Hence $(X^n)_{n\ge0}$ converges uniformly, $\P-a.s.$ in $\mathcal{H}_{p,T}$ to some
		$X\in\mathcal{H}_{p,T}$ with continuous path a.s., which is a fixed point of
		$\mathcal{T}$.
		Finally, it is straightforward to
		see that \(X\) is the unique strong solution of~\eqref{eq:Volterra} with initial profile \(x_0\).
		 Moreover, the following estimate holds
		 true for some constant \(C>0\)  that only depends only on $p$, $T$, $\beta$, $b$, $\sigma$ and the kernel $K$:
		 \begin{equation}\label{eq:moment bound2}
		 	 \sup_{t \leq T}\E[|X_t|^p] \leq  \E[\sup_{t \leq T}|X_t|^p] \leq C\,\big(1+ \E[\sup_{t \leq T}|x_0(t)|^p] ) < \infty
		 \end{equation}
		In fact, the estimate~\eqref{eq:moment bound2} is obtained, by taking \( X = 0 \) in~\eqref{eq:norm_X2} and \( Y = 0 \), in inequality~\eqref{eq:lipschtbound2} 
		and owing in this latter to the fact that \(\mathcal{T}(X) = X\) since \(X\) is a fixed point of the application \(\mathcal{T}\). The result follows by an application of Gr\"onwall's lemma.
		
		\smallskip
		\noindent
		At this stage, the inequalities~\eqref{eq:Lpincrements} and~\eqref{eq:Holderpaths} 
		follow from Lemma~\ref{L:Holder_bound} (see Equation~\ref{eq:FinTightness}) together with the estimate~\eqref{eq:moment bound2} and Jensen's inequality.
     This completes the proof and we are done.	 \hfill $\square$
\end{document}